\documentclass[12pt]{amsart}

\usepackage{amsmath,amssymb}
\usepackage{graphicx}
\usepackage{amsthm}
\usepackage[colorlinks=true]{hyperref}
\numberwithin{equation}{section}
\newcommand{\R}{\mathbb{R}}
\newcommand{\T}{\mathbb{T}}
\newcommand{\N}{\mathbb{N}}
\newcommand{\Z}{\mathbb{Z}}
\newcommand{\C}{\mathbb{C}}

\newcommand{\bS}{\mathbb{S}}

\newcommand{\cF}{\mathcal{F}}
\newcommand{\sym}{{\rm Symp}_c}

\newcommand{\beqnn}{\begin{eqnarray*}}
\newcommand{\eeqnn}{\end{eqnarray*}}
\newcommand{\beqn}{\begin{eqnarray}}
\newcommand{\eeqn}{\end{eqnarray}}
\newcommand{\beq}{\begin{equation}}
\newcommand{\eeq}{\end{equation}}

\theoremstyle{plain}
\theoremstyle{definition}
\newtheorem{thm}{Theorem}[section]
\newtheorem{prop}[thm]{Proposition}
\newtheorem{conj}[thm]{Conjecture}
\newtheorem{lem}[thm]{Lemma}

\newtheorem{rmk}[thm]{Remark}
\newtheorem{defi}{Definition}[section]
\newtheorem{exm}[defi]{Example}
\def\rk{\mathrm{rank}}

 \def\HF{\mathrm{HF}}
 \def\CF{\mathrm{CF}}
\begin{document}

\title{Dynamics of composite symplectic Dehn twists}

\author{Wenmin Gong, Zhijing Wendy Wang, Jinxin Xue}
\address{School of Mathematical Sciences\\ Beijing Normal University\\ Beijing, China}
\email{wmgong@bnu.edu.cn}

\address{Department of Mathematics\\ University of Chicago\\ Chicago, United States of America}
\email{zhijingw@uchicago.edu}

\address{New Cornerstone Sciences Laboratory, Department of Mathematics, \\ Tsinghua University\\ Beijing, China}
\email{jxue@tsinghua.edu.cn}

\begin{abstract}
	This paper appears as the confluence of hyperbolic dynamics, symplectic topology and low dimensional topology, etc. We show that  composite symplectic Dehn twists have certain form of nonuniform hyperbolicity: it has positive topological entropy as well as two families of local stable and unstable Lagrangian manifolds, which are analogous to signatures of pseudo{-}Anosov mapping classes. Moreover, we show that the rank of the Floer cohomology group of these compositions grows exponentially under iterations, and provide a classification of the symplectic mapping class group of the $A^2_m$ configuration, which partially answers a question of Smith concerning the classification of symplectic mapping class group in higher dimensions. Finally, we propose a conjecture on the positive metric entropy of our model and point out its relationship with the standard map.
\end{abstract}

\maketitle
 {\noindent \textbf{Mathematics Subject Classification:} 53D40; 37D25; 37E40.}

 \section{Introduction}

In this paper, we study the dynamical behaviors of elements in the symplectic mapping class group of a symplectic manifold.

When the symplectic manifold is a surface $S$, the classical Nielsen-Thurston theory provides valuable insights into the dynamics of mapping classes: every automorphism $f$ of $S$ is homotopic to a homeomorphism $\phi$ that satisfies one of the following,
periodic ($\phi^k=\mathrm{id}$ for some $k\in \N$), reducible (there exists a closed loop $\gamma\subset S$ preserved by $\phi$),  {or} pseudo-Anosov. The  { pseudo-Anosov case} is characterized by two significant properties, among others:
\begin{enumerate}
	\item[(a)] Existence of two singular transverse foliations invariant under a representative $\phi$;
	\item[(b)] Expansion/contraction along the leaves, measured by positive topological entropy.
\end{enumerate}
Pseudo-Anosov mapping classes are of particular interest due to their close connection to Anosov systems, which are central objects studied in the field of hyperbolic dynamical systems. Furthermore, pseudo-Anosov mapping classes constitute the majority of classes within the mapping class group.

To gain a clearer understanding of these concepts, let us consider the classification on the two dimensional torus. We examine the linear automorphisms in $\mathrm{PSL}_2\Z=\mathrm{SL}_2\Z/\{\pm\mathrm{id}\}$ that act on the torus $\R^2/\Z^2$. An automorphism $f$ in $\mathrm{PSL}_2\Z$ can be classified into one of the following three types:
\begin{enumerate}
	\item  $\mathrm{tr}(f) < 2$: In this case, $f$ is periodic, and there exists some $k \in \N$ such that $f^k=\mathrm{id}$.
	\item  $\mathrm{tr}(f) = 2$: Here, $f$ is reducible, i.e. it fixes a closed curve and exhibits linear growth of geometric intersection numbers, i.e. $I( [\gamma_1],f_*^n[\gamma_2])$ grows at most linearly, for any two simple closed curves $\gamma_1,\gamma_2$ where $I$ is the geometric intersection number.
 \item $\mathrm{tr}(f) > 2$: This type corresponds to Anosov automorphisms. In this case, $f$ expands along one eigen-direction while contracts along the other eigen-direction, which form the stable and unstable foliations of the automorphism. It also exhibits exponential growth of geometric intersection numbers, i.e. $I( [\gamma_1],f_*^n[\gamma_2])$ grows exponentially, for any two simple closed curves $\gamma_1,\gamma_2$ with nontrivial homology classes. 
\end{enumerate}

In this paper, we aim to generalize this phenomenon to higher dimensional symplectic manifolds. In particular, we consider $2n$-dimensional symplectic manifolds $A^n_m$, where $$A^n_m=\{x_1^2+x_2^2+\ldots +x_n^2=x_{n+1}^{m+1}+\frac{1}{2}\}\subset (\C^{n+1},\omega_0),$$ which can be obtained by plumbing $m$ copies of $T^*\mathbb{S}^n$, see~\cite{Ab,AS}. In the manifold $A^n_m$, there exist $m$ embedded Lagrangian spheres $S_i, 1\le i\le m$ that transversely intersect with one another, such that $|S_i\cap S_j|=1$ if $|i-j|=1$ and $|S_i\cap S_j|=0$ if $|i-j|>1$. We shall explain the $A^n_m$ manifold in more details in Example \ref{Am}.

Our first main result gives a dynamical classification of the symplectic mapping class group in terms of Lagrangian Floer cohomology. This provides the first example answering Smith's question in~\cite[Section 3.2.2]{Sm}, which asks for classification of the symplectic mapping class group in higher dimensions, analogous to the classification of the mapping class group of surfaces.

\begin{defi}
A \emph{Lagrangian isotopy} on a symplectic manifold $(M,\omega)$ is a smooth family of  {Lagrangian embeddings} $\Phi:L\times [0,1]\to M$.
\end{defi}

We obtain the following classification of symplectomorphisms on the $A^2_m$ configuration.

\begin{thm}\label{thm:class}
For each symplectic mapping class  $[f] \in \pi_0(\sym (A^2_m, \omega))$, any representative $f$ satisfies one of the following:
    \begin{itemize}
     \item $f$ is reducible:  {there exists a Lagrangian submanifold $L$, which is an embedded sphere or torus, such that some power $f^k,k\neq 0$ preserves the Lagrangian isotopy class of $L$, i.e. $f^k(L)$} is Lagrangian isotopic to $L$;
        \item $f$ is periodic: for any Lagrangian spheres $\alpha,\beta\subseteq A^2_m$, the sequence $${\rm rank}\;\HF^*(\alpha,f^n\beta)$$is periodic in $n$;
        \item $f$ is hyperbolic: for any Lagrangian spheres $\alpha,\beta\subseteq A^2_m$, the sequence $${\rm rank}\;\HF^*(\alpha,f^n\beta)$$ grows exponentially in $n$.
    \end{itemize}

   {In the case that $m=2$, there is a natural homomorphism $\rho: \pi_0(\sym (A^2_m, \omega))\simeq \mathrm{Br}_2\to \mathrm{PSL}_2\Z$, generated by $\rho(a)=\begin{pmatrix}
      1&1\\0&1
  \end{pmatrix}$ and $\rho(b)=\begin{pmatrix}
      1&0\\-1&1
  \end{pmatrix}$ for the two generators $a,b$ of $\mathrm{Br}_2$, corresponding to the symplectic Dehn twists around the two Lagrangian spheres. Any element $[f]\in \pi_0(\sym (A^2_m, \omega))$ is reducible/periodic/hyperbolic if and only if $\rho([f])$ is reducible/periodic/Anosov.}
\end{thm}

  More generally,  it is known by  Evans~\cite[{Theorem~4}]{Eva} and Weiwei Wu \cite[Theorem~1.3]{wu} that the symplectic mapping class group of $A^2_m$ is isomorphic to the braid group $\mathrm{Br}_m$, thus inherits a classification into periodic, reducible and pseudo Anosov mapping classes from the latter. However, such an isomorphism does not automatically give the Lagrangian isotopy or the growth of Floer cohomology groups under iteration of symplectomorphims as stated in the above theorem.

A classical example of a hyperbolic automorphism is the Arnold's cat map:
$f=\begin{pmatrix} 2&1\\1&1\end{pmatrix}$
on the two-torus.  We observe that it can be expressed as a composition of upper and lower triangular matrices:
$$\begin{pmatrix} 1&1\\0&1\end{pmatrix}\quad \text{and} \quad \begin{pmatrix} 1&0\\1&1\end{pmatrix},$$ 
which are instances of Dehn twists on the torus. Dehn twists involve cutting the surface along a closed curve, twisting a tubular neighborhood of one side of the curve by $2\pi$, and then reassembling the curve.

In general, the composition of Dehn twists is a common method for constructing pseudo-Anosov elements on surfaces. This algorithm can be found in the works of Thurston \cite{Th}, Penner \cite{Pe}, Fathi \cite{fathi}, and others.

It was noted by Arnold \cite{Arnold} that ``the natural generalization of $\mathrm{SL}(2)$ is, from many points of view, not $\mathrm{SL}(n)$ but rather the linear symplectomorphisms group $\mathrm{Sp}(n)$ of $2n$-space, and correspondingly the group $\mathrm{Symp}(M)$." Therefore, it is natural to try and extend the picture to $2n$-dimensional symplectic manifolds with $n>1$. Arnold \cite{Arnold} and Seidel \cite{Se} introduced a generalization of Dehn twists as a symplectomorphism on $T^*\mathbb S^n$. If $L$ is a Lagrangian sphere in a symplectic manifold $M$, the Dehn twist $\tau_L$ with respect to $L$ roughly twists a tubular neighborhood of $L$ in a way that maps the points on $L$ to its antipodal points and transitions to identity outside the neighborhood.
For $n=2$, the square of a symplectic Dehn twist is smoothly but not symplectically isotopic to the identity \cite{Se}.
Keating and Randal-Williams further discussed this in different dimensions, see~\cite{KR}.

We anticipate that by composing symplectic Dehn twists, one can generate symplectic mapping classes that exhibit characteristics of pseudo-Anosov mapping classes. We aim to uncover the hyperbolicity of such elements and establish properties similar to pseudo-Anosov elements mentioned above (see Theorems \ref{2twists}, \ref{mtwist}, \ref{foli} below). Additionally, we will investigate the exponential growth of ranks in Floer cohomology groups associated with these elements (see Theorems \ref{thm:symrate} and \ref{thm:symrate2} below). This also leads to the classification of the symplectic mapping class group of the $A_m^2$ manifold in terms of the growth of Floer cohomology as given in Theorem \ref{thm:class}. 

The study of composite symplectic Dehn twists has gained attention in recent years from the perspective of categorical dynamics. Interested readers can refer to \cite{BK,FFHKL,KO} and the references therein for more information. 

\subsection{Positive topological entropy}
We begin by demonstrating that specific natural representatives of compositions of symplectic Dehn twists exhibit positive topological entropy.

\begin{defi}
We say that a symplectic manifold $M$ \emph{contains an $A^n_m$-configuration}, if $\dim M=2n$ and there exists $m$ embedded Lagrangian spheres $S_i, 1\le i\le m$, that transversely intersect with one another, such that $|S_i\cap S_j|=1$ if $|i-j|=1$ and $|S_i\cap S_j|=0$ if $|i-j|>1$.
\end{defi}
We consider the case of $m=2$ as the higher dimensional generalization of the Arnold cat map. 
\begin{thm}\label{2twists}
Let $M$ be a symplectic manifold containing an $A^n_2$-configuration of Lagrangian spheres $S_1$ and $S_2$ that intersect transversely at a single point. Then there exist symplectic Dehn twists $\tau_1$ and $\tau_2$ of $S_1$ and $S_2$, respectively, such that for any $k, \ell\in\mathbb{Z}$ where $k\ell \neq 0,1,2,3,4$, the composition $\tau=\tau_1^k\tau_2^\ell$ of symplectic Dehn twists has positive topological entropy, i.e., $h_{top}(\tau_1^k\tau_2^\ell)>0$.
\end{thm}

For $m>2$, we have the following theorem.
\begin{thm}\label{mtwist}
 Let $M$ be a symplectic manifold containing an $A^n_m$ configuration and let $S_i, 1\leq i\leq m$, be the $m$ embedded Lagrangian spheres. Then there exist Dehn twists $\tau_i$ on $S_i$ for each $1\leq i\leq m$ such that for any $k_i\in \mathbb{Z}$ satisfying $k_ik_{i+1}<0$ for all $i=1,2,\ldots, m-1$, the topological entropy of $\tau=\tau_1^{k_1}\cdots \tau_m^{k_m}$ is positive, i.e., $$h_{top}(\tau_1^{k_1}\cdots \tau_m^{k_m})>0.$$
\end{thm}

The proof of this theorem involves identifying an invariant submanifold on which the restrictions of the symplectic Dehn twists correspond to the usual Dehn twists on surfaces. The restricted dynamics is a generalization of the linked twist map and related to the egg-beater map in the work of \cite{PS}. Positive topological entropy is only one of the many consequences of the presence of the linked twist map.
Additionally, we show that the dynamics on the subsystem exhibits hyperbolic behavior, characterized by positive Lyapunov exponent on a full Lebesgue measure set of the subsystem. Using Pesin's stable manifold theory (Theorem \ref{thm:Pesin}), we further demonstrate the existence of two families of Lagrangian submanifolds invariant under the composite symplectic Dehn twists.

\begin{thm}\label{foli} {Let $M$ be a symplectic manifold that contains an $A^n_m$-configuration, where $S_i$ are the embedded Lagrangian spheres in the configuration. Then there exists symplectic Dehn twists $\tau_i$ on $S_i$ such that for $\tau=\tau_1^{k_1}\cdots \tau_m^{k_m}$ with $k_ik_{i+1}<0$ for each $i=1,2,\ldots, m-1$,}  there exist two families of Lagrangian submanifolds $\cF^s, \cF^u$ on $M$ in a neighborhood $N$ of $\cup S_i$ which are invariant under $\tau$. Furthermore, there is an invariant subset $P$ of $N$ admitting an $\mathrm{SO}(n)$ action that commutes with $\tau$ and preserves each leaf in $\cF^s$ and $\cF^u$, such that
	\begin{enumerate}
		\item 	$P/\mathrm{SO}(n)$ is a two-dimensional manifold with boundary,
		\item  either of $\cF^s/\mathrm{SO}(n)$ and $\cF^u/\mathrm{SO}(n)$ has full Lebesgue measure on $P/\mathrm{SO}(n)$,
		\item  for a.e. $x\in P/\mathrm{SO}(n)$, the leaf $f^s(x)$ $($resp. $f^u(x))$ of $\cF^s/\mathrm{SO}(n)$ $($resp. $\cF^u/\mathrm{SO}(n))$ passing through $x$ is contracted exponentially under iterations of $\tau$ $($resp. $\tau^{-1})$.
	\end{enumerate}
\end{thm}

The question of existence of stable/unstable Lagrangian foliations for pseudo-Anosov elements generalizing Nielsen-Thurston theory was initially posed as an open problem in \cite{DHKK}. Our Theorem \ref{foli} presented in this paper offers a non-trivial example of existence of stable and unstable  laminations for symplectic diffeomorphisms. Furthermore, the presence of nonuniformly hyperbolic dynamics yields various standard consequences. For instance, it leads to the existence of horseshoes on $P/\mathrm{SO}(n)$ that are conjugate to Bernoulli shifts. As a result, exponential growth of periodic orbits and other related phenomena such as certain forms of ergodicity can be observed. For further exploration of these topics, interested readers are referred to Chapter S.5 of \cite{Katok}.

\subsection{Exponential growth of the ranks of Floer cohomology groups}
 {Theorems \ref{thm:class}, \ref{2twists}, \ref{mtwist} and \ref{foli}} rely on a specific choice of representative in the symplectic mapping classes, analogous to the Nielsen-Thurston theory. As a result, Theorems~\ref{2twists}--\ref{foli} may not generally remain robust under perturbations. However, due to the symplectic nature of the problem, it is natural and necessary to describe the complexity of the dynamics using symplectic invariants that are insensitive to the choice of representative in the symplectic mapping classes.

In this paper, we consider the rank of Floer cohomology as the invariant that characterizes the dynamics. We primarily consider Lagrangian Floer cohomology group $\HF(S_1, \tau^n S_2)$, where $\tau$ represents the composite symplectic Dehn twists as mentioned earlier. Throughout the paper, unless otherwise specified, we work with Lagrangian Floer cohomology using $\mathbb{Z}/2\mathbb{Z}$ coefficients.

Let us define the relative symplectic growth rate as follows:

\begin{defi}[Relative symplectic growth rate]
Let $(M, \omega)$ be a symplectic manifold, and let $\phi$ be a symplectomorphism of $M$. Consider a pair of connected compact Lagrangian submanifolds $(L_1, L_2)$ in $M$. Assume that the Lagrangian Floer cohomology groups $\HF(L_1, \phi^n(L_2))$ are well defined for all $n \in \mathbb{N}$. We define the symplectic growth rate of $\phi$ relative to the pair of Lagrangians $(L_1, L_2)$ as:

$$\Gamma(\phi, L_1, L_2) = \liminf_{n\to\infty} \frac{\log \mathrm{rank}\; \HF(L_1, \phi^n(L_2))}{n}.$$
\end{defi}

From the works of Arnold \cite{Ar2} and Seidel \cite{Se2}, the relative symplectic growth rate (when defined) has a uniform bound that only depends on $M$ and $\phi$. 

We now proceed to demonstrate that exponential growth can be achieved by composite symplectic Dehn twists, in the symplectic manifolds $A^n_m$. This is summarized in the following theorem:

\begin{thm}\label{thm:symrate}
Let $(M^{2n}, \omega_0)$ be the $A_2^n$-manifold\ with Lagrangian spheres $(S_1, S_2)$. Consider the symplectic Dehn twists $\tau_1$ and $\tau_2$ along the spheres $S_1$ and $S_2$, respectively. Then the following  hold:
\begin{enumerate}
\item For any $k, \ell \in \mathbb{Z}$ with $k\ell \neq 0,1,2,3,4$, we have $\Gamma(\tau_1^k\tau_2^\ell, S_1, S_2) > 0$.
\item For any $k, \ell \in \mathbb{Z}$ with $k\ell \in \{0,1,2,3,4\}$, we have $\Gamma(\tau_1^k\tau_2^\ell, S_1, S_2) = 0$.
\item For $i, j \in {1,2}$ with $i \neq j$, we have $\Gamma(\tau_i^2, S_j, S_j) = 0$.
\end{enumerate}
\end{thm}
While the symplectic growth rate is zero in case (3) of the theorem, it can be shown that the rank of $\HF(S_j, \tau_i^{2n}(S_j))$ grows with a polynomial rate (see Section \ref{subsec:proof}). A similar computation of the rank of Lagrangian Floer homology reveals that the ``slow volume growth behavior" of the symplectic Dehn twist on the cotangent bundle $T^*\mathbb{S}^d$ of the $d$-sphere is positive (see \cite{FS}). 

We extend the results to the more general $A_m^n$-type configuration as follows:

\begin{thm}\label{thm:symrate2}
Consider the $m$ embedded Lagrangian spheres $S_i, 1\leq i\leq m,$ in an $A_m^n$-configuration of Lagrangian $n$-spheres. Let $\tau_i$, $i=1,\ldots,m$, be the symplectic Dehn twists along the $n$-spheres $S_i$ respectively. For $\tau=\tau_1^{k_1}\cdots\tau_m^{k_m}$, where each $k_i\in\mathbb{Z}\setminus{0}$ and $k_{i}\cdot k_{i+1}<0$ for all $i=1,\ldots,m-1$, we have $\Gamma(\tau,S_k,S_\ell)>0$ for all $k,\ell\in{1,\ldots,m}$ with $k\neq \ell$.
\end{thm}
These results are proved by combining the techniques developed in \cite{KS} and the classical theory of mapping class groups of surfaces. The last theorem can be considered as an analogue of Theorem 3.1 in \cite{Pe}. An important innovation in the proof is the reduction of the calculation of the rank of Floer homology group to the intersection number on surfaces (see Lemma \ref{multiintersect}), which can be considered as a higher dimensional generalization of the corresponding result in \cite{DHKK}. 


\subsection{Speculations on the complexity of the dynamics of composite Dehn twists}
It is important to note that the study presented so far does not uncover the full complexity of the dynamics of composite symplectic Dehn twists. In this section, we aim to discuss the richness of the dynamics of composite symplectic Dehn twists.

In the field of hyperbolic dynamics, a significant indicator of the complexity of a differentiable dynamical system is the presence of positive metric entropy with respect to the volume measure. This is equivalent to having a positive Lyapunov exponent on a set of positive measure, where the Lyapunov exponent is defined as $\lambda(x) := \limsup_{n\to \infty} \frac{\log |Df^n(x)|}{n}$ for a map $f: M\to M$ and $x\in M$. The existence of positive metric entropy allows for strong conclusions regarding ergodicity. However, establishing positive metric entropy for a given system is generally a highly challenging task.

A prominent example in this context is the ``standard map" $f: \mathbb{T}^2 \to \mathbb{T}^2$ defined by $(x,y) \mapsto (x+y+\epsilon \sin(x), y+\epsilon \sin(x))$. It is conjectured that this map exhibits positive metric entropy for the Lebesgue measure, for all $\epsilon > 0$. To help topologists appreciate this conjecture, we can reframe it as follows: For a map with an Anosov mapping class in $\mathrm{SL}_2(\mathbb{Z})$ such as $\left(\begin{array}{cc}
2&1\\1&1\end{array}\right)$, as well as its small perturbations, it is relatively straightforward to establish positive metric entropy. However, for a reducible element such as  $\left(\begin{array}{cc}
1&1\\0&1\end{array}\right)$, the metric entropy is zero. The standard map corresponds to an element with the mapping class $\left(\begin{array}{cc}
1&1\\0&1\end{array}\right)$, and we would like to show that the complicated dynamical behaviors appear naturally by perturbing $\left(\begin{array}{cc}
1&1\\0&1
\end{array}\right)$. However, proving this conjecture has proven to be an extremely difficult task.
From the same perspective, it is known that for $\dim(M)=4$ or $12$ the compositions of Dehn twists $\tau_1^k\tau_2^\ell$, where $k,l\in 2\mathbb{Z}$, are isotopic to the identity but symplectically nontrivial \cite{Se}. This implies that we obtain diffeomorphisms with trivial mapping class but intricate dynamics.

Motivated by the analogy with the standard map conjecture, we put forward the following conjecture: the measure-theoretic entropy of the composition $\tau_1^k\tau_2^\ell$ as described in Theorem \ref{2twists} is positive.

\begin{conj}\label{mec}
Let $S_1$ and $S_2$ be two Lagrangian $n$-spheres in an $A_2^n$-configuration, intersecting transversely at a single point. Then, there exist symplectic Dehn twists $\tau_1$ and $\tau_2$ of $S_1$ and $S_2$ respectively, such that for any $k,\ell\in\mathbb{Z}$ with $k\ell < 0$, the measure-theoretic entropy of the composition $\tau_1^k\tau_2^\ell$ of symplectic Dehn twists with respect to the Lebesgue measure $vol$ is positive, i.e., $h_{vol}(\tau_1^k\tau_2^\ell)>0$.
\end{conj}

Conjecture \ref{mec} represents a significantly stronger statement than Theorem \ref{2twists}. If this conjecture were true, it would provide an intriguing example of a dynamical system with positive entropy that is homotopic to the identity. However, we acknowledge that this conjecture poses a considerable challenge, as we will discuss in Section \ref{me}.

\subsection{Organization of the paper}


In Section~\ref{sec:pre}, we provide a review of the definitions of symplectic Dehn twists, plumbing spaces and the construction of pseudo-Anosov maps on surfaces, and give related classical results. In Section~\ref{sec:rot&invo}, we set the stage of the construction of rotation and involution maps on plumbing domains and introduce the notions of ``admissible curves" and ``admissibly isotopic" to reconstruct Lagrangian submanifolds and the Lagrangian isotopies from them. 

In Section \ref{topent}, we find a subsystem of $A_m^n$ equivalent to Dehn twists on surfaces and prove Theorems \ref{2twists} and \ref{mtwist}, which establish positive topological entropy for composite symplectic Dehn twists. In Section \ref{pf}, we establish the hyperbolicity of the subsystem, demonstrating the presence of positive Lyapunov exponents on a positive measure set. In Section \ref{SFoliation}, we prove Theorem \ref{foli}, establishing the existence of local stable and unstable Lagrangian manifolds within the dynamics of composite symplectic Dehn twists.

In Section \ref{sec:invo}, we provide the proof for Theorems \ref{thm:symrate} and \ref{thm:symrate2}, which focus on the growth of Floer cohomology groups. In section~\ref{sec:classfy}, we prove our classification result--Theorem~\ref{thm:class}. 
In Section \ref{SIntersection}, we establish a connection between the rank of Floer cohomology groups and intersection numbers, completing the proof of our main technical Lemmas \ref{LmIntersection} and \ref{multiintersect} from Section \ref{sec:invo}. In Section \ref{me}, we discuss the conjecture on metric entropy and address the challenges associated to it.

Finally, we include two appendices. Appendix \ref{apd:etp} provides preliminaries on topological entropy, metric entropy, and Pesin theory for non-uniformly hyperbolic dynamics. Appendix \ref{apd:fl} presents the preliminaries on Lagrangian Floer cohomology. 

\section*{Acknowledgment}
We would like to express our sincere appreciation to Weiwei Wu for engaging in several insightful discussions related to this work. We are also grateful to Simion Filip, Leonid Polterovich, Amie Wilkinson, and Seraphina Lee for their valuable comments and suggestions, which have greatly enriched this research. We thank the referees for many valuable suggestions and comments that greatly improves the results, especially Theorem \ref{thm:class}, as well as the writing of the paper.

J.X. acknowledges the support from the New Cornerstone Investigator program, the National Natural Science Foundation of China No.11790273 and No.12271285 and the Xiaomi endowed professorship at Tsinghua University. W.G. is grateful for the support from the NSFC grant 12271285.
This work was conducted while Z.W. was at Tsinghua University.

\section{Preliminaries}\label{sec:pre}
In this section, we provide some preliminaries on symplectic Dehn twist and the plumbing space, and set up notations for later sections. We also include some basics of pseudo Anosov diffeomorphisms such as the Penner's construction.
\subsection{Symplectic Dehn twist and plumbing space}\label{prl}
In this section we shall introduce the preliminary definitions of symplectic Dehn twist and plumbing space.

\begin{defi}[Symplectic Dehn twist]\label{symdeh}
Let $\omega_0$ be the standard symplectic form on $T^*\bS^n$.  {The complement of the zero section $T^*\bS^n\setminus \bS^n$ admits a Hamiltonian circle action $\sigma$ with moment map given by the length function $\mu(\xi)=|\xi|$ with respect to the standard metric.  }
We define the model Dehn twist $\tau $on $T^*\bS^n$ to be 
\beq\tau(x,v)=\left\{\begin{array}{lll}\sigma_{r({|v|})}(x,v) & & |v|>0,\\(A(x),0)& & |v|=0,
\end{array}\right.\eeq 
for $x\in \bS^n,v\in T^*_x\bS^n$, where $A(x)$ is the antipodal map on $\bS^n$ and $r(t)\in C^\infty(\R_{\ge 0},\R_{\ge 0})$ is a smooth function such that $r(0)=\pi$, $r(t)=0$,  {$\frac{dr}{dt}(s)<0, 0< s<\epsilon$,}  and $\frac{d^kr}{dt^k}(0)=0,\forall\ k\geq 1$. Here $|v|$ is taken under the standard metric.

More concretely, under the identification $$T^*\bS^n\simeq \{(x,v)\in\R^{n+1}\times \R^{n+1}\ |\ x,v\in \R^{n+1}, |x|^2=1,\langle x,v\rangle=0\} \subset \R^{2n+2},$$ the symplectic Dehn twist acts by
\beq \label{deheq} \tau\begin{pmatrix}x\\v\end{pmatrix}=\left\{\begin{array}{lll} \begin{pmatrix}\cos(  r(|v|))& \frac{\sin(  r(|v|))}{|v|}\\-\sin(  r(|v|))|v|&\cos(  r(|v|))\end{pmatrix}\begin{pmatrix}x\\v\end{pmatrix}& &v\neq 0\\\begin{pmatrix}-x\\0\end{pmatrix} && v=0
\end{array}\right..\eeq

Since $r(t)=0,\forall\  t \ge \epsilon$, we have that $\tau$ acts by identity outside $T^*_\epsilon(\bS^n)=\{(x,v)\in T^*\bS^n\ |\ |v|<\epsilon\}$, the $\epsilon$-neighborhood of $\bS^n$.

For any Lagrangian $n$-sphere $L$ in a sympelectic manifold $M$, we can find a symplectomorphism $\phi: T^*_\epsilon(\bS^n)\to N(L)$ where $N(L)$ is a tubular neighbourhood of $L$. The symplectic Dehn twist on $L$ is defined as $\tau_L=\phi\tau\phi^{-1}$.

\end{defi}

In \cite{Se}, Seidel gave a configuration of $3$ Lagrangian spheres in which the iterations of $\tau^2$ are not symplectically isotopic to identity.
In this paper, we consider the compositions of symplectic Dehn twists which are supported in small neighborhoods of some Lagrangian spheres that intersect transversely. A model example of the configuration of such Lagrangian spheres is the plumbing space.

The plumbing space of two $n$-spheres $S_1$ and $S_2$, plumbed at $p_1\in S_1$ and $p_2\in S_2$,  is given by locally gluing a neighborhood of $p_1$ in the disk unit cotangent space $D_{p_1}^*S_1$ with a neighborhood of $p_2$ in the unit disk cotangent space $D_{p_2}^*S_2$ and then taking an appropriate corner-smoothing. Following~\cite{Ab,AS} we define the plumbing space as follows.

\begin{defi}[Plumbing]\label{plumb}
Given two $n$-spheres $S_1$ and $S_2$ with points $p_1\in S_1$ and $p_2\in S_2$,  we pick a Riemannian metric on $S_i$ whose restriction to a neighborhood of $p$ is mapped by $\psi_i$ locally isometric to the ball of radius $2$ in $\R^n$ with $p_i$ mapping to the origin and $U_i\subset S_i$ to the ball of radius $1$. This induces a symplectomophism $T^*\psi_i$ from $D^*U_i$ to $\mathbb{D}^n\times \mathbb{D}^n\subset D^*\R^n\simeq \R^{2n}$, where $D^*U_i, D^*\R^n$ stand for unit disk bundles. Take a standard linear symplectomorphism $J: T^*\R^n\to T^*\R^n$ given by $J(x_i)=dx_i,J(dx_i)=-x_i,i=1,2,...,n$. We have a symplectomorphism from $D^*U_1$ to $ D^*U_2$ given by $\psi=(T^*\psi_2)^{-1}JT^*\psi_1$. The resulting \emph{plumbing space} $T^*S_1\sharp T^*S_2$ is given as the completion of the plumbing domain $$P_\psi(S_1,S_2):=D^*S_1\cup_\psi D^*S_2/x\sim \psi(x)\;\hbox{for all}\; x\in D^*S_1.$$
\end{defi}

We can also plumb multiple $n$-spheres together, by plumbing the spheres one by one.  {In this paper, we shall consider the plumbing given by picking antipodal points $p_i,q_i\in S_i$ and plumbing $T^*S_i$ and $T^*S_{i+1}$ by identifying $p_i$ with $q_{i+1}$, for $1\le i\le m-1$.} More precisely, taking $\psi_{i}$ to be the local identification plumbing $p_i\in D^* S_i$ and $q_{i+1}\in D^*S_{i+1}$ together, we obtain the plumbing space $T^*S_1\sharp\cdots\sharp\T^*S_m$ as the completion of the plumbing domain $P_\psi(S_1,S_2,...,S_m)$  where $\psi=(\psi_1,\ldots,\psi_{m-1})$ given by $$D^*S_1\cup_{\psi_1}\cdots\cup_{\psi_{m-1}}D^*S_m =\bigcup\limits_{i=1}^m D^*S_i/x\sim \psi_i(x),\;\hbox{for all}\; x\in D^*S_i,\;\forall i.$$ 
 The restriction of $\psi$ gives the plumbing map between unit disk cotangent circles.


\begin{rmk}
Since the dynamical systems and symplectic topology we are primarily concerned with in this paper are concentrated near the zero section of the region $D^*S_1\cup_{\psi_1}\cdots\cup_{\psi_{m-1}}D^*S_m$, and the computations of topological entropy and Floer homology groups of compositions of symplectic Dehn twists are not affected by the completion of a small Liouville domain $W$ containing the zero section in $D^*S_1\cup_{\psi_1}\cdots\cup_{\psi_{m-1}}D^*S_m$, we always do our computations inside $D^*S_1\cup_{\psi_1}\cdots\cup_{\psi_{m-1}}D^*S_m$,  and  by abuse of notation, sometimes we do not distinguish between the notations 
$P_\psi(S_1,S_2,...,S_m)$ and $W$ provided that no confusion arises. 

\end{rmk}

The following example illustrates that transversely intersecting Lagrangian 2-spheres, as given in the plumbing space, occur naturally on algebraic surfaces.
\begin{exm}[$A^n_m$-configuration, see Section 8 of \cite{Se}]\label{Am}
We consider the symplectic manifold $A_m^n$ given as an affine hypersurface $$A_m^n=\bigg\{(x_1,x_2,\ldots,x_{n+1})\in \C^{n+1}\; \big|\;\ x_1^2+x_2^2+...+x_{n}^2=x_{n+1}^{m+1}+\frac{1}{2}\bigg\}\subset \C^{n+1}$$ in $\C^{n+1}$ equipped with the standard symplectic form $\omega_{\C^{n+1}}=\frac{i}{2\pi}\sum^{n+1}_{j=1}dz_j\wedge d\bar{z}_j$.

For $n=2$, we briefly describe these Lagrangian spheres. Following Seidel~\cite{Se}  we consider the projection $\pi:A_m^n\to \C^n:(x_1,x_2,\ldots,x_n,x_{n+1})\to (x_1,x_2,\ldots,x_n)$, then $\pi$ is a $(m+1)$-fold covering branched along $C=\{x_1^2+\ldots +x_n^2=\frac{1}{2}\}$ and the covering group is generated by $$\sigma: (x_1,x_2,\ldots,x_n,x_{n+1})\to (x_1,x_2,\ldots,x_n,e^{\frac{  2\pi i}{m+1}}x_{n+1}).$$

        Take the figure-8 map $$f:\R^{n+1}\supset \mathbb S^{n}\to \C^n\backslash C: (t_1,t_2,\ldots, t_{n+1})\mapsto (t_2(1+it_1),t_3(1+it_1),\ldots, t_{n+1}(1+it_1)).$$ Then the map is an immersion with one double point $0$ at which the north and south pole $(\pm1,0,\ldots,0)$ of $\mathbb S^n$ which meets transversely in $\C^n$. Lifting the immersion to $A_m^n\subset \C^{n+1}$, the map becomes an embedding $\tilde{f}$ with $\tilde{f}(-1,0,\ldots,0)=\sigma (\tilde{f}(1,0,\ldots,0))$. This shows that $$S_1=\mathrm{Im}(\tilde f),\ S_2=\sigma(S_1),\ldots,\ S_m=\sigma^{m-1}S_1$$ are a collection of smoothly embedded spheres whose only intersections are $S_i\cap S_{i+1},1\le i\le m-1$. We can further show that these embedded spheres can be perturbed into transversely intersecting Lagrangian spheres $L_i$ that such that $$|L_i\cap L_j|=\left\{\begin{array}{lll}1, && |i-j|=1\\ 0,&& |i-j|>1.\end{array}\right.$$

\end{exm}

The symplectic manifold $A_m^n$ can be described as the plumbing space of $m$ complex n-dimensional spheres at antipodal points along a straight line, see~\cite{AS}. For the complex $2$-dimensional Milnor fiber, we record the following theorems of Weiwei Wu on the classification of Lagrangian spheres in an $A^2_m$-surface.

\begin{thm}[{\cite[Theorem~1.1]{wu}}]\label{thm:lagiso}
    Lagrangian spheres in $A^2_m$-surface are Hamiltonian isotopic to the standard spheres up to symplectic Dehn twists along them.
\end{thm}

\begin{thm}[{\cite[Theorem~1.3]{wu}}]\label{thm:hamiso}
    Every compactly supported symplectomorphism of $A^2_m$-surface is Hamiltonian isotopic to a composition of symplectic Dehn twists along the standard spheres.
\end{thm}

\subsection{Construction of pseudo-Anosov maps on surfaces}

By definition, a collection of isotopy classes of simple closed curves in a closed surface $S$ with genus $g\geq 2$ \emph{fills} $S$  if the
complement in $S$ of the representatives in the surface is a collection of topological disks (equivalently, any simple closed curve in $S$ has positive geometric intersection with some isotopy class in the collection). A \emph{multicurve} in $S$ is the union of a finite collection of disjoint simple closed curves in $S$.
\begin{thm}[{Penner~\cite[Theorem~3.1]{Pe}}]\label{thm:PA}
	Let $A=\{c_1,\ldots,c_n\}$ and $B=\{d_1,\ldots,d_m\}$ be two multicurves in a surface $S$ which fill $S$.
	Then any composition of positive powers of Dehn twists $T_{c_i}$ along $c_i$ and negative powers of Dehn twists $T_{d_j}$ along $d_j$ is pseudo-Anosov, where each $c_i$ and $d_j$ appears at least once.
\end{thm}

\begin{thm}[\cite{FLP}]\label{thm:anos}
	Let $f$ be a pseudo-Anosov mapping class of a closed surface $S$ of genus $g\geq 2$ with stretch factor $\lambda>1$. Then for any two isotopy classes of curves
	$\alpha$ and $\beta$ we have
	$\lim_{n\to\infty}\sqrt[n]{\langle\alpha,f^n(\beta)\rangle}=\lambda$  where $\langle\cdot,\cdot\rangle$ is the algebraic intersection form. 
\end{thm}

\begin{figure}[ht]
	\begin{center}
		\includegraphics[width=0.8\textwidth]{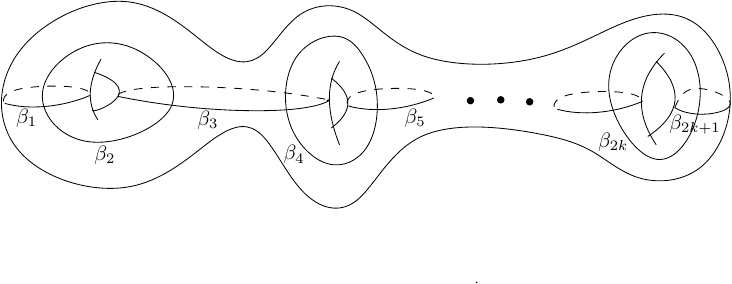}
	\end{center}
     \caption{Curves filling a surface of genus $g=k$}\label{fig:surface}
\end{figure}

\section{Rotation and involution maps on the plumbing domain}\label{sec:rot&invo}
An important observation that we use in this paper is to reduce the problem to a sunsystem which is a 2-dimensional surface. This involves an involution procedure for the reduction and a rotation procedure that reconstruct the higher dimensional manifolds from the 2-dimensional one. 

In this section, we set the stage of our construction by defining these maps is the plumbing of cotangent bundles of spheres.

\subsection{Rotation map on the standard sphere}
First we define the rotation map on the standard sphere $\mathbb{S}^n$.

\begin{defi}Let $\gamma$ be the great circle $\{(y_1,0,\ldots,0,y_{n+1})\}$ on $\mathbb S^n$ passing through the two points  $\mathbf{n}=(0,\ldots,0,1)$ and $-\mathbf{n}=(0,\ldots,0,-1)$ called north pole and south pole respectively. Let $\mathbb S^{n-1}=\{(\theta,0)\in \R^{n}\times \R\ |\ |\theta|^2=1\}$ be the the equator of $\mathbb S^n$. We suppose that $\gamma$ has angle coordinate $\theta=(1,0,\ldots,0)\in \mathbb S^{n-1}$ and $\gamma_\theta$ is a great circle with angle $\theta\in \mathbb S^{n-1}$. We associate to $\theta$ the rotation map $\mathcal R_\theta$, an element in $\mathrm{SO}(n)$ sending $T^*\gamma$ to $T^*\gamma_\theta$ isometrically. More precisely, let $x=(\theta,0)$ with $\theta\in \mathbb{S}^{n-1}$ be a point on the equator. We define
	\[
	 \mathcal{R}_{\theta}:\ T^*\gamma=\{(y_1,0,\ldots,0,y_{n+1},\;v_1,0,\ldots,0,v_{n+1})\in T^*\mathbb S^n\}\longrightarrow T^*\mathbb S^n
	\]
	\[
	(y_1,0,\ldots,0,y_{n+1},\;v_1,0,\ldots,0,v_{n+1})\longmapsto (y_{n+1}\cdot \mathbf{n}+y_1\cdot x,\;v_{n+1} \cdot \mathbf{n}+v_1 \cdot x).
	\]
	Hereafter we call $\mathcal{R}_{\theta}$ the \emph{rotation map}.
\end{defi}

Note that for all points $\theta\in\mathbb S^{n-1}$ we have $\mathcal{R}_{\theta}(\pm\mathbf{n},\mathbf{0})=(\pm\mathbf{n},\mathbf{0})$, and if $p\notin\{(\pm\mathbf{n},\mathbf{0})\}$ in $ T^*\gamma$ the union of images
$\cup_{\theta\in \mathbb S^{n-1}}\mathcal{R}_{\theta}(p)$ is an embedded $(n-1)$-sphere in $T^*\mathbb S^n$.

\begin{lem}\label{LmLagrangian}
Let $c:\ [0,1]\to T^*\gamma$ be an embedded curve. Then the set
\begin{equation}\label{e:Lag}
	L_c=\bigcup_{\theta\in \mathbb S^{n-1}}\mathcal{R}_{\theta}\big(c([0,1])\big)
\end{equation}
is an immersed Lagrangian submanifold $($possibly with boundary$)$ of $T^*\mathbb S^n$ which possibly has non-smooth points at $(\pm\mathbf{n},\mathbf{0})$.

	
\end{lem}
\begin{proof} Let $c(t)=(y_1(t),0,\ldots,0,y_n(t),\;v_1(t),0,\ldots,0,v_n(t))$. By definition we have
 $$L_c=\{(y_1(t)\theta,y_n(t); v_1(t)\theta,v_n(t))\ |\ t\in[0,1],\theta\in \mathbb{S}^{n-1}\subset \R^n\}.$$ Thus it is a smoothly immersed submanifold outside the south and north pole.

 To finish the proof we notice that $\omega|_{L_c}=0 $ if and only if $\sum_{i=1}^{n+1} dy_i\wedge dv_i|_{L_c}=0$. To verify the latter, we compute
\begin{equation*}
	\begin{aligned}
&\sum_i d(y_1(t)\theta_i)\wedge d(v_1(t)\theta_i)+dy_{n+1}(t)\wedge dv_{n+1}(t)\\
&=(y_1(t)v_1^\prime(t)-v_1(t)y_1^\prime(t))\sum_i \theta_i d\theta_i=\frac{(y_1(t)v_1^\prime(t)-v_1(t)y_1^\prime(t))}{2}d\big(\sum_i \theta_i^2\big)
=0,\notag
\end{aligned}
\end{equation*}
where in the last inequality we have used $\sum \theta_i^2=1$. Thus we see that $L_c$ is locally Lagrangian.
\end{proof}

By the definition of the symplectic Dehn twist $\tau$ on $T^*\mathbb{S}^n$, we see that it is commutative with the rotation map.
\begin{lem}
For the symplectic Dehn twist $\tau$ on $T^*\mathbb{S}^n$ we have
	\begin{equation}\label{e:tauro}
		\mathcal R_\theta \circ\tau|_{T^*\gamma}=\tau\circ \mathcal R_{\theta}|_{T^*\gamma},
	\end{equation}
	and
	\begin{equation}
		\tau(L_c)=L_{\tau(c)}
	\end{equation}
   for any smooth curve $c:\ [0,1]\to T^*\gamma$.

\end{lem}

\subsection{Rotation map on the plumbing domain}\label{sec:subL}

Let $\gamma_i\subset S_i$, $i=1,\ldots,m$ be a geodesic circle which contains a plumbing point and its antipodal point (denoted by $p_i$ and $-p_{i}$) of $S_i$ as given in Section~\ref{prl}. Then for each $i=1,\ldots,m-1$, the plumbing map $\psi_i$ sends an open subset of $p_i$   in $D^*\gamma_i$  onto that of $-p_{i+1}$ in $D^*\gamma_{i+1}$. In other words,  the restriction of $\psi_i$ to $D^*\gamma_{i}$ gives rise to the plumbing map from $D^*\gamma_i$ to $D^*\gamma_{i+1}$, from which we get a plumbing domain denoted by $P_\psi(\gamma_1,\gamma_2,\ldots,\gamma_m)$.

We identify each $S_i$ with $\mathbb{S}^n$ through a map $\phi_i:\ S_i\to\mathbb{S}^n $ such that the image of the great circle $\gamma$ is exactly $\gamma_i$. Then for each $\theta\in\mathbb S^{n-1}$, the rotation maps defined above give rise to a symplectic map $\mathcal{R}_{\theta}^i:\ T^*\gamma_i\to T^*S_i$ by
\[
\mathcal{R}_{\theta}^i=(T^*\phi_i)^{-1}\circ\mathcal{R}_{\theta}\circ T^*\phi_i|_{T^*\gamma_i}
\] which we call the \textit{rotation map} on $S_i$.

To extend the Lagrangian submanifolds $L_c^i$ from a single sphere to the full space $P_\psi(S_1,S_2,\ldots,S_m)$, we need to know the effect of the rotation maps $\mathcal{R}_{\theta}^i$ under the local identification maps $\psi_i:\ T^*U_i^+\to T^*U_{i+1}^-$, where the open subsets $U_i^\pm\subset S_i$ are the preimages of the open subsets
 $$U_\pm=\{x=(x_1,x_2,...,x_{n+1})\in\R^{n+1}\ | \ |x|^2=1,\ \pm x_{n+1}>0\}\subset  \mathbb S^{n}$$
 under the map $\phi_i: \ S_i\to\mathbb S^{n}$.
\begin{lem}\label{rotcom}
We may choose  local identification maps
	$\psi_i$ of $P_\psi(S_1,\ldots, S_n)$ so that
 for any $\theta\in \mathbb S^{n-1}$ and any geodesic circle $\gamma$ of $S_i$ containing the plumbing point $p$ of $S_i$ with $S_{i+1}$, we have
	\begin{equation}\label{e:Rt}
		\mathcal{R}_{\theta}^{i+1}\circ \psi_i=\psi_i\circ \mathcal{R}_{\theta}^i\quad \hbox{on}\;D^*\gamma\cap U_i^+,
	\end{equation}
	where  $U_i^+$ is the plumbing region of $p_i$ in  $T^*S_i$.
	
\end{lem}
\begin{proof}
We take the plumbing map $\psi_i$ to be given by $(d\varphi_{i+1}^-)^{-1}Jd\varphi_i^+$ where $\varphi_j^\pm=\phi_j|_{U_j^\pm}\varphi_\pm: U_j^\pm\to \R^n$ with $$\varphi_\pm:\ (x_1,\ldots,x_{n+1})\mapsto \frac{\arccos(\pm x_{n+1})}{\sqrt{1-x_{n+1}^2}}(x_1,\ldots,x_n).$$ Then by the definition of $\mathcal R_\theta$ we  calculate that the conjugation of $\mathcal R^j_\theta$  via $d\varphi_j^\pm$ acting on $(d\varphi^\pm_j)^{-1}(T^*\gamma_j)\subset T^*\R^n$ is simply
$$d\varphi_j^\pm \mathcal R^j_\theta(d\varphi_j^\pm)^{-1}=d\varphi_{\pm}\mathcal R_\theta d\varphi_\pm^{-1}:\ (t,0,\ldots,0;s,0,\ldots,0)\mapsto (t\theta,s\theta) \in T^*\R^n$$ for any $j$.

Therefore, the rotation map rotates both the base coordinates $x=(t,0,\ldots,0)$ and the fiber coordinate $v=(s,0,\ldots,0)$ to the direction of $\theta$ while the plumbing map from $S_i$ to $S_{i+1}$ switches the roles of base and fiber via the complex map $J$. Thus we see that they are commutative.
\end{proof}

On the great circle $\gamma=\{(y_1,0,\ldots,0,y_{n+1})\subset \mathbb S^n\}$, the rotation  $$\mathcal{R}_{(-1,0,\ldots,0)}: (y_1,0,\ldots,0,y_{n+1},\;v_1,0,\ldots,0,v_{n+1})\longmapsto (-y_1,0,\ldots,0,y_{n+1},\;-v_1,0,\ldots,0,v_{n+1})$$ acts as an involution on $T^*\gamma$, and satisfies $\mathcal{R}_{(-1,0,\ldots,0)}(T^*\gamma)=T^*\gamma$ and $\mathcal{R}_{(-1,0,\ldots,0)}^2=Id$. By Lemma \ref{rotcom}, $\mathcal{R}^i_{(-1,0,\ldots,0)}$ agrees in the plumbed neighborhoods, therefore, one can define a map 
on the whole plumbing domain $P_\psi(\gamma_1,\gamma_2,\ldots,\gamma_m)$ given as follows.

 \begin{defi}\label{barR} We define the map $\bar{ \mathcal{R}}$ on the plumbing domain $P_\psi(\gamma_1,\gamma_2,\ldots,\gamma_m)$, given by
\begin{equation}\label{e:reflec}
	\begin{split}
		\bar{ \mathcal{R}}:\ P_\psi(\gamma_1,\gamma_2,\ldots,\gamma_m)\longrightarrow P_\psi(\gamma_1,\gamma_2,\ldots,\gamma_m),\\
		\bar{ \mathcal{R}}(z)=\mathcal{R}_{(-1,0,\ldots,0)}^i(z),\quad\hbox{if $z\in D^*\gamma_i$}.
	\end{split}
\end{equation}
\end{defi}

\subsection{The involution on the plumbing domain}\label{subsec:inv}
Now we consider an explicit involution on the plumbing domain $P_\psi(S_1,S_2,\ldots,S_m)$. This would be useful to compute the rank of Lagrangian Floer cohomology.

To do this, we first consider a canonical involution on the cotangent bundle of the standard $n$-sphere. As before, we write
$$T^*\mathbb S^n=\big\{(x,v)\in\mathbb{R}^{n+1}\times \mathbb{R}^{n+1}\ |\  |x|=1,\;\langle x, v\rangle=0\big\}.$$
Set $\mathbf{x}=(x_2,\ldots,x_n)$ and $\mathbf{v}=(v_2,\ldots,v_n)$.
We can define a symplectic involution map on $T^*\mathbb S^n$ by
\[
\iota_0(x_1,\mathbf{x},x_{n+1},v_1,\mathbf{v},v_{n+1})=(x_1,-\mathbf{x},x_{n+1},v_1,-\mathbf{v},v_{n+1}).
\]
Then the fixed point set of $\iota_0$ is precisely the cotangent bundle of a great circle $\gamma\in \mathbb S^n$
\begin{equation}\label{e:C}
	\big(T^*\mathbb S^n\big)^{\iota_0}=T^*\gamma=\big\{(x_1,\mathbf{0},x_{n+1},v_1,\mathbf{0},v_{n+1})\in T^*\mathbb S^n\big\}.
\end{equation}
Using the involution $\iota_0$ we can now define an involution on $P_\psi(S_1,S_2,\ldots,S_m)$ as follows.
For each sphere $S_i$ we pick the geodesic circle $\gamma_i$, which contains every plumbing point of $S_i$ and a symplectomorphism
$\psi_i:\ T^*S_i\to T^*\mathbb S^n$ such that $\psi_i$ maps the zero section $S_i$ of $T^*S_i$ to $\mathbb S^n$ of $T^*\mathbb S^n$, and $\gamma_i$ to $\gamma$. Then we define the involution $\iota_{\gamma_i}$ on $T^*S_i$ by
$
\iota_{\gamma_i}=(\psi_i)^{-1}\circ\iota_0\circ\psi_i
$.

Furthermore, by the definition of $P_\psi(S_1,S_2,\ldots,S_m)$ we can assume that $\iota_{\gamma_i}(p)=\iota_{\gamma_{i+1}}(p)$ for all points $p\in D^*S_i\cap D^*S_{i+1}$, $i=1,\ldots,m-1$. The \textit{involution} $\iota$ on $P_\psi(S_1,S_2,\ldots,S_m)$ is defined by
\begin{equation}\label{iota}
    \iota(p)=\iota_{\gamma_i}(p)\quad \hbox{if} \quad p\in  D^*S_i.
\end{equation}

By (\ref{e:C}), for each sphere $S_i$ we have
\begin{equation}\label{e:fix}
	\big(T^*S_i\big)^\iota=T^*\gamma_i.
\end{equation}
This implies that the fixed point set of $\iota$ is symplectomorphic to the plumbing domain $P_\psi(\gamma_1,\gamma_2,\ldots,\gamma_m)$,  { which is topologically a punctured surface}. Also, each zero section $S_i$ of $P_\psi(S_1,S_2,\ldots,S_m)$ is invariant under the map $\iota$.
Moreover, a direct calculation shows that $\iota$ and the symplectic Dehn twist $\tau_i$ along $S_i$ have the following relations:
\begin{equation}\label{e:inv+tst}
	\iota\circ\tau_i=\tau_i\circ\iota,\quad i=1,\ldots,m.
\end{equation}

 {Moreover, picking the plumbing map to be given as in the proof of Lemma \ref{rotcom}, a direct calculation shows that $\tau_i$ restricted to ${\rm Fix}(\iota)$ is exactly the normal Dehn twist around $\gamma_i$ in the surface $P_\psi(\gamma_1,\gamma_2,\ldots,\gamma_m)$. }

Furthermore, we note that the involution and the rotation map on $\mathbb{S}^n$ satisfy $$\iota \circ \mathcal R_{\theta}=\mathcal R_{\theta} \circ \iota$$
which gives us the following lemma.
\begin{lem}
	
Let $\gamma_i$ be circles chosen as above and $c:[0,1]\to P_\psi(\gamma_1,\gamma_2,\cdots, \gamma_m)$ be an embedded curve.  Then Lagrangian submanifold $L_c$ given as in (\ref{e:Lag}) is invariant under the involution $\iota$, that is,
	\begin{equation}\label{e:laginv}
		\iota(L_c)=L_c.
	\end{equation}

\end{lem}

\subsection{Admissible curves and Lagrangian isotopies}
In this section, we introduce the notion of admissible curves which generates Lagrangian submanifolds after the rotation map, and relate it to the Lagrangian isotopies. 
\begin{defi}\label{admiss}
	We call a smooth embedded curve $c: [0,1]\to P_\psi(\gamma_1,\gamma_2,\ldots,\gamma_m)$ \emph{admissible} if $\bar{ \mathcal{R}}(c([0,1]))=c([0,1])$ (see (\ref{e:reflec}) for the definition of $\bar{ \mathcal{R}}$).
The set of admissible curves in $P_\psi(\gamma_1,\gamma_2,\ldots,\gamma_m)$ is denoted by $\mathcal{C}_{ad}$.
\end{defi}

We remark here that the curve in the above definition is not neccesarily closed. By abuse of notation, we sometimes use the same notaion $c$ to stand for the image of a smooth curve $c: [0,1]\to P_\psi(\gamma_1,\gamma_2,\ldots,\gamma_m)$.

We also define the notion of two admissible curves being admissibly isotopic.
\begin{defi}\label{def:lagiso}
Two curves $\gamma_0,\gamma_1\in \mathcal{C}_{ad}$ are called \emph{admissibly isotopic} if there is an isotopy $(f_s)_{0\leq s\leq 1}$ in $\hbox{Diff}_c(P_\psi(\gamma_1,\gamma_2,\ldots,\gamma_m))$ such that $\gamma_s=f_s(\gamma_0)$ is a family of admissible curves connecting $\gamma_0$ to $\gamma_1$. 	
\end{defi}

\begin{lem}\label{lem:smlg} Let $c: [0,1]\to P_\psi(\gamma_1,\gamma_2,\ldots,\gamma_m)$ be a smooth embedded admissible curve. Then $L_c$ is a smoothly embedded Lagrangian submanifold in $P_\psi(S_1,\ldots,S_m)$.  {Furthermore, if $c: \mathbb S^1\to P_\psi(\gamma_1,\gamma_2,\ldots,\gamma_m)$ is a smooth embedded admissible closed curve, then $L_c$ is a closed smoothly embedded Lagrangian submanifold. }
\end{lem}
\begin{proof} By Lemma \ref{LmLagrangian} and Lemma \ref{rotcom}, $L_c$ is an immersed Lagrangian submanifold. Furthermore, the condition that $\bar{\mathcal R}(c)=c$ ensures that $L_c$ does not have any self-intersection.

In fact, we claim that $\mathcal{R}^{i(p)}_\theta(p)\neq \mathcal{R}^{i(q)}_{\theta'}(q)$ for any  $(p,\theta), (q,\theta^\prime)\in c([0,1])\times \mathbb S^{n-1}$ with $(p,\theta)\neq (q,\theta^\prime)$, unless $(q,\theta^\prime)=(\bar{\mathcal{R}}(p),-\theta)$, here $i(p)$ is the index $i$ such that $p\in T^*S_i$ (by Lemma~\ref{rotcom}, different choices of $i(p)$ give the same map $\mathcal{R}^{i(p)}_\theta$). If $\mathcal{R}^{i(p)}_\theta(p)=\mathcal{R}^{i(q)}_\theta(q)\in T^*S_i$ and $(q,\theta^\prime)\neq(p,\theta)$, we may take $i(p)=i(q)=i$ and conjugate the maps to the standard sphere. On the standard sphere we have $\mathcal R_\theta(T^*\gamma)\cap \mathcal R_{\theta^\prime}(T^*\gamma)=\{(\pm \textbf{n},\textbf{0})\}$ unless $\theta^\prime=\pm\theta$. Therefore, we see that $p=q\in \{(\pm \textbf{n},\textbf{0}) \}$ or $\theta=-\theta^\prime$. By the definition of $\mathcal{R_\theta}$, $\mathcal{R}^{i(p)}_\theta(p)=\mathcal{R}^{i(q)}_{-\theta}(q)$ is equivalent to $\bar{\mathcal{R}}(p)=q$. 

Therefore, we see that if $\bar{\mathcal R}(c)=c$, then $\cup_{\psi_i}\mathcal R_{\theta}^i (c[0,1])=\cup_{\psi_i}\mathcal R_{-\theta}^i (c[0,1])$, and $\big (\cup_{\psi_i}\mathcal R_{\theta}^i (c[0,1])\big)\cap \big(\cup_{\psi_i}\mathcal R_{\theta^\prime}^i (c[0,1])\big)\in \{\text{plumbing points}\}$ if $\theta^\prime\neq \pm\theta$. This shows that $L_c=\cup_{\psi_i}\cup_\theta R_\theta^i (c[0,1])$ has no self-intersection outside the plumbing points. Without loss of generality we may assume that the curve $c$ has unit speed. If $p=c(t)\in S_i$ is a plumbing point for $t\in [0,1]$, then the neighborhood $U_p$ of $p$ in $L_c$ is given by $U_p=\{ \mathcal R_\theta^i (c(t+s)), \theta\in \mathbb S^{n-1},s\in (-\epsilon,\epsilon)\}$, and since $\mathcal R_\theta^i$ is an isometry the only relation between the points in the set are $\mathcal R_\theta^i (c(t+s))=\mathcal R_{-\theta}^i (c(t-s))$. Therefore, $(|s|,\theta)$ gives a smooth polar coordinate of $U_p$, which shows that $L_c$ is also smoothly embedded at $p$. Therefore, $L_c$ is a smoothly embedded Lagrangian submanifold. 

If $c: \mathbb S^1\to P_\psi(\gamma_1,\gamma_2,\ldots,\gamma_m)$ is a smooth embedded admissible closed curve, then $L_c$, as a union of $\mathcal R_{\pm\theta}^i (c(\mathbb S^1))$ also has no boundary. This shows that $L_c$ is a closed smoothly embedded Lagrangian submanifold. 
\end{proof}

\begin{defi}\label{df:gecur}
Let  {$c:[0,1]\to P_\psi(\gamma_1,\gamma_2,\ldots,\gamma_m)$} be an admissible curve. We call $c$ the \emph{generating curve} of the Lagrangian submanifold $L_c$. We also say that $L_c$ is \emph{generated} by $c$.
\end{defi}

The next lemma is important for constructing Lagrangian isotopies through admissibly isotopic curves in $W:=P_\psi(\gamma_1,\gamma_2,\ldots,\gamma_m)$.
Denote by $\Sigma_\psi$ the set consisting of plumbing points and their antipodal points in $W$.
Note that if $c\in \mathcal{C}_{ad}$ is an admissible curve containing exactly two points of $\Sigma_\psi$ then $L_c$ is a Lagrangian sphere in $P_\psi(S_1,S_2,\ldots,S_m)$. Let $\hbox{Diff}_c(W,\Sigma_\psi)$ be the group of compactly supported diffeomorphisms $f:W\to W$ with $f(\Sigma_\psi)=\Sigma_\psi$.

\begin{defi}
A \emph{Lagrangian isotopy} on a symplectic manifold $(M,\omega)$ is a smooth family of Lagrangian embedding $\Phi:L\times [0,1]\to M$. Since $\Phi^*\omega$ vanishes on the fibers $L\times\{\hbox{point}\}$, we have $\Phi^*\omega=\theta_s\wedge ds$. We call a Lagrangian isotopy $\Phi$ exact if $\theta_s\in\Lambda^1(L)$ is exact for all $s$.	
\end{defi}

 We remark here that a Lagrangian isotopy on $M$ is exact if and only if it can be extended to an  ambient Hamiltonian isotopy of $M$ since if we write $\theta_s=dH_s$ on $L$ and extend $H_s\circ \Phi_s^{-1}$ to a time dependent Hamiltonian function on $M$ then we get the desired Hamiltion isotopy, and on the contrary, if we have a Hamiltonian isotopy $\phi_s$ on $M$, then the vector field $X_s$ generated  by $\phi_s$ gives the desired one-forms $\theta_s=\omega(\cdot,X_s)$.

The following lemma follows immediately from Lemma \ref{LmLagrangian}.
\begin{lem}\label{lem:isot}
	If $c_1,c_2\in \mathcal{C}_{ad}$ are admissibly isotopic in $P_\psi(\gamma_1,\gamma_2,\ldots,\gamma_m)$, then $L_{c_1}$ and $L_{c_2}$ are Lagrangian isotopic.
\end{lem}

By (\ref{e:laginv}), we know that for any $c\in \mathcal{C}_{ad}$, the Lagrangian $L_c$ is invariant under the involution $\iota$ associated to the circles $\gamma_1,\gamma_2,\ldots,\gamma_m$, i.e.,
\begin{equation}\label{e:Laginv}
	\iota(L_c)=L_c.
\end{equation}

\section{Topological entropy of composite Dehn twists}\label{topent}

In this section, we give the proofs of Theorems \ref{2twists} and \ref{mtwist} on positive topological entropy of composite Dehn twists. The key observation is that we can find an invariant subset of dimension~$2$ on which the system is reduced to the usual Dehn twists on surfaces.

Let $M$ be a symplectic manifold admitting $m$-Lagrangian spheres $S_i$ which pairwise intersect as in the $A^n_m$ configuration, which as given in our assumption in Theorems \ref{2twists} and \ref{mtwist}. By Weinstein's tubular neighborhood theorem, we may identify a tubular neighborhood $W$ of $\cup S_i\subset M$ with the neighborhood of $\cup S_i$ in the plumbing domain $P_\psi(S_1,S_2,\ldots ,S_m)$ of $m$ standard spheres. Let $\iota$ be the involution on the plumbing domain $P_\psi(S_1,S_2,\ldots ,S_m)$, given in Equation (\ref{iota}). 

We denote by $X$ the set of fixed points of the involution $\iota$ in $W$. Let $\gamma_i\subset S_i$ be the circle fixed by $\iota$ in the Lagrangian sphere $S_i$. Then we see that $X\subset P_\psi(\gamma_1,\gamma_2,\ldots,\gamma_m)$ is a real $2$-dimensional manifold which is homeomorphic to a subset of $\R^2$ with parallel sides identified, as shown in Figures \ref{m=3}, \ref{m=4}, \ref{modd} and \ref{meven}. Therefore, $X$ is homeomorphic to a surface $\hat \Sigma_{[m/2]}$ of genus $[m/2]$ with one punctured point if $m$ is even and $2$ punctured points if $m$ is odd. The Lagrangian spheres $S_i$ restricts to circles $\gamma_i\subset X$, as illustrated in Figures \ref{m=3}, \ref{m=4}, \ref{modd} and \ref{meven}.

\begin{figure}
    \centering
    \includegraphics[width=0.8\linewidth]{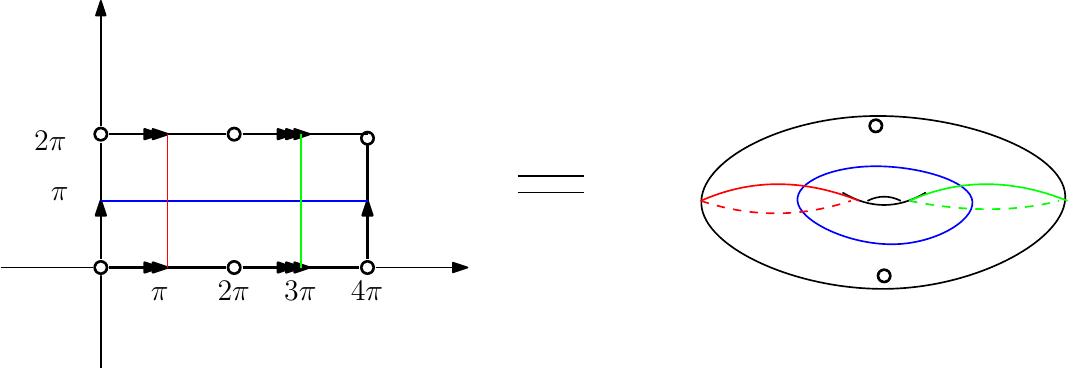}
 
    \caption{The surface $\hat \Sigma_{[m/2]}$ for $m=3$}
    \label{m=3}
\end{figure}
\begin{figure}
    \centering
    \includegraphics[width=0.82\linewidth]{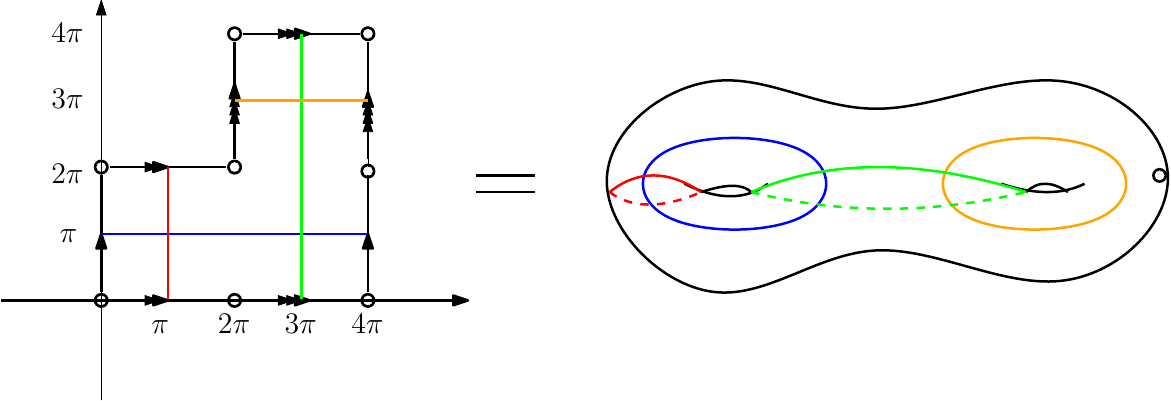}
    \caption{The surface $\hat \Sigma_{[m/2]}$ for $m=4$}
    \label{m=4}
\end{figure}
\begin{figure}
    \centering
    \includegraphics[width=0.8\linewidth]{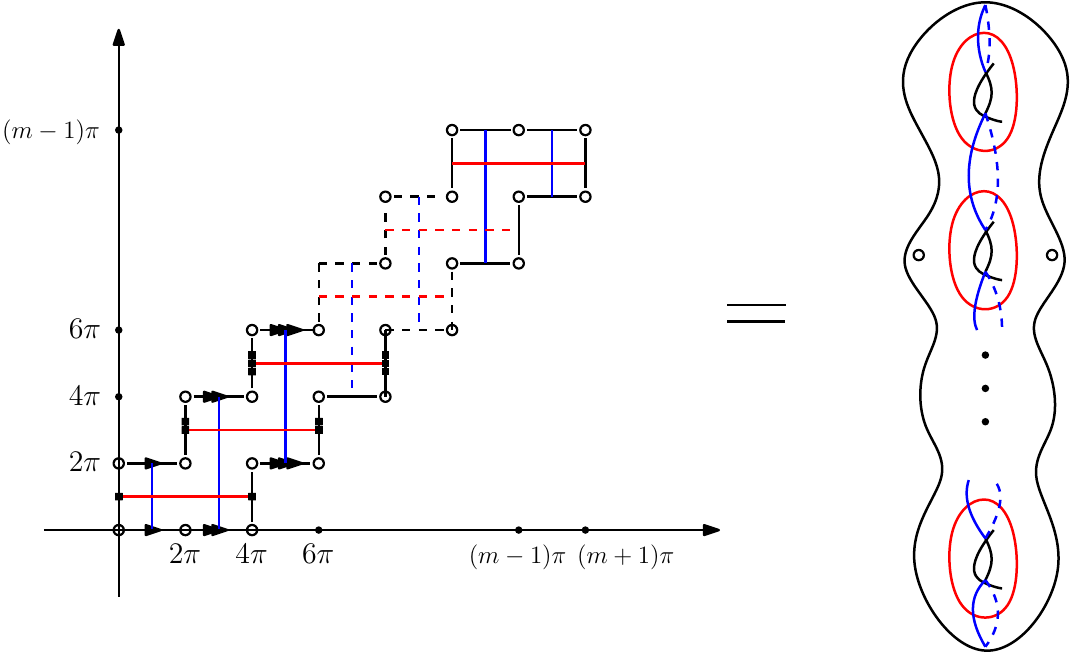}
    \caption{The surface $\hat \Sigma_{[m/2]}$ for odd $m$}
    \label{modd}
\end{figure}
\begin{figure}
    \centering
    \includegraphics[width=0.8\linewidth]{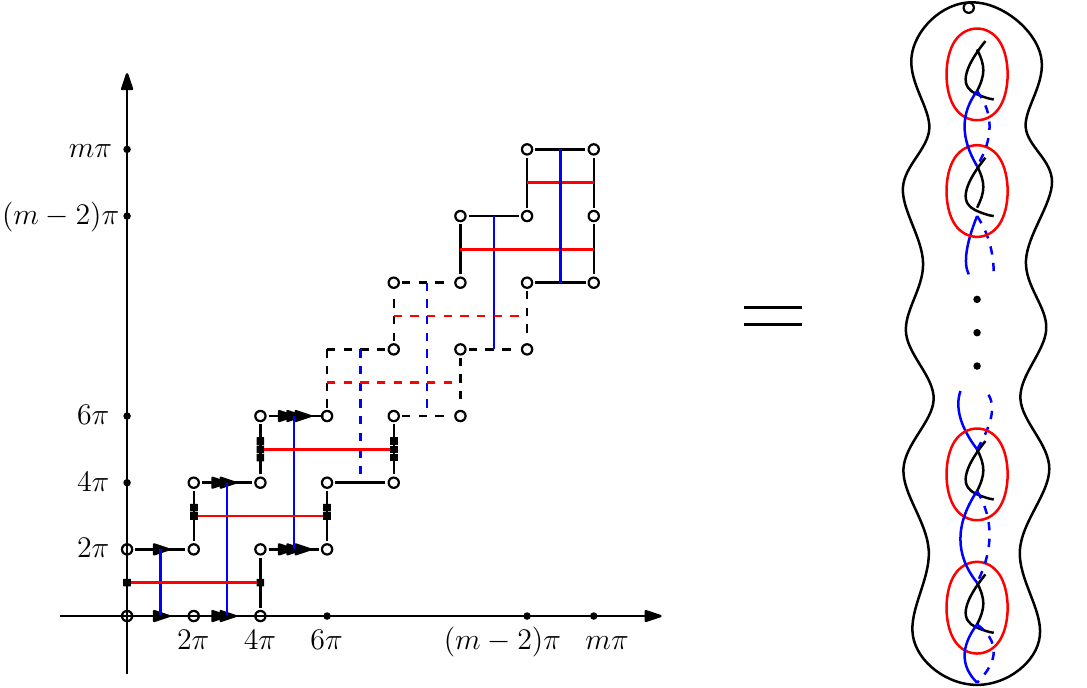}
    \caption{The surface $\hat \Sigma_{[m/2]}$ for even $m$}
    \label{meven}
\end{figure}

It is shown in~(\ref{e:inv+tst}) that we may choose the Dehn twists on $P_\psi(S_1,\ldots ,S_m)$ to be commutative with $\iota$ such that the restriction of a single symplectic Dehn twist $\tau$ on $X$ is exactly a Dehn twist of $\gamma_i$ on $P_\psi(\gamma_1,\ldots ,\gamma_m)$.

Thus, in the case $m=2$, the symplectic Dehn twists restricted to $X$ is simply the Dehn twists on a punctured torus.

\begin{lem}\label{LmLTM}
The set  $X$ is diffeomorphic to a subset of the punctured torus $\mathbb{T}^2-\{*\}$. We identify $\mathbb T^2-\{*\}=[0,2\pi]\times [0,2\pi]-\{(2\pi,2\pi),(2\pi,0),(0,2\pi),(0,0)\}/\sim$ where $(x,y)\sim (z,w)$ if $|x-z|=2\pi $ or $|y-w|=2\pi$. Then the restriction of $\tau_1$ and $\tau_2$ on $X$ is conjugate to the maps $T_1$ and $T_2$ of the form $$T_1(x,y)=(x+r(y-\pi),y);\quad T_2(x,y)=(x,y-r(x-\pi))$$ where $r(t)$ is the smooth function $r\in C^\infty(\R,\R)$ given in Definition \ref{symdeh}, extended to $\R$ by $r(t)+r(-t)=2\pi$.
\end{lem}

The topological entropy of the subsystem restricted to $X$ is no greater than the topological entropy of the whole system on $M$. On the other hand, the topological entropy of compositions of Dehn twists on the torus can be explicitly bounded below using Yomdin's inequality. This proves positive topological entropy of the composition of symplectic Dehn twists.

\begin{proof}[Proof of Theorem \ref{2twists}]

Let $\mathcal{T}=T_1^k T_2^\ell$ on $X\subset \mathbb{T}^2$, then under appropriate basis in $H^1(\mathbb{T}^2,\R)$, the induced linear map $\mathcal{T}^*: H^1(\mathbb{T}^2,\R)\to H^1(\mathbb{T}^2,\R)$ is exactly $\begin{pmatrix}1 & 1\\0& 1\end{pmatrix}^k\begin{pmatrix}1 & 0\\-1& 1\end{pmatrix}^\ell$.  {Let $\rho(\mathcal{T}^*)$ denote the largest absolute value of the eigenvalues of $\mathcal{T}^*: H_1(\mathbb{T}^2)\to H_1(\mathbb{T}^2)$. Then we have $\rho(\mathcal{T}^*)>1$ if $k\ell\neq 0,1,2,3,4$.}

By Yomdin's inequality (see Theorem \ref{yom}), we have $h_{top}(\mathcal{T})\ge \log \rho(\mathcal{T}^*)>0$. This shows that $$ h_{top}(\mathcal{T},\mathbb{T}^2)>0 .$$ Combined with Lemma \ref{LmLTM}, we have $$ h_{top}(\tau_1^k\tau_2^\ell, X)=h_{top}(\mathcal{T},\mathbb{T}^2\backslash\{*\})=h_{top}(\mathcal{T},\mathbb{T}^2)>0 .$$

Let $M$ be a symplectic manifold with two Lagrangian spheres transversely intersecting at one point, then a small neighborhood $U$ of the two spheres can be symplectically identified with a small neighborhood of $S_1\cup S_2$ in $P_\psi(S_1,S_2)$. We may pick the symplectic Dehn twists around $S_1$ and $S_2$ to be identity outside $U$. Therefore, we have $h_{top}(\tau, M)\ge h_{top}(\tau, U)= h_{top}(\tau_1^k\tau_2^\ell, A^n_m)\ge  h_{top}(\tau_1^k\tau_2^\ell, X)\ge h_{top}(\mathcal{T},\mathbb{T}^2)>0$.
\end{proof}

 {The proof of Theorem \ref{mtwist} for the case of $m$ Lagrangian spheres is similar.}
\begin{proof}[Proof of Theorem \ref{mtwist}]
For more than $2$ Lagrangian spheres in the $A^n_m$ configuration, the Dehn twists $\tau_i$ restricted to $X$ is conjugate to the usual Dehn twists around simple closed curves.

\begin{lem}\label{mLTM}
The manifold $X\subset P_\psi(S_1,S_2,\ldots,S_m)$ fixed by the involution $\iota$ is diffeomorphic to a subset of the punctured surface $\hat \Sigma_{[m/2]}$ which we view as a subset of $\R^2$ with parallel sides identified as in Figures \ref{modd} and \ref{meven}, and the Dehn twists $\tau_i$ on $X$ are conjugate to Dehn twists $T_i$ on $\Sigma_{[m/2]}$ of the form $T_i(x,y)=(x+r(y-i\pi),y)$ for $i$ odd and $T_j(x,y)=(x,y-r(x-(j+1)\pi))$ for $j$ even.  
\end{lem}

 {
Penner's construction of pseudo-Anosov diffeomorphisms in Theorem \ref{thm:PA} shows that the map $\mathcal{T}$ is in a pseudo-Anosov mapping class on $\Sigma_{[m/2]}$. By Theorem \ref{thm:anos}, the intersection number between any two curves grow exponentially, which forces $\rho(\mathcal{T}^*)>1$, where $\mathcal{T}^*$ is the induced action of $\mathcal{T}$ on $H^1(\Sigma_{[m/2]},\R)$, and $\rho(\cdot)$ is the largest absolute value of the eigenvalues of a matrix. By Yomdin's inequality (see Theorem \ref{yom}), this shows that $h_{top}(\mathcal{T}_{[m/2]},S)\ge \log\rho(\mathcal{T}^*)>0$. So we have $$ h_{top}(\tau, A_m^n)\ge h_{top}(\tau, X)\ge h_{top}(\mathcal{T},\hat \Sigma_{[m/2]}\})\ge h_{top}(\mathcal{T},\Sigma_{[m/2]})>0 $$ which proves Theorem \ref{mtwist}.} 
\end{proof}


\section{Hyperbolicity of the subsystem}\label{pf}
In this section, we shall prove the hyperbolicity of the composition $T_1^{k_1}\ldots T_m^{k_m}$, which by Pesin theory gives not only positive topological entropy, but also stable and unstable manifolds as well as standard consequences from ergodic theory. We can apply this result for Theorem \ref{foli} on the existence of stable/unstable laminations of Dehn twists and for an alternative proof of Theorems \ref{2twists} and \ref{mtwist}. For a preliminary on Pesin theory, we refer readers to Appendix \ref{apd:etp}. In the following we set $\mathcal{T}=T_1^{k_1}\ldots T_m^{k_m}$.

\subsection{Hyperbolicity of the subsystem and existance of Lagrangian submanifolds}
\begin{thm}\label{hyp} Let $X$ be the punctured surface and let $T_i$ be the Dehn twists as in Lemma \ref{mLTM}, then for $\mathcal{T}=T_1^{k_1}\ldots T_m^{k_m} $, $k_ik_{i+1}<0$, the Lyapunov exponent $\chi^+(x)$ $($see Definition \ref{DefLyap}$)$ of $\mathcal{T}$ is positive for almost every $x\in X$.
\end{thm}

\begin{rmk}
The subsystem for $m=2$, i.e. $(\tau_1^k\tau_2^\ell,X)\simeq (T_1^kT_2^\ell ,\mathbb{T}^2\setminus\{*\})$, is  similar to  the linked twist map on the torus studied in literature \cite{bur}.  However, there is a difference between them. Indeed, the linked twist map considered there is defined by composing $T_1, T_2$ with $r^\prime(t)>0,-\epsilon<t<\epsilon$ for all $\epsilon>0$. However, in our case, we have $r^\prime(0)=0$ and $r^\prime (t)>0, \forall \ 0<|t|<\epsilon$. When $k\ell<0$, the corresponding map is called co-rotating linked twist map in \cite{bur} and $k\ell>0$ counter-rotating. Technically, the latter case is much harder, so we focus only on the co-rotating case in this paper.

\end{rmk}

In what follows we shall extend the ideas to $m>2$ and study the multi-linked twist map $(X,\mathcal{T})$.

We denote by $X_\epsilon$ the $\epsilon$-neighborhood of $\cup\gamma_i$ in $X$, which we identify with the $\epsilon$-neighborhood of $\cup\gamma_i$ in $P_\psi(\gamma_1,\ldots,\gamma_m)$. We denote by $A_i$ the $\epsilon$-neighborhood of $\gamma_{2i-1}$ and by $B_i$ the $\epsilon$-neighborhood of $\gamma_{2i}$.

For $0\le \delta<\epsilon<1$, we take
$$Y^1_\delta=\{(x,y)\in X\ |\  \exists\ j\in \N\text{ even},\ s.t.\ \delta<|x-j\pi|<\epsilon-\delta;\  \delta<|y-(j-1)\pi|<\epsilon-\delta\}$$ to be the intersections of $\delta$-interior of $X_\epsilon$ minus the $\delta$-neighborhood of the median lines. We also take
$$Y^2_\delta=\{(x,y)\in X\ |\ \exists\ j\in \N \text{ odd},\ s.t.\ \delta<|x-j\pi|<\epsilon-\delta;\ \delta<|y-j\pi|<\epsilon-\delta\}$$ to be the intersections of $\delta$-interior of $X_\epsilon$ minus the $\delta$-neighborhood of the median lines as shown in Figure \ref{Ydelta}.

\begin{figure}[ht]
	\begin{center}
		\includegraphics[width=0.6\textwidth]{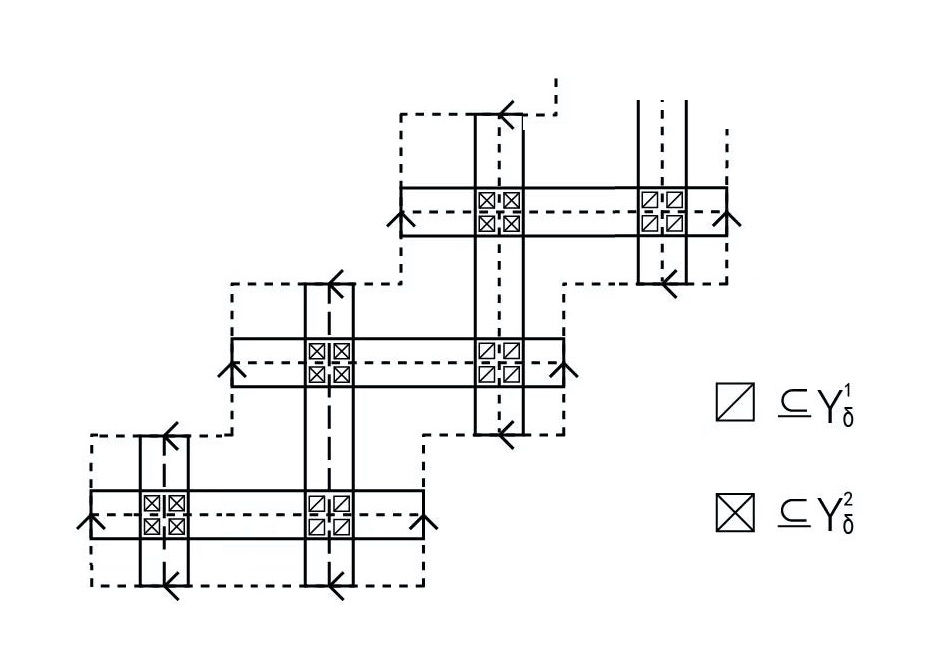}
	\end{center}
		\caption{$Y_\delta^1$ and $Y_\delta^2$ in $X$}\label{Ydelta}
\end{figure}

We denote the set of points returning sufficiently often to $Y^1_\delta$ by $$A_\delta=\{p\in X\ |\ \liminf_{N\to \infty}\frac{1}{N}\sum_{k=1}^N 1_{Y^1_\delta}(\mathcal{T}^k p)\ge \delta \},$$ and the set of points returning sufficiently often to $Y^2_\delta$ by $$B_\delta=\{p\in X\ |\ \liminf_{N\to \infty}\frac{1}{N}\sum_{k=1}^N 1_{Y^2_\delta}(\mathcal{T}^k p)\ge \delta \}.$$

We first show that $\mathcal{T}$ is hyperbolic on $A_\delta$ and $B_\delta$.
\begin{lem}\label{lya} Given $\delta>0$, there exists $c_\delta$ such that for any $p\in A_\delta\cup B_\delta$, the Lyapunov exponent $\chi^+(p)\ge c_\delta$.
\end{lem}


\begin{lem}\label{fullm}Let $\mu$ be the Lebesgue probability measure on $X$, we have
$\mu(\cup_{\delta>0}(A_\delta\cup B_\delta))=1$.
\end{lem}
We postpone the proofs to Section \ref{sub:lem}.

Theorem \ref{hyp} comes as a direct corollary of the above lemmata. Furthermore, these lemmata give an alternative proof of Theorem \ref{mtwist}.
\begin{proof}[Alternative proof of Theorem \ref{mtwist}]
By Lemma~\ref{fullm} we have $\mu(\cup_{\delta>0}(A_\delta\cup B_\delta))=1$, thus there exists $\delta_0>0$ for which $\mu(A_{\delta_0}\cup B_{\delta_0})>0$. By Lemma \ref{lya}, there exists $c>0$ such that $\chi^+(x)>c$ for any $x\in A_{\delta_0}\cup B_{\delta_0}$. Thus by Pesin entropy formula we have $h_{\mu}(\mathcal{T})\ge c\mu( A_{\delta_0}\cup B_{\delta_0})>0$, so by the variational principle we have $$h_{top}(\tau)\ge h_{top}(\tau\mid X)= h_{top}(\mathcal{T} )\ge h_{\mu}(\mathcal{T})>0.$$
\end{proof}

\subsection{Local stable/unstable Lagrangian manifolds}\label{SFoliation}

As a corollary of the non-uniform hyperbolicity of composite Dehn twists on surfaces established in Lemma \ref{lya} and Lemma \ref{fullm}, in this section, we  prove the existence of local stable/unstable Lagrangian manifolds using Pesin's stable manifold theorem and a rotation construction. We start with the definition of rotation map that brings curves from $X=A_m^1$ back to Lagrangian submanifolds in $A_m^n.$ In this section, we use the standard identification of $A_m^n$ with the plumbing space of $m$ cotangent spaces of $n$-dimensional spheres $P_\psi(S_1,S_2,...,S_n)$ and perform the rotation map on the plumbing space.

\begin{proof}[Proof of Theorem \ref{foli}]
The group $\mathrm{SO}(n)$ acts on the plumbing space via its standard action on the sphere $\mathbb{S}^{n-1}$, more precisely, it is given by $\mathrm{SO}(n)\times P_\psi(S_1,\ldots, S_n)\to P_\psi(S_1,\ldots, S_n): (g,p)\mapsto \mathcal{R}_{g(1,0,\ldots,0)}p$. We take $N$ to be the $\epsilon$-neighborhood of $\cup S_i$ in the plumbing space.  { Let $X_\epsilon$ be the $\epsilon$-neighborhood of $\cup_{i=1}^m \gamma_i$ in $P_\psi(\gamma_1,\gamma_2,\cdots,\gamma_m)$},

By Theorem \ref{hyp}, we see from Pesin's stable manifold theorem (see~Theorem~\ref{thm:Pesin}) that  {$\tau$ restricted to $X_\epsilon$} admits stable and unstable curves $W^s(p)$ and $W^u(p)$ for almost every $p\in X_\epsilon$. Let $\Lambda$ be the set of points $p\in X_\epsilon$ that is hyperbolic under $\tau$. Since the stable curves are given by the Dehn twists which is commutative with the rotation map $\mathcal{R}_\theta$, we see that $\Lambda$ is invariant under the map $\bar{\mathcal{R}}$ and the family of stable curves $W^s(p)$ is  {invariant} under $\bar{\mathcal{R}}$, i.e. $\{W^s(p)\ |\ p\in \Lambda\}=\{\bar{\mathcal{R}}(W^s(p))\ |\ p\in \Lambda\}$. Thus by Lemma \ref{lem:smlg} we define a family of invariant Lagrangian submanifolds $\cF^s=\{L_{W^s(p)}\ |\ p\in \Lambda\}$. Similarly, the family of unstable curves $W^u(p)$ in $X_\epsilon$ given by the inverse map also gives us a family of invariant Lagrangian submanifolds  {$\cF^u=\{L_{W^u(p)}\ |\ p\in \Lambda\}$.}
\end{proof}

\subsection{Proofs of  Lemmata \ref{lya}--\ref{fullm}}\label{sub:lem}
In this section, we give proofs of Lemmata \ref{lya}--\ref{fullm}.

\begin{proof}[Proof of Lemma \ref{lya}]
In this proof we take $\prod$ to be the ordered product which is read from right to left, and
we denote by $1_\N$ the indicator function that takes value $1$ on $\N$ and 0 elsewhere.
	For any $n\in \N$ and any $p\in A_\delta\cup B_\delta$, we have $$D\mathcal{T}^n(p)=\prod_{i=1}^n\left(\begin{pmatrix}1 &t_1^i\\0&1\end{pmatrix}\begin{pmatrix}1 &0\\t_2^i&1\end{pmatrix}...\begin{pmatrix}1 &1_\N((m+1)/2)\cdot t_m^i\\1_\N(m/2)\cdot t_m^i &1\end{pmatrix}\right)$$
where
	$$t^i_j=\left\{\begin{array}{l r}k_jr^\prime(y(\prod_{\ell=m}^{j+1}T_\ell^{k_\ell}\mathcal{T}^{i-1}p)-j\pi) &j\text{ odd}\\ k_jr^\prime(x(\prod_{\ell=m}^{j+1}T_\ell^{k_\ell}\mathcal{T}^{i-1}p)-(j+1)\pi) &j\text{ even}\end{array}\right. .$$ Thus we have $t^i_j\ge 0$ for any $1\le j\le m,1\le i\le n$.
	
	For $p\in A_\delta\cup B_\delta$, assume without loss of generality that $p\in A_\delta$, then for any $n$ sufficiently large, we have $\#\{\{\mathcal{T}^i(p)\ |\ 0\le i\le n\}\cap A_\delta\}\ge \delta n/2$. We number the iterations when $p$ hits $A_\delta$ by $\{n_k\}_{k\in \N}=\{n\in \N\ |\  \mathcal{T}^n(p)\in A_\delta\}$ where $n_i< n_{i+1}$. Then for any $k\in \N$, $\mathcal{T}^{n_k}(p)$ and $\mathcal{T}^{n_{k+1}}(p)$ are both in $A_\delta$. Suppose $\mathcal{T}^{n_k}(p)\in A_{j(n_k)}\cap B_{j(n_k)}$, $\mathcal{T}^{n_{k+1}}(p)\in A_{j(n_{k+1})}\cap B_{j(n_{k+1})}$ then we claim that $$t_{2j(n_k)}^{n_k}>c_1\text{ and } t_{2j(n_{k+1})-1}^{n_{k+1}-1}>c_1,$$  where the constant $c_1>0$ is taken so that $r^\prime(t)>c_1,\forall\ \delta \le |t|\le \epsilon-\delta$ and $c_1$ depends only on $r$ and $\delta$.
	
	To prove the claim we just note that if $q=\mathcal{T}^{n_k}(p)$ starts from $A_\delta\cap A_{j(n_k)}\cap B_{j(n_k)}$, $T_j$ acts by identity on $q$ for $j> 2j(n_k)$ and so $t_{2j(n_k)}^{n_k}>k_{2j(n_k)}c_1$. Similarly, whenever $q=\mathcal{T}^{n_{k+1}}$ lands in $A_\delta\cap A_{j(n_{k+1})}\cap B_{j(n_{k+1})}$, we see that $T_j^{-1}(q)=q$ for $j<2j(n_{k+1})-1$ which ensures that $t_{2j(n_{k+1})-1}^{n_{k+1}-1}>k_{2j(n_{k+1})-1}c_1$.

Thus we have that for $p\in A_\delta$, it holds that $\#\{1\le i\le n \ | \  t^i_j>c_1 \text{ for some } j\}\ge \delta n/2$ for $n$ sufficiently large. The case $p\in B_\delta$ can be proved in the same way.
	
	We next introduce a partial ordering $\le$ on $\R^2$ given by $u=(u_1,u_2)\le v=(v_1,v_2)$ if $u_1\le v_1,u_2\le v_2$. We note that $T_j^{k_i}u\le T_j^{k_j}v$ if $u\le v$ and $T_j^{k_i}u\ge u$ if $u\ge 0$. Also, for $u\ge 0$, $t^i_j>c_1$, we have $$\begin{pmatrix}1 &1_\N((j+1)/2) \cdot t_j^i\\1_\N(j/2) \cdot t_j^i &1\end{pmatrix} u\ge\begin{pmatrix}1 &1_\N((j+1)/2)\cdot c_1\\1_\N(j/2)\cdot c_1&1\end{pmatrix} u.$$ Thus by our claim above, for $u$ in the first quadrant, we have $$D\mathcal{T}^n(p)[u]\ge \left(\begin{pmatrix}1 &0\\c_1&1\end{pmatrix}\begin{pmatrix}1 &c_1\\0&1\end{pmatrix}\right)^{[\delta n/2]}u=\begin{pmatrix}1 &c_1\\c_1&1+c_1^2\end{pmatrix}^{[\delta n/2]} u.$$ This implies that
	$$\big\|D\mathcal{T}^n(p)[u]\big\|\ge \left\|\begin{pmatrix}1 &c_1\\c_1&1+c_1^2\end{pmatrix}^{[\delta n/2]} u\right\|\ge c_2 \lambda^{[\delta n/2]}\|u\|$$ for a constant $c_2$ depending on $u$ and $c_1$, where $\lambda$ is the larger eigenvalue of $\begin{pmatrix}1 &c_1\\c_1&1+c_1^2\end{pmatrix}$. Thus we have $\chi^+(p,u)\ge \lambda^{\delta/2}=c_\delta$ for any $u\ge 0$. So we get our desired result $\chi^+(p)\ge c_\delta>0$.
\end{proof}

The following lemma by Burton and Easton shows that almost every point that hits a subset of $X$ must hit the subset sufficiently often (which can be viewed as a generalization of Poincar\'e recurrence theorem).
\begin{lem}[Burton and Easton, \cite{bur}]\label{rec}Let $(F,X,\mu)$ be a measure preserving dynamical system, and $V$ be a subset of $X$. Suppose $$P(V)=\{p\in X\ |\  \exists\ N\in \N, \ s.t.\ F^N(p)\in V\}$$ is the set of points that hits $V$, $$Z(V)=\{p\in X\ |\  \liminf_{N\to \infty}\frac{1}{N}\sum_{k=0}^N 1_{V}(F^N(p))=0\}$$ is the set of points that do not hit $V$ at a positive average rate, then we have $\mu(P(V)\cap Z(V))=0$.
\end{lem}
\begin{proof}
	By definition $Z(V)$ is an $F$-invariant set of $X$. Thus we have $$\int_{Z(V)}1_{V}(x)d\mu(x)=\int_{Z(V)}  1_V(F^N(x))d\mu(x)=$$$$\int_{Z(V)}\liminf_{N\to \infty}\frac{1}{N}\sum_{k=0}^N 1_{V}(F^N(x))d\mu(x)=0.$$ So we get $\mu(V\cap Z(V))=0$. Furthermore, we have $\mathcal{T}^{-k}(V\cap Z(V))=\mathcal{T}^{-k}(V)\cap Z(V)$ for any $k\in \N$. This implies that $\mu(P(V)\cap Z(V))=\mu(\cup_{k\in \N}(\mathcal{T}^{-k}(V)\cap Z(V)))=0$.
\end{proof}

We next show that indeed almost every point in $X$ hits $Y^1_0\cup Y^2_0$ eventually.
\begin{lem}\label{hit} For Lebesgue almost every point $p$ in $X$, there exists $N\in \N$ such that $\mathcal{T}^N(p)\in Y^1_0\cup Y^2_0$.
\end{lem}
\begin{proof}
	We notice that $$Y^1_0\cup Y^2_0=\cup_{i\neq j} (Y_i\cap Y_j)-\{(x,y)\ |\ x\in \N\text{ or }y\in\N \}.$$ First we note that the set of points ever hitting $\{(x,y)\ |\ x\in \N\text{ or }y\in\N \}$ has zero Lebesgue measure since $vol(\{(x,y)\ |\ x\in \N\text{ or }y\in\N \})=0$. Thus we only need to consider points that never hits $\cup_{i\neq j} (Y_i\cap Y_j)$.
	
	If $\mathcal{T}^n(p)\notin \cup_{i\neq j} (Y_i\cap Y_j), \forall n\in \N$, then suppose $p\in Y_i$, then any iteration $\mathcal{T}^n(p)$ is in $Y_i\backslash \cup_{j\neq i}Y_j$. $\mathcal{T}$ restricted to $Y_i\backslash \cup_{j\neq i}Y_j$ is simply $T_i^{k_i}$ which is a single Dehn twist. Thus for $\mathcal{T}^n(p)\in Y_i\backslash \cup_{j\neq i}Y_j,\forall n$, the point $p$ has to be periodic with $|\mathcal{T}(p)-p|>\epsilon$. Thus we see that such points must lie on finitely many 1-dimensional segments, so they make up a measure zero subset of $X$.
\end{proof}

Now we finish the proof the the lemmas.
\begin{proof}[Proof of Lemma \ref{fullm}]
	We only need to prove that $\mu(\cup_{n\in \N}(A_{\frac{1}{n}}\cup B_{\frac{1}{n}}))=1$.
	
	We follow the notations in Lemma \ref{hit}. By Lemma \ref{hit}, since $\cup_{\delta>0}Y^i_\delta=Y^i_0,i=1,2$, for almost every $p\in X$, there exists $n>0,N\in \N$ such that $\mathcal{T}^N(p)\in Y^1_{\frac{1}{n}}\cup Y^2_{\frac{1}{n}}$, i.e. $$\mu \big( \cup_{n\in \N}( P(Y^1_{\frac{1}{n}}\big)\cup P(Y^2_{\frac{1}{n}})))=1.$$ By Lemma \ref{rec}, we have $\mu(P(Y^i_\delta)\cap Z(Y^i_\delta))=0,\forall \delta$ which implies
	$$\mu\left(\bigcup_{m\in \N}\left(P(Y^i_{\frac{1}{n}})\bigcap \left\{p\in X\ |\  \liminf_{N\to \infty}\frac{1}{N}\sum_{k=0}^N 1_{Y^i_\frac{1}{n}}(F^N(p))>\frac{1}{m} \right\}\right)\right)=\mu(P(Y^i_{\frac{1}{n}})),$$ for $i=1,2.$
	Now since $Y^i_{\frac{1}{n}}\subset Y^i_{\frac{1}{m}}$ for $n<m$, the last equation implies that $$\mu\big(P(Y^1_{\frac{1}{n}})\cap(\cup_{m\in \N} A_{\frac{1}{m}})\big)=\mu(P(Y^1_{\frac{1}{n}})),\quad \mu(P(Y^2_{\frac{1}{n}})\cap(\cup_{m\in \N} B_{\frac{1}{m}}))=\mu(P(Y^2_{\frac{1}{n}}))$$ for any $n\in \N$. Combining this with $P(Y^i_{\frac{1}{n}})\subset  P(Y^i_{\frac{1}{m}})$ for any $n<m$ and $\mu( \cup_{n\in \N}( P(Y^1_{\frac{1}{n}})\cup P(Y^2_{\frac{1}{n}})))=1$, we have

	\begin{equation*}
		\begin{aligned}
			\mu\big(\cup_{n\in \N} (A_{\frac{1}{n}}\cup B_{\frac{1}{n}})\big)&= \mu\left(\big(\cup_{m\in \N} (A_{\frac{1}{m}}\cup B_{\frac{1}{m}})\big)\bigcap \big(\cup_{n\in \N} (P(Y^1_{\frac{1}{n}})\cup P(Y^2_{\frac{1}{n}}))\big)\right)\\
			&\ge \mu\left(\cup_{n\in \N}\big(P(Y^1_{\frac{1}{n}})\cap(\cup_{m\in \N} A_{\frac{1}{m}})\big)\bigcup \big(\cup_{n\in \N}(P(Y^2_{\frac{1}{n}})\cap(\cup_{m\in \N} B_{\frac{1}{m}}))\big)\right)\\
			&=\mu\left(\cup_{n\in \N}\big(P(Y^1_{\frac{1}{n}})\cup P(Y^2_{\frac{1}{n}})\big)\right)=1.		
		\end{aligned}
	\end{equation*}
	This completes the proof.
\end{proof}


\section{The Growth of Floer cohomology groups}\label{sec:invo}

In this section, we explore the symplectic aspect of the composite Dehn twists. As we have explained in the introduction, we are interested in the growth of Floer cohomology group. For a brief introduction to Lagrangian Floer theory, we refer readers to Appendix \ref{apd:fl}.

In order to estimate the growth of Lagrangian Floer cohomology groups,  we shall employ a result due to Khovanov and Seidel~\cite{KS}. 

Let $(M,\omega=d\lambda)$ be an exact symplectic manifold with contact boundary $(\partial M, \lambda)$. We assume in addition that $M$ admits an involution, i.e., $\iota:M\to M$ with $\iota\circ\iota=Id$ and $\iota^*\omega=\omega$ and $\iota^*\lambda=\lambda$. Clearly, the fixed point set $M^\iota$ is a symplectic submanifold of $M$. Moreover, when $L$ is a Lagrangian submanifold of $M$, its fixed part $L^\iota=L\cap M^\iota$ is again a Lagrangian submanifold of $M^\iota$.

\begin{thm}[{Khovanov and Seidel~\cite[Proposition~5.15]{KS}}]\label{thm:KS}
	Let $(S_1,S_2)$ be a pair of closed $\lambda$-exact Lagrangian submanifolds of $(M,\omega)$ with $\iota(S_i)=S_i,i=1,2$. Suppose that
	\begin{itemize}
		\item[(C1)]the intersection $C=S_1\cap S_2$ is clean\footnote{Here ``clean intersection" means that $C=S_1\cap S_2$ is a smooth manifold and $TC=TS_1|_C\cap TS_2|_C$.}, and there is an $\iota$-invariant Morse function $f$ on $C$ such that its critical points are precisely the points of $C^\iota=C\cap M^\iota$;
		\item [(C2)] there is no continuous map $u:\ [0,1]\times[0,1]\to M^\iota$ such that $u(0,t)=x$ and $u(1,t)=y$ for all $t$, and $u(s,0)\in S_1$ and $u(s,1)\in S_2$ for all $s$, where $x$ and $y$ are two different points of $S_1\cap S_2$.
	\end{itemize}
	Then $\rk\; \HF(S_1,S_2)=\rk\; H^*(C,\mathbb{Z}/2)=|C^\iota|$.
	
\end{thm}
The proof of this theorem is based on equivariant transversality and a symmetry argument, for which we refer readers to \cite{KS} for details.
It is important to note that the involution  plays the  role of reducing the dimension of the symplectic manifold and Lagrangian submanifold for the purpose of calculating the rank of Lagrangian Floer cohomology, similar to what we did in the last section when proving positive topological entropy.

\begin{rmk}
As mentioned before, the $A^n_m$-configuration can be seen as a completion of plumbing domain $P_\psi(S_1,\ldots,S_m)$ of unit disk cotangent bundles of spheres. According to Appendix~\ref{apd:fl}, for Lagrangian spheres $L_1,L_2$ in $A^n_m$ one can define the Lagrangian Floer cohomology as 
\[
\HF^{A^n_m}(L_1,L_2)=\HF(L_1,L_2;H,J)
\]
for any Hamiltonian function $H\in C_c^\infty([0,1]\times A^n_m)$ such that $\phi^1_H(L_1)$ intersects $L_2$ transversely and any generic compatible almost complex structure $J$ on $A^n_m$ which is of contact type at infinity. 
Since those Lagrangian spheres involved in Theorems~\ref{thm:class} and \ref{thm:symrate} (resp. \ref{thm:symrate2}) are contained in the interior of a Liouville domain $W$ inside plumbing domain  $P_\psi(S_1,S_2)$ (resp. $P_\psi(S_1,\ldots,S_m)$), in the above definition of $\HF^{A^n_m}(L_1,L_2)$ one can choose a Hamiltonian function $H$ on $A^n_m$ with $\cup_{t\in [0,1]}{\rm supp}(H_t)\subset W\setminus\partial W$. As a result, we have
\[
\HF^{A^n_m}(L_1,L_2)=\HF^{\widehat{W}}(L_1,L_2)
\]
where $\widehat{W}$ is a symplectic completion of $W$.  For this reason, we do not distinguish the notations between  $\HF^{A^n_m}(L_1,L_2)$ and $\HF^{\widehat{W}}(L_1,L_2)$, and use $\HF(L_1,L_2)$ to stand for them. Note that by our choice of Liouville domain, $W$ admits a natural involution as defined in Section~\ref{subsec:inv}, one can use Theorem~\ref{thm:KS} to compute $\HF(L_1,L_2)$.

\end{rmk}

\subsection{Proof of Theorem~\ref{thm:symrate}}\label{subsec:proof}
In this section, we give the proof of Theorem~\ref{thm:symrate}. In order to apply Theorem \ref{thm:KS}, we need to verify the conditions (C1) and (C2). In the following, for any $k,\ell\in\Z$ we denote $\tau=\tau_1^k\tau_2^\ell$ the composition of symplectic Dehn twists.

\begin{lem}\label{LmClean}
	 {Let $\gamma_1,\gamma_2$ be the set fixed by the involution $\iota$ in the Lagrangian spheres $S_1$ and $S_2$.} If $\gamma_1$ and $\tau^n(\gamma_2)$ intersect transversely in $P_\psi(\gamma_1,\gamma_2)$, then  $S_1$ and $\tau^n(S_2)$ satisfy condition~$(C1)$ of Theorem~\ref{thm:KS}. 	
\end{lem}


We shall give a proof of this lemma in Section~\ref{subsec:rotmap}. The advantage of  {this proof} is that it allows us to perturb $S_i$ in a smooth exact Lagrangian isotopic way (hence in a Hamiltonian isotopic one) via deforming the generating curves so that the intersections of the resulting closed curves $\gamma_i'$ and $\tau^n(\gamma_j)$ become minimal, see~Section~\ref{SIntersection}.

To achieve the condition~(C2) in Theorem~\ref{thm:KS},  we use the following concept which can be found in the book~\cite{FM}.

\begin{defi}[Geometric intersection number]
	
	Let $\alpha$ and $\beta$ be two free homotopy classes of simple closed curves in a surface $S$. We call the minimal number of intersection points between a representative curve $a$ in the class $\alpha$ and a representative curve $b$ in the class $\beta$ the \emph{geometric intersection number} between $\alpha$ and $\beta$, which we denote by $I(\alpha,\beta)$.
\end{defi}

\begin{lem}\label{LmIntersection}
If  $\gamma_1$ and $\tau^n(\gamma_2)$ are not admissibly isotopic,   we have the following equality
	\begin{equation}\label{e:hf=i}
		\rk\; \HF \big(S_1,\tau^n(S_2)\big)=I\big([\gamma_1],[\tau^n(\gamma_2)]\big).
	\end{equation}
	
\end{lem}
We postpone the proof of the lemma to Section \ref{SIntersection} and complete the proof of Theorem~\ref{thm:symrate}.
\begin{proof}[Proof of Theorem~\ref{thm:symrate}]
	
	Since symplectic Dehn twists $\tau_1$ and $\tau_2$ are supported in a compact neighborhood of the zero section of
	$P_\psi(S_1,S_2)$, as in Section~\ref{topent} we  identify the action of $\tau=\tau^k_2\tau^\ell_1$ on an invariant submanifold $X_\epsilon\subseteq P_\psi(\gamma_1,\gamma_2)$ with the action of $\mathcal{T}=T_2^kT_1^\ell$ on an open subset $X$ of $2$-torus $\mathbb{T}^2$. And the curves $\gamma_1$ and $\gamma_2$ are identified with $\beta_1=\{(s,0)\in \mathbb{T}^2\}$ and $\beta_2=\{(0,t)\in \mathbb{T}^2\}$ respectively.  {We note that $\mathcal{T}$ can be smoothly extended to $\mathbb T^2$ by taking $\mathcal{T}$ to act as identity outside $X$, and the intersection number restricted to $X$ is greater than or equal to the intersection number on the whole torus. This shows that}
	\begin{equation}\label{e:insec}
		 I\big([\gamma_1],[\tau^n(\gamma_2)]\big)=I_{X}\big([\beta_1],[\mathcal{T}^n(\beta_2)]\big)\geq I_{\mathbb{T}^2}\big([\beta_1],[\mathcal{T}^n(\beta_2)]\big).
	\end{equation}
	So by Lemma \ref{LmIntersection}, we have $\rk\; \HF \big(S_1,\tau^n(S_2)\big)\ge I_{\mathbb{T}^2}\big([\beta_1],[\mathcal{T}^n(\beta_2)]\big).$ To prove Theorem~\ref{thm:symrate} we need to estimate the growth rate of the geometric intersection number $I_{\mathbb{T}^2}\big([\beta_1],[\mathcal{T}^n(\beta_2)])$ as $n\to \infty$.
	
	It is well-known that the mapping class group $\hbox{Mod}(\mathbb{T}^2)$ of $\mathbb{T}^2$ is $\mathrm{PSL}_2\Z$, and the nontrivial free homotopy classes of oriented simple closed curves in $\mathbb{T}^2$ are in bijective correspondence with primitive elements of $\mathbb{Z}^2$. Moreover, the geometric intersection number of two such homotopy classes $(a,b),(c,d)\in \mathbb{Z}^2$ can be computed as
	\begin{equation}\label{e:torus}
		I_{\mathbb{T}^2}\big((a,b),(c,d)\big)=|ad-bc|,
	\end{equation}
	see, for instance, \cite[Section~1.2.3]{FM}.
	From Lemma~\ref{LmLTM} we see that $T_1$ and $T_2$, representing  {elements} of $\hbox{Mod}(\mathbb{T}^2)$, have the forms
$
		\left(\begin{array}{cc}
			1 &    1    \\
			0&    1
		\end{array}\right)\quad\hbox{and}\quad
		\left(\begin{array}{cc}
			1 &    0    \\
			-1&    1
		\end{array}\right)
$
	respectively. For $k,\ell\in\mathbb{N}$, the mapping class of $\mathcal{T}=T_2^{k}T_1^\ell$ is given by
	$
	\left(\begin{array}{cc}
		1 -k\ell&    l    \\
		-k&    1
	\end{array}\right)
	$. If $k\ell\neq 0,1,2,3,4$, then   this matrix has a similar diagonalizable matrix $
	\left(\begin{array}{cc}
		\lambda&    0   \\
		0&    \lambda^{-1}
	\end{array}\right)
	$ with the eigenvalue $|\lambda|>1$. In this case, using (\ref{e:torus}) an easy calculation shows that
	\begin{equation}\label{e:numgrowth}
		 \lim_{n\to\infty}\sqrt[n]{I_{\mathbb{T}^2}\big([\beta_1],[\mathcal{T}^n(\beta_2)]\big)}=|\lambda|.
	\end{equation}
	Clearly, for sufficiently large $n$, $\beta_1$ and $\mathcal{T}^n(\beta_2)$ are not admissibly isotopic, and so for $\gamma_1$ and $\tau^n(\gamma_2)$. Then it follows from (\ref{e:hf=i}), (\ref{e:insec}) and (\ref{e:numgrowth}) that
	$$\Gamma(\tau,S_1,S_2)=\liminf_{n\to\infty}\frac{\log\mathrm{rank}\; \HF(S_1,\phi^n(S_2))}{n}\geq \log|\lambda|>0.$$
	The proof of the case of $\tau=\tau_1^k\tau_2^\ell$ is similar, and so we omit it here. This completes the proof of statement~(1).

\noindent {\it Proof of statement (2)}. When $k\ell=0,1,2,3,4$, then we have that the matrix $\left(\begin{array}{cc}
		1 -k\ell&    l    \\
		-k&    1
	\end{array}\right)$ is a periodic or reducible mapping class, then the  growth of Floer cohomology groups is zero.

	\noindent {\it Proof of statement (3)}.
	The claim that $\Gamma(\tau_i^2,S_j,S_j)=0$, $i,j\in\{1,2\}$ with $i\neq j$ follows from the following result immediately.
	\begin{lem}[{\cite[Proposition~4.7]{Ja}}]
		Let $(M^{2n},\omega)$ be a connected Liouville domain  $2c_1(M,\omega)=0$. If $M$ contains an $A_2^n$-configuration of Lagrangian spheres $(S_1,S_2)$, then for any $k\in\mathbb{N}\cup\{0\}$, it holds that
		$$\rk\;\HF(S_i,\tau_j^{2k}(S_i))=2k,\;i,j\in\{1,2\}\;\hbox{with}\;i\neq j.$$
	\end{lem}
	The proof of this lemma is essentially due to Seidel's long exact sequence for Floer cohomology of Dehn twists~\cite[Theorem~1]{Se}, see also Keating~\cite[Section~6]{Ke}. The plumbing space $P(S_1,S_2)$ is a symplectization of Liouville domain with contact-type boundary which satisfies the conditions of the above lemma, see, for instance, \cite[Section~2.3]{Ab}.
	\end{proof}

\subsection{Proof of Theorem~\ref{thm:symrate2}}

The proof of Theorem~\ref{thm:symrate2} is similar to that of Theorem~\ref{thm:symrate}. Let $P_\psi(S_1,S_2,\ldots,S_m)$ and $P_\psi(\gamma_1,\gamma_2,\ldots,\gamma_m)$  be the plumbing domains as given in Section~\ref{sec:subL}.
We need to verify the conditions (C1) and (C2) of Theorem~\ref{thm:KS}. Similar to Lemma \ref{LmIntersection}, we have the following lemma whose proof is postponed to Section \ref{proofmultint}. 

\begin{lem}\label{multiintersect}
 If the composition  $\tau=\tau_1^{k_1}\cdots\tau_m^{k_m}$ of symplectic Dehn twists along spheres $S_i$ satisfies that $\tau^n(\gamma_j)$ is not admissibly isotopic to $\gamma_i$, then we have the equality
\begin{equation}\label{e:mhf=i}
		\rk\; \HF\big(S_i,\tau^n(S_j)\big)=I\big([\gamma_i],[\tau^n(\gamma_j)]\big).
	\end{equation}
\end{lem}
\begin{rmk}
	In the case of dimension two, Dimitrov, Haiden, Katzarkov, and Kontsevich \cite{DHKK} showed that for two non-homotopic curves $\alpha$ and $\beta$ on a closed oriented surface $\Sigma_g$ of genus $g$, we have $\rk \; \HF(\alpha, \tau^n(\beta)) = I([\alpha], [\tau^n(\beta)])$. This result provides a connection between the Thurston classification and a classification of symplectic mapping classes based on the growth rate of Floer cohomology.
	
	However, in higher dimensions, finding analogues of this result has remained a long-standing challenge, as mentioned in \cite[Section 3.2.2]{Sm}. Our Lemma \ref{e:mhf=i} presents a non-trivial case in this direction, making the first attempt towards understanding the classification of symplectic mapping classes and their growth rates in higher dimensions.
	\end{rmk}

To complete the proof of Theorem~\ref{thm:symrate2} the remaining task is to show that $I([\gamma_i],[\tau^n(\gamma_j)])$ grows exponentially as $n\to\infty$.
To do this, we apply some classical results from the construction of pseudo-Anosovs on surfaces, which we stated in Theorems \ref{thm:PA} and \ref{thm:anos}.

\begin{proof}[Proof of Theorem~\ref{thm:symrate2}]
    Since each symplectic Dehn twist $\tau_i$ is supported in a compact neighborhood of the zero section of $T^*S_i$, one can identify the action of $\tau=\tau_1^{k_1}\tau_2^{k_2}\cdots\tau_m^{k_m}$ on an invariant Liouville domain $X_\epsilon\subseteq P_\psi(\gamma_1,\gamma_2,\ldots,\gamma_m)$ with the action of $\mathcal{T}=T_1^{k_1}T_2^{k_2}\cdots T_m^{k_m}$ on an open subset $X$ of a closed surface of genus $g$.  More precisely, for $m=2k,2k+1$ one can embed curves $\gamma_1,\ldots,\gamma_m$ into a surface $S$ of genus $g=k$ with images $\beta_1,\ldots,\beta_m$ as illustrated in Picture~\ref{fig:surface} (when $m=2k$, we do not need $\beta_{2k+1}$), and $X_\epsilon$ is identified with the union of open neighborhoods of simple closed curves $\beta_1,\ldots,\beta_m$ in $S$, and each $\tau_i$ corresponds to the Dehn twist $T_i$ along $\beta_i$. So we have
	\begin{equation}\label{e:minsec}
		 I\big([\gamma_i],[\tau^n(\gamma_j)]\big)=I_{X}\big([\beta_i],[\mathcal{T}^n(\beta_j)]\big)\geq I_{S}\big([\beta_i],[\mathcal{T}^n(\beta_j)]\big).
	\end{equation}
	Notice that by our construction the collection of curves $\beta_1,\ldots,\beta_m$ fills $S$, and that $\beta_i$ and $\beta_{i+2}$, $i=1,\ldots, m-2$ are disjoint. By Theorem~\ref{thm:PA},  if either $k_i>0$ for $i$ odd and $k_i<0$ for $k$ even, or $k_i>0$ for $i$ even and $k_i<0$ for $k$ odd, $\mathcal{T}$ is pseudo-Anosov. Then by Theorem~\ref{thm:anos}, we have
	 $\lim_{n\to\infty}\sqrt[n]{I_{S}\big([\beta_i],[\mathcal{T}^n(\beta_j)]\big)}\ge \lambda.$
	This, together with (\ref{e:mhf=i}) and (\ref{e:minsec}), implies the desired result.
	\end{proof}

\subsection{Proof of Lemma \ref{LmClean}}\label{subsec:rotmap}

Set $M=P_\psi(S_1,S_2)$ and $\tau=\tau^k_2\tau^\ell_1$.  As before,  we identify two open neighborhoods $U_1$ and $U_2$ of the plumbing point $p$ in $T^*S_1$ and $T^*S_2$ by a symplectomorphism $\psi$. Pick a geodesic circle $\gamma_1$ of $S_1$ containing $p$, and let $\gamma_2$ be the unique geodesic circle of $S_2$ passing through $p$  such that two open subsets of $T^*\gamma_1$ and $T^*\gamma_2$ are identified near $p$ under the map $\psi$. 
Let $\mathcal{R}_{\theta}^i:T^*\gamma_i\to M, i=1,2$ be the rotation maps associated to  $S_i$ as described in Section~\ref{sec:subL}.


\begin{lem}\label{lem:dehnrot}
	Let $c_1\subseteq D^*\gamma_1$ and $c_2\subseteq D^*\gamma_2$ be two smooth curves. Then for any $k\in\mathbb{Z}$ and any $\theta\in \mathbb S^{n-1}$ we have
	\begin{equation}\label{lem:tauR}
		\begin{split}
			 \tau_1^k\big(\mathcal{R}_{\theta}^2(c_2)\big)=\mathcal{R}_{\theta}^2\big(\tau_1^k(c_2)\cap T^*\gamma_2\big)\cup_\psi\mathcal{R}_{\theta}^1\big(\tau_1^k(c_2)\cap T^*\gamma_1\big),\\
			 \tau_2^k\big(\mathcal{R}_{\theta}^1(c_1)\big)=\mathcal{R}_{\theta}^1\big(\tau_2^k(c_1)\cap T^*\gamma_1\big)\cup_\psi \mathcal{R}_{\theta}^2\big(\tau_2^k(c_1)\cap T^*\gamma_2\big).
		\end{split}
	\end{equation}

\end{lem}
\begin{proof}
	(\ref{lem:tauR}) follows from (\ref{e:tauro}) and (\ref{e:Rt}) and the fact that every symplectic Dehn twist along a sphere preserves all cotangent bundles of geodesic circles of this sphere. We only prove the former. Let $c\subset D^*\gamma_2$ be a smooth curve. Then we have
	\begin{eqnarray}
		 \tau_1\big(\mathcal{R}_{\theta}^2(c)\big)&=&\tau_1\big(\mathcal{R}_{\theta}^2(c)\cap \mathcal{R}_{\theta}^1(T^*\gamma_1)\big)\cup\tau_1\big(\mathcal{R}_{\theta}^2(c)\setminus \mathcal{R}_{\theta}^1(T^*\gamma_1)\big)\notag\\&=&\tau_1\big(\mathcal{R}_{\theta}^1(c\cap T^*\gamma_1)\big)\cup\big(\mathcal{R}_{\theta}^2(c)\setminus \mathcal{R}_{\theta}^1(T^*\gamma_1)\big)\notag\\&=&\mathcal{R}_{\theta}^1\big(\tau_1(c\cap T^*\gamma_1)\big)\cup\big(\mathcal{R}_{\theta}^2(c\setminus T^*\gamma_1)\big)\notag\\&=&\mathcal{R}_{\theta}^2\big(\tau_1(c\cap T^*\gamma_1)\cap T^*\gamma_2\big)\cup\mathcal{R}_{\theta}^1\big(\tau_1(c\cap T^*\gamma_1)\setminus T^*\gamma_2\big)\cup\big(\mathcal{R}_{\theta}^2\tau_1(c\setminus T^*\gamma_1)\big)\notag
		\\&=&\mathcal{R}_{\theta}^2\big(\tau_1(c)\cap T^*\gamma_2\big)\cup\mathcal{R}_{\theta}^1\big(\tau_1(c)\setminus T^*\gamma_2\big)\notag
		\\&=&\mathcal{R}_{\theta}^2\big(\tau_1(c)\cap T^*\gamma_2\big)\cup_\psi\mathcal{R}_{\theta}^1\big(\tau_1(c)\cap T^*\gamma_1\big).\notag
	\end{eqnarray}
\end{proof}

Now we are ready to give the proof of Lemma \ref{LmClean}.
\begin{proof}[Proof of Lemma \ref{LmClean}]

Since $S_i=\cup_{\theta\in \mathbb S^{n-1}}\mathcal{R}_{\theta}^i(\gamma_i), i=1,2$, we have
\begin{eqnarray}\label{e:cap}
	S_1\cap \tau^n(S_2)&=&\bigcup\limits_{\theta\in \mathbb S^{n-1}}\big(\mathcal R_\theta^1(\gamma_1)\cap \tau^n(S_2)\big)
	=\bigcup\limits_{\theta\in \mathbb S^{n-1}} \bigcup\limits_{\theta'\in \mathbb S^{n-1}}\big(\mathcal R_{\theta}^1(\gamma_1)\cap \tau^n(\mathcal{R}_{\theta'}^2(\gamma_2))\big)\notag\\
	&=&\bigcup\limits_{\theta\in \mathbb S^{n-1}}\mathcal R_\theta^1(\gamma_1\cap \tau^n(\gamma_2)),
\end{eqnarray}
where in the last equality we have used Lemma~\ref{lem:dehnrot} and the fact that each cotangent bundle of geodesic circle
$\mathcal{R}_{\theta}^1(T^*\gamma_1)\subset T^*S_1$ is identified locally with an unique one $\mathcal{R}_{\theta}^2(T^*\gamma_2)\subset T^*S_2$ near $p$ under the plumbing map $\psi$. Here we note that both $\gamma_1$ and $\tau^n(\gamma_2)$ locate in the plumbing domain $P_\psi(\gamma_1,\gamma_2)$. From (\ref{e:cap}) we can see that the intersection of $S_1$ and $\tau^n(S_2)$ is the union of $(n-1)$-spheres and/or the plumbing point $\{p\}$ and/or its antipodal point $\{-p\}$. Furthermore, when $\gamma_1$ and $\tau^n(\gamma_2)$ intersect transversely in $P_\psi(\gamma_1,\gamma_2)$, the Lagrangian $n$-spheres $S_1$ and $\tau^n(S_2)$ intersect cleanly.
Since $L_{\gamma_1}=S_1$ and $L_{\tau^n(\gamma_2)}=\tau^n(S_2)$, it follows from (\ref{e:Laginv}) that
\[
\iota(S_1)=S_1,\qquad\iota(\tau^n(S_2))=\tau^n(S_2).
\]
Set $N=S_1\cap \tau^n(S_2)$. Since $M^\iota=P_\psi(\gamma_1,\gamma_2)$, we have that $N^\iota=\gamma_1\cap \tau^n(\gamma_2)$. Clearly, for each sphere $\alpha$ belonging to $N$, one can pick an $\iota$-invariant Morse function $f$ on $\alpha$ such that the critical points of $f$ are precisely the unique maximal and minimal points belonging to $\alpha\cap N^\iota$.\end{proof}

	Similar to Lemma~\ref{lem:dehnrot}, we have
	
	\begin{lem}\label{lem:mdehnrot}
		Let $c_i\subseteq D^*\gamma_i$ be a smooth curve. Then for any $k\in\mathbb{Z}$ we have
		\begin{equation}\notag
			 \tau_j^k\big(\mathcal{R}_{\theta}^i(c_i)\big)=\mathcal{R}_{\theta}^1\big(\tau_j^k(c_i)\cap T^*\gamma_1\big)\cup_{\psi_1}\cdots\cup_{\psi_{m-1}}\mathcal{R}_{\theta}^m\big(\tau_j^k(c_i)\cap T^*\gamma_m\big),\quad j=1,\ldots,m.
		\end{equation}
	\end{lem}
	Put $\tau=\tau_1^{k_1}\tau_2^{k_2}\cdots\tau_m^{k_m}$.
With the last lemma, similar to \eqref{e:cap}, we have
	\[
	S_i\bigcap \tau^n(S_j)=\bigcup\limits_{\theta\in \mathbb S^{n-1}}\mathcal{R}_{\theta}^i(\gamma_i\cap \tau^n(\gamma_j)),
	\]
which implies the following

\begin{lem}\label{mLmClean}
	If $\gamma_i$ and $\tau^n(\gamma_j)$ intersect transversely in $P_\psi(S_1,\ldots,
    S_m)$, then  $S_i$ and $\tau^n(S_j)$ satisfy condition~$(C1)$ of Theorem~\ref{thm:KS}. 	
\end{lem}

\section{Classification of symplectic mapping class group}\label{sec:classfy}

In this section, we prove  Theorem \ref{thm:class} which is an analogue of Nielsen-Thurston classification for symplectic mapping class group in terms of growth of Floer cohomology in four dimension.

\subsection{The map $\bar{\mathcal R}$ on the surface $S_g$}

In this section, we explain more precisely how the map $\bar{\mathcal R}$ act on on the plumbing space $P_\psi(\gamma_1,\gamma_2,\ldots \gamma_m)$

We note that if we view $P_\psi(\gamma_1,\gamma_2,\ldots \gamma_m)$ as the one- or two-punctured genus $g$-surface $\hat{S}_g$, then the map $\bar{ \mathcal{R}}$ can be extended to the closed surface $S_g$ and it is exactly the hyperelliptic involution on $S_g$, illustrated in Figure \ref{figbarR}. We may also define a projection $\pi:P_\psi(\gamma_1,\gamma_2,\ldots \gamma_m)\to P_\psi(\gamma_1,\gamma_2,\ldots \gamma_m)/\bar{ \mathcal{R}}$ by projecting to the quotient of  $\bar{ \mathcal{R}}$ and the quotient is a punctured sphere. 

\begin{figure}
    \centering
    \includegraphics[width=\linewidth]{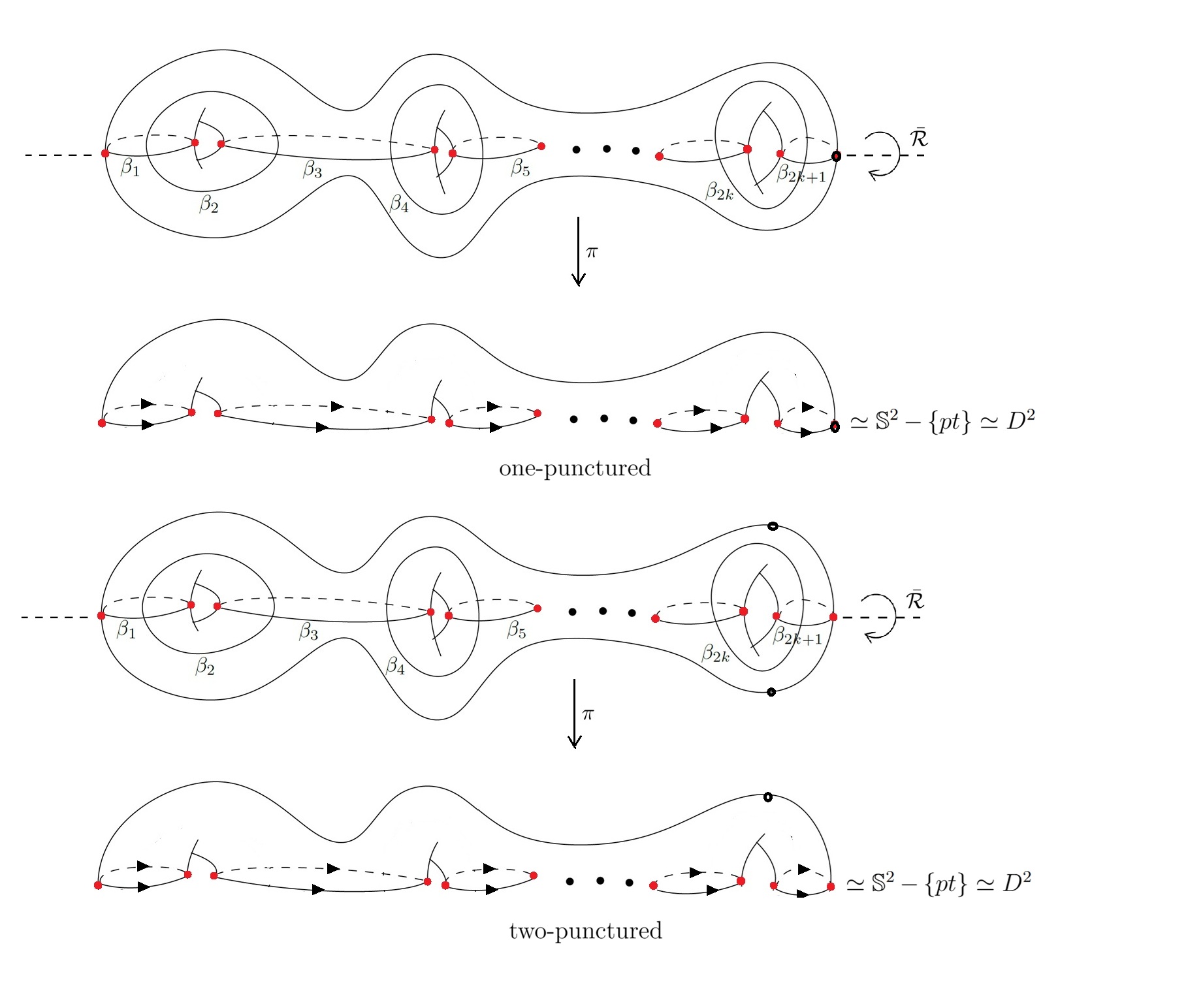}
    \caption{The hyperelliptic involution on genus $g$ surface}
   
    \label{figbarR}
\end{figure}

\begin{prop}\label{projtohypelip}
   The projection $\pi:P_\psi(\gamma_1,\gamma_2,\ldots \gamma_m)\to P_\psi(\gamma_1,\gamma_2,\ldots \gamma_m)/\bar{ \mathcal{R}}$ is a branched covering, branched over $(m+1)$ points, and the quotient $P_\psi(\gamma_1,\gamma_2,\ldots \gamma_m)/\bar{ \mathcal{R}}$ is homeomorphic to a disk $D^2$.
\end{prop}
\begin{proof}
    
Indeed, the hyperelliptic involution extends to the surface $S_g$, we consider the quotient $S_g$ under the hyperelliptic involution. A picture illustration is given in Figure \ref{figbarR}, where the red points are the branching points which are exactly fixed points of $\bar {\mathcal R}$, the black circled points are the punctures of $P_\psi(\gamma_1,\gamma_2,\ldots \gamma_m)$, and the arrowed sides are glued together. The quotient by $\bar{\mathcal R}$ is given by cutting along the curves $\beta_{2i-1}$ (which cuts $S_g$ into two parts), taking one part of the surface, and gluing each $\beta_{2i-1}$ into a line segments, which yields a topological sphere. Recall that $P_\psi(\gamma_1,\gamma_2,\ldots \gamma_m)$ is topologically a one-punctured or two-punctured surface, as in Figure \ref{meven} or \ref{modd}, where the set of punctures is invariant under the extension of $\bar{\mathcal R}$, and maps to one point in the quotient by $\bar{\mathcal R}$, the quotient manifold $P_\psi(\gamma_1,\gamma_2,\ldots \gamma_m)$ is topologically a one-punctured sphere, which is topologically a disk $D^2$. 
\end{proof}

Let $\gamma\subset P_\psi(\gamma_1,\gamma_2,\ldots \gamma_m)$ be a simple closed curve, then its image under the projection $\pi(\gamma)$ is a simple closed curve if and only if $\gamma\cap \bar{\mathcal{R}}\gamma=\varnothing$. Otherwise it has self-intersections corresponding to intersections of $\gamma$ and $\bar{\mathcal{R}}\gamma$.

\subsection{Proof of Theorem~\ref{thm:class}}

 In the following we identify $P_\psi(\gamma_1,\ldots,\gamma_m)$ with the surface embedded in $S_g$ as in the proof of Theorem~\ref{thm:symrate2}. We shall use the following lemma to establish the existence of invariant Lagrangian spheres or tori for reducible symplectic mapping classes. Recall that a curve is called admissible if it is fixed by the rotation map $\bar{\mathcal R}$. We need to rotate an admissible curve to recover a Lagrangian sphere in the symplectic manifold  $P_\psi(S_1,\ldots,S_m)$.

\begin{prop}
    \label{curveadmiss}
 {Let $[\gamma]$ be a free homotopy class in $P_\psi(\gamma_1,\ldots,\gamma_m)$ and let $\bar{\mathcal R}$ be the involution given in Definition \ref{barR}. If $[\bar{\mathcal R}\gamma]=[\gamma]$,} then there exists an admissible representative $\gamma^\prime$ which generates a smoothly embedded Lagrangian sphere $L_{\gamma^\prime}=\cup_i\cup_{\theta\in \mathbb{S}^1} \mathcal{R}^i_\theta(\gamma^\prime)$ in $P_\psi(S_1,\ldots,S_m)$.  {If $[\bar{\mathcal R}\gamma]\neq[\gamma]$, then there exists a representative $\gamma^\prime $ such that $\gamma^\prime$ is disjoint from $\bar{\mathcal R}\gamma^\prime$.}
\end{prop}

We postpone the proof to Section~\ref{sec:repr}. 

\begin{proof}[Proof of {Theorem~\ref{thm:class}}]
Let $f$ be a compactly supported symplectomorphism  of $A^2_m$-surface. It follows from Theorem~\ref{thm:hamiso} that $f=\phi_0\circ\tau_0$ for some compactly supported Hamiltonian diffeomorphism $\phi_0$ of $A^2_m$ and a composition $\tau_0$ of symplectic Dehn twists $\tau_{S_i},i=1\ldots,m$ along spheres $S_i$, i.e.,
$
\tau_0=\tau_{S_{j_1}}^{k_1}\circ\cdots\circ\tau_{S_{j_m}}^{k_m}$
for some integers  {$k_i\in\Z$}, where $j_1,\ldots,j_m$ is a permutation from $1$ to $m$.
For any Lagrangian spheres $\alpha,\beta\subseteq A^2_m$,
by Theorem~\ref{thm:lagiso}, we have
\[
\alpha=\phi_1\circ \tau_1(S_i),\quad \beta=\phi_2\circ \tau_2(S_j),
\]
for some $i,j\in\{1,\ldots,m\}$, where $\phi_1,\phi_2$ are compactly supported Hamiltonian diffeomorphisms of $A^2_m$ and $\tau_1,\tau_2$ are compositions of symplectic Dehn twists. Here we allow the case $i=j$.   We have that $f^n\beta$ is Lagrangian isotopic to $\tau_0^n\tau_2(S_j)$, since we can isotope $\phi_0$ and $\phi_2$ to identity through a path of Hamiltonian diffeomorphisms. Therefore $f^n\beta$ is Hamiltonian isotopic to $\tau_0^n\tau_2(S_j)$ due to ${\rm H}^1(\bS^2;\R)=0$. Since Lagrangian Floer cohomology is invariant with respect to Hamiltonian isotopies and symplectic Dehn twists are symplectomorphisms, we get
\[
\HF(\alpha,f^n\beta)=\HF(\tau_1(S_i),\tau_0^n\tau_2(S_j))=\HF\big(S_i,\tau_1^{-1}\tau_0^n\tau_2(S_j)\big).
\]

Now we discuss in two cases:
\begin{enumerate}
    \item  $\alpha$ is Lagrangian isotopic to $f^n\beta$;
    \item  $\alpha$ is not Lagrangian isotopic to $f^n\beta$.
\end{enumerate}
In case (1), the map $f$ is automatically blackucible. We next focus on case (2), then $\tau_1(S_i)$ is not Lagrangian isotopic to $\tau_0^n\tau_2(S_j)$, and hence,
by Lemma~\ref{lem:isot}, $\gamma_i$ is not isotopic to $\tau_1^{-1}\tau_0^n\tau_2(\gamma_j)$ in the sense of Definition~\ref{def:lagiso}. Then by
 {Lemma~\ref{e:mhf=i}},
\[
\rk\; \HF(\alpha,f^n\beta)=I\big(\gamma_i,\tau_1^{-1}\tau_0^n\tau_2(\gamma_j)\big)=I\big(\tau_1(\gamma_i),\tau_0^n\tau_2(\gamma_j)\big).
\]

Notice that for $m=2g,2g+1$, the plumbing space $P_\psi(\gamma_1,\ldots,\gamma_m)$ is homeomorphic to an oriented surface ${S}_{g}$ of genus $g$ with one or two punctures.  By the Nielsen-Thurston's classification~\cite[Theorem 13.2]{FM},  the mapping class $[\tau_0]\in \hbox{Mod} (\hat{S}_{g})$ has a representative $\phi$ of one of the three types:
\begin{enumerate}
    \item[I.] Reducible: $\phi$ leaves invariant a finite collection of pairwise disjoint simple closed curves in $\hat{S}_{g}$;
    \item [II.] Periodic: $\phi^m=Id$ for some positive integer $m$;
    \item [III.] Pseudo-Anosov: there are transverse measured foliations $(\mathcal{F}^s,\mu_s)$ and $(\mathcal{F}^u,\mu_u)$ on $\hat{S}_{g}$ and a real number $\lambda>1$ so that
    \[
\phi\cdot(\mathcal{F}^s,\mu_s)=(\mathcal{F}^s,\lambda\mu_s)\quad\hbox{and}\quad
\phi\cdot(\mathcal{F}^u,\mu_u)=(\mathcal{F}^u,\lambda^{-1}\mu_u).
    \]
\end{enumerate}

Therefore, the geometric intersection number $I\big(\tau_1(\gamma_i),\tau_0^n\tau_2(\gamma_j)\big)$ is periodic in $n$ for type II, and grows exponentially in $n$ for type III, see for instance~\cite[Theorem 14.24]{FM}; and so does $\rk\; \HF(\alpha,f^n\beta)$ for these two types.

In the case that $\phi$ is reducible, for some positive integer $k$, $\phi^k$ preserves a closed curve $\gamma$ in $P_\psi(\gamma_1,\gamma_2,\ldots,\gamma_m)$.

By Proposition \ref{curveadmiss},  {in the case that $[\bar{\mathcal R}\gamma]=[\gamma]$, there exists an admissible representative $\gamma^\prime$ such that $L_{\gamma^\prime}$ is an embedded Lagrangian sphere. }Therefore, by Proposition \ref{lem:isot}, $\tau_0^k$ preserves the Lagrangian sphere $L$ generated by $\gamma$ up to Lagrangian isotopy, and so does $f^k$.

In the case that $[\bar{\mathcal R}\gamma]\neq [\gamma]$, by Proposition \ref{curveadmiss}, there exists a representative $\gamma^\prime$ such that $\gamma^\prime\cap \bar{\mathcal R}\gamma^\prime=\varnothing$. By the same reasoning as in the proof of Lemma \ref{lem:smlg}, this shows that $\big(\cup_{\psi_i}\mathcal R^i_\theta(\gamma^\prime (\mathbb S^1))\big)\cap \big(\cup_{\psi_i}\mathcal R^i_{\theta^\prime}(\gamma^\prime (\mathbb S^1))\big)=\varnothing$ when $\theta\neq \theta^\prime$, where each $\cup_{\psi_i}\mathcal R^i_\theta(\gamma^\prime (\mathbb S^1))$ is diffeomorphic to a circle $\mathbb S^1$. Therefore, applying Lemma \ref{LmLagrangian}, we see that $L=\cup_{\psi_i}\cup_\theta \mathcal R^i_\theta(\gamma^\prime (\mathbb S^1))$ is a smoothly embedded Lagrangian submanifold in $M$. 

Since $\tau_0$ commutes with $\bar{\mathcal{R}}$, the mapping class $[\tau_0]\in \mathrm{Mod}(S_g)$ preserves both homotopy classes $[\gamma^\prime]$ and $[\bar{\mathcal R}\gamma^\prime] $. Since $\gamma^\prime$ and $\tau_0^k\gamma^\prime$ are homotopic in the class of $[\gamma]$, and both do not intersect with their image under $\bar{\mathcal R}$. Consider the quotient map $\pi: P_\psi(\gamma_1,\cdots,\gamma_m)\to D^2$ (see Proposition \ref{projtohypelip}).  We see that  $\pi(\gamma^\prime)$ and $\pi(\tau_0^k\gamma^\prime)$ must be homotopic simple closed curves in $D^2-\{\text{branching points}\}$. Therefore, there exists a homotopy from $\gamma^\prime$ to $\tau_0^k\gamma^\prime$, such that every intermediate loop $\gamma_1$ is mapped to a simple closed curve by $\pi$. This shows that any intermediate loop satisfies $\gamma_1\cap \bar{\mathcal R}\gamma_1=\varnothing$, and by the same reasoning as above, every $\gamma_1$ generates a smoothly embedded Lagrangian submanifold $L_{\gamma_1}$ in $M$. This shows that $\tau_0^k$ preserves the Lagrangian submanifold $L$ up to Lagrangian isotopy, and so does $f^k$.

 In the special case that $m=2$, since $\tau_0$ is identity away from the spheres $L_1,L_2$, we may consider its restriction on $P_\psi(\gamma_1,\gamma_2)\subset \mathbb T^2$, and gain a mapping class $[\tau_0]\in \mathrm{Mod}(\mathbb T^2)\simeq \mathrm{PSL}_2\Z$. From our analysis above, $[f]$ is reducible/periodic/hyperbolic if and only if $[\tau_0]$ is reducible/periodic/hyperbolic. By Lemma \ref{LmLTM}, the map $\tau_0\to [\tau_0]$ is generated by $\tau_{S_1}\mapsto \begin{pmatrix}
     1&1\\0&1
 \end{pmatrix}$ and $\tau_{S_1}\mapsto \begin{pmatrix}
     1&0\\-1&1
 \end{pmatrix}$. Therefore the map $\pi_0(\mathrm{Symp}_c(A^2_m,\omega)\to  \mathrm{Mod}(\mathbb T^2)\simeq \mathrm{PSL}_2\Z$ is given by the homomorphism $\rho$.

\end{proof}



\subsection{Proof of Proposition \ref{curveadmiss} (Existence of admissible representatives).}\label{sec:repr}
In what follows, we identify $P_\psi(\gamma_1,\ldots,\gamma_m)$ with a subset of $\mathbb \R^2$ with parallel sides identified as illustrated in Figure~\ref{Ydelta} and equip it with the Euclidean metric. Denote by $d$ the distance between two points by this metric.

By abuse of notations, for a simple closed curve $\gamma$ lying in $P_\psi(\gamma_1,\ldots,\gamma_m)$,  we also use $\gamma$ to denote its image in it.

In the case that $[\bar{\mathcal{R}}\gamma]\neq [\gamma]$, we consider the branched double cover $\pi:P_\psi(\gamma_1,\gamma_2,\cdots \gamma_n)\to \mathbb S^2-\{pt\}$ given by modding out by the hyperelliptic involution $\bar{\mathcal R}$. We may find a representative $\gamma^\prime$ that avoids all fixed points of $\bar{\mathcal R}$, hence $\pi(\gamma^\prime)$ avoids all branched points in $\mathbb{S}^2$. Further perturb $\pi(\gamma^\prime)$ to remove any self-intersection, then the lifting of $\pi(\gamma^\prime)$, which is exactly $\gamma^\prime \cup \bar{\mathcal{R}}\gamma'$, must be a disjoint union of two simple closed curves.  

In the case when $\bar{\mathcal{R}}\gamma$ is homotopic to $\gamma$,  we shall proceed in two steps:
\begin{itemize}
    \item First we show that $[\gamma]$ has an admissible representative $\gamma^\prime$. Thus by Lemma \ref{lem:smlg}, the submanifold $L_{\gamma^\prime}$ generated by $\gamma^\prime$ under the rotation maps is a smoothly embedded Lagrangian submanifold of $A^2_m$.

    \item  Next we show that such a representative must have exactly two fixed points under $\bar{\mathcal{R}}$. Therefore, the submanifold $L_{\gamma^\prime}$ must be a topological sphere, hence a smoothly embedded Lagrangian sphere in $A^2_m$.
\end{itemize}

Now we prove these two claims.
\begin{lem}
Any  {$\bar{\mathcal R}$-invariant} free homotopy class of simple closed curves in $P_\psi(\gamma_1,\ldots,\gamma_m)$  has a representative as an admissible simple closed curve.
\end{lem}
\begin{proof}

Let $\gamma:\mathbb S^1\to P_\psi(\gamma_1,\ldots,\gamma_m)$ be a smooth simple closed curve in the class $[\gamma]$. We will take the set $\mathcal C(\epsilon,C)$ of Lipchitz curves in a compact subset of $P_\psi(\gamma_1,\ldots,\gamma_m)$ and consider the area swept out by the homotopy between $\gamma^\prime$ and $\bar{\mathcal R}\gamma^\prime$ for each $\gamma^\prime\in \mathcal C(\epsilon,C)$. We shall show that the curve that minimizes this area must be admissible.

If $\gamma=\bar{\mathcal{R}}(\gamma)$, nothing needs to be proved. Otherwise, one can choose a set $\mathcal D$ to be a compact $\bar{\mathcal R}$-invariant subset that contains $\gamma$, so that there exists a homotopy $$F: [0,1]\times \mathbb S^1\to P_\psi(\gamma_1,\ldots,\gamma_m),\quad F(0,\cdot)=\gamma,\; F(1,\cdot)=\bar{\mathcal{R}}\circ\gamma$$
such that the image of $F$ is in the set $\mathcal{D}$.


For any closed curve $\sigma:\mathbb S^1 \to P_\psi(\gamma_1,\ldots,\gamma_m)$, we denote
$$
\hbox{Lip}(\sigma):=\sup_{s,t\in \mathbb S^1}\frac{d\big(\sigma(s),\sigma(t)\big)}{d_{\mathbb S^1}(s,t)}
$$
where $d$ is the distance function induced by a Riemannian metric on $P_\psi(\gamma_1,\ldots,\gamma_m)$ and $d_{\mathbb S^1}$ is the standard distance on $\mathbb S^1$.

Given $C>\hbox{Lip}(\gamma)$, we pick $\epsilon>0$ small enough ($C \epsilon\ll  1$) and consider the set of curves
\begin{equation}
\begin{split}
    \mathcal{C}(\epsilon, C)=\big\{\sigma\in C^\infty(\mathbb S^1, \mathcal{D}) \big|&\;[\sigma]=[\gamma];\; \hbox{Lip} (\sigma)\le C  ;  d(\sigma(s),\sigma(t))\ge \epsilon \text{ if }  d_{\mathbb S^1}(s,t)\ge C\epsilon\big\}.\notag
\end{split}
\end{equation}
Clearly, $\gamma,\bar{\mathcal{R}}(\gamma)\in\mathcal{C}(\epsilon, C)$ for $\epsilon$ small enough.


For any $\sigma\in \mathcal{C}(\epsilon, C)$, we denote by $\mathcal{A}_\sigma=\cap_{G} G([0,1]\times \mathbb{S}^1)$ where $G:[0,1]\times \mathbb{S}^1\to \mathcal{D}$ is taken over all possible homotopies from $\sigma$ to $\bar{\mathcal{R}}(\sigma)$ inside $\mathcal{D}$.

We observe that $\mathcal{A}_\sigma$ is a compact subset of $\mathcal{D}$ that is bounded by $\sigma$ and $\bar{\mathcal{R}}(\sigma)$. We claim that the region $\mathcal{A}_\sigma$ must be invariant under $\bar{\mathcal{R}}$.  If not, then since $\bar{\mathcal{R}}\circ \bar{\mathcal{R}}=Id$, $\bar{\mathcal{R}}(\mathcal{A}_\sigma)$ is also a region bounded by $\sigma$ and $\bar{\mathcal{R}}(\sigma)$. 
If $\sigma$ and $\bar{\mathcal{R}}(\sigma)$ intersect, then the set of intersection points is $\bar{\mathcal{R}}$-invariant. And we see that $\mathcal{A}_\sigma$ is a union of bigons (possibly including some line segments) whose boundaries are $\bar{\mathcal{R}}$-invariant, which implies that $\mathcal{A}_\sigma$ is $\bar{\mathcal{R}}$-invariant. If $\sigma$ and $\bar{\mathcal{R}}(\sigma)$ do not intersect, then the two regions $\mathcal{A}_\sigma$ and $\bar{\mathcal{R}}(\mathcal{A}_\sigma)$ are either the same, or their union must be the whole surface $P_\psi(\gamma_1,\ldots,\gamma_m)$. The latter is impossible since $\bar{\mathcal{R}}$ fixes the plumbing points and their antipodal points.

We now consider the map $\mathcal{F}_{\bar{\mathcal{R}}}: \mathcal{C}(\epsilon, C)\to \R_{\ge 0}$,  $$\mathcal{F}_{\bar{\mathcal{R}}}(\sigma)=Area(\mathcal{A}_\sigma).$$ Here, we take the unsigned (non-negative) area of a region given by the metric on the surface.

Since $\mathcal{C}(\epsilon, C)$ is compact in the space of Lipschitz continuous curves from $\mathbb S^1$ to $\mathcal{D}$ with respect to the Lipschitz norm, one can take $\gamma^\prime\in \mathcal{C}(\epsilon, C)$ to achieve the minimum of $\mathcal{F}_{\bar{\mathcal{R}}}$. We claim that $\gamma^\prime= \bar{\mathcal{R}}(\gamma^\prime)$. Indeed, if $\gamma^\prime\ne  \bar{\mathcal{R}}(\gamma^\prime)$, then for any curve $\sigma$ in the homotopy class lying in $\mathcal{A}_{\gamma^\prime}$, the area $\mathcal{F}_{\bar{\mathcal{R}}}(\sigma)$ is smaller than $\mathcal{F}_{\bar{\mathcal{R}}}(\gamma')$. This contradiction implies that $\gamma^\prime= \bar{\mathcal{R}}(\gamma^\prime)$.


Moreover, the curve $\gamma^\prime$ is an admissible Lipschitz curve in the homotopy class of $\gamma$ whose only self intersection points are those $\gamma'(s)=\gamma'(t)$ with  {$d_{\mathbb S^1}(s,t)\leq C\epsilon$}. Hence, for a small enough $\epsilon$,  any self-intersection of $\gamma^\prime$ must give rise to a small disk and thus can be removed by deforming the curve in the homotopy class. Since $\gamma^\prime$ is admissible,  the resulting simple closed curve is also admissible.
\end{proof}

\begin{lem}
Any smooth admissible simple closed curve in $P_\psi(\gamma_1,\ldots,\gamma_m)$ has exactly two fixed points under $\bar{\mathcal{R}}$.

\end{lem}
\begin{proof}

For an admissible curve $\gamma\subset P_\psi(\gamma_1,\ldots,\gamma_m)$,  {given by $\gamma:\R/\Z\to P_\psi(\gamma_1,\ldots,\gamma_m)$,} we assume without loss of generality that $\gamma(0)$ is not a fixed point under $\bar{\mathcal R}$ and $\gamma$ has unit speed under the Euclidean metric $\R^2$ when we view $P_\psi(\gamma_1,\ldots,\gamma_m)$ as a subset of $\R^2$ with opposite sides identified as in Figures \ref{meven} and \ref{modd}. Under this assumption, there exists a unique $t\in[0,1]$ such that $\gamma(t)=\bar{\mathcal{R}}(\gamma(0))$. We claim that $\gamma$ has at least two fixed points by $\bar{\mathcal{R}}$. Indeed, since $\gamma$ is admissible, simple closed with unit speed,  {at least one of the following holds:\begin{itemize}
    \item  $\gamma(s)=\bar{\mathcal{R}}\circ\gamma(t-s)$ and $\gamma(t+s)=\bar{\mathcal{R}}\circ\gamma(1-s)$;
    \item $\gamma(s)=\bar{\mathcal{R}}\circ\gamma(1-s)$ and $\gamma(t+s)=\bar{\mathcal{R}}\circ\gamma(t-s)$, for any $s\in[0,t]$.
\end{itemize} Plug in $s=0$, we see that the second case cannot hold since we assumed that $\gamma(0)$ is not fixed by  $\bar{\mathcal{R}}$.  We plug in $s=\frac{t}{2}$ and $s=\frac{-t+1}{2}$ in the first case. This  shows that $\gamma(\frac{t}{2})$ and $\gamma(\frac{1+t}{2})$ are fixed points of $\bar{\mathcal{R}}$. }


We next show that an admissible curve has at most two fixed points. By shifting the time of $\gamma(s)$ to $s^\prime=s-\frac{t}{2}$, we may assume without loss of generality that $\gamma$ is unit speed and $\gamma(0)$ is a fixed point by $\bar{\mathcal{R}}$, then by the same reasoning as above, $\gamma(\frac{1}{2})$ must be a fixed point of $\bar{\mathcal{R}}$.

Let  {$t_1=\min\{s\in (0,\frac{1}{2}]\ |\ \bar{\mathcal{R}}\circ\gamma(s)=\gamma(s)\}$ be the first intersecting time between $\gamma$ and $\bar{\mathcal R}\gamma$. Then we have $\gamma(t_1+s)=\bar{\mathcal{R}}\circ\gamma(t_1-s)$
for any $0\le s\le t_1$. Plugging in $s=t_1$, we have $\gamma(0)=\gamma(2t_1)$, but since $\gamma$ is a simple closed curve, we have $2t_1=1$. Therefore, $t_1=\frac{1}{2}$,} and $\gamma$ only has two fixed points.
\end{proof}

\section{ Floer cohomology groups and geometric intersection numbers}\label{SIntersection}
In this section, we relate Lagrangian Floer cohomology groups to the geometric intersection numbers in the reduced space $P_\psi(\gamma_1,\ldots, \gamma_n)$ and give the proof of Lemma  \ref{LmIntersection} and \ref{multiintersect}.

\subsection{Lagrangian Floer cohomology and geometric intersection number}

The geometric intersection number between $\alpha$ and $\beta$ is a homotopic invariant. If two simple closed curves $a$ and $b$ realize the minimal intersection in their homotopy classes $\alpha=[a]$ and $\beta=[b]$, i.e., $I(\alpha,\beta)=\sharp (a\cap b)$, then we say that $a$ and $b$ \emph{in minimal position}. The following bigon criterion from ~\cite[Proposition 1.7]{FM} is useful for verifying the minimal position.

\begin{lem}[The bigon criterion]\label{lem:bigon}
	Two transverse simple closed curves $a$ and $b$ in a surface $S$ are in minimal position if and only if they do not form a bigon. Here a \emph{bigon} refers to an embedded disk in $S$ whose boundary is the union of an arc of $a$ and an arc of $b$ intersecting in exactly two points.
\end{lem}
Thus, in order to prove Lemma \ref{LmIntersection}, it is enough to show  how to eliminate bigons.
\begin{lem}\label{elimbigon}
	Let $\gamma_1$ and $\gamma_2$ be the big circles that we have chosen for the plumbing domain $P_\psi(S_1,S_2)$. Then any bigon formed by $\gamma_1$ and $\tau^n(\gamma_2)$  on $P_\psi(\gamma_1,\gamma_2)$ can be eliminated by homotoping $\gamma_1$ provided that $\gamma_1$ and $\tau^n(\gamma_2)$ are not admissibly isotopic.
\end{lem}
\begin{proof}

We parameterize $D^*\gamma_2$ by
$
T^*\gamma_2\cong S^1\times [-1,1]=\big\{(\theta,t)\in [0,2\pi)\times [-1,1]\big\}.
$
Then we can define two subsets $T_\pm^*\gamma_2\subset T^*\gamma_2$ by
\[
D_+^*\gamma_2\cong\big\{(\theta,t)\in [0,2\pi)\times [0,1] \big\},
\quad
D_-^*\gamma_2\cong\big\{(\theta,t)\in [0,2\pi)\times [-1,0] \big\}.
\]
Let $\gamma_1^+$ (resp. $\gamma_1^-$) be one of half geodesic circles of $\gamma$ connecting  the plumbing point $p$ to its antipodal point $q_1$ such that $D^*\gamma_1^+\subseteq T^*\gamma_1$ (resp. $D^*\gamma_1^-\subseteq T^*\gamma_1$) is identified with $D_+^*\gamma_2$ (resp. $D_-^*\gamma_2$ ) near $p$ under the plumbing map $\psi$, as illustrated in Figure \ref{p+}.

\begin{figure}[ht]
	\begin{center}
		\includegraphics[width=0.5\textwidth]{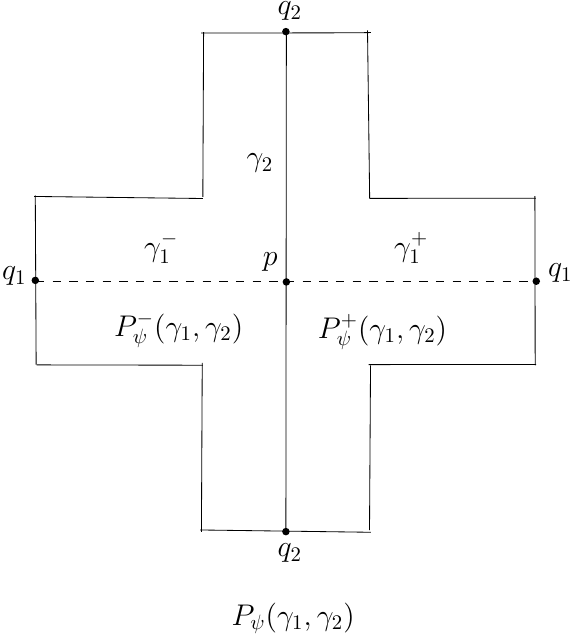}
	\end{center}
		\caption{$P^+_\psi(\gamma_1,\gamma_2)$ and $P^-_\psi(\gamma_1,\gamma_2)$ in the plumbing space of $T^*\gamma_1$ and $T^*\gamma_2$}\label{p+}
\end{figure}

Clearly, we have {$D^*\gamma_1=D^*\gamma_1^+\cup D^*\gamma_1^-$ and $D^*\gamma_2=D^*_+\gamma_2\cup D^*_-\gamma_2$. Denote
\[
P_\psi^+(\gamma_1,\gamma_2)=D^*\gamma_1^+\cup_\psi D^*_+\gamma_2,\quad P_\psi^-(\gamma_1,\gamma_2)=D^*\gamma_1^-\cup_\psi D^*_-\gamma_2.
\]
Then $P_\psi(\gamma_1,\gamma_2)=P_\psi^+(\gamma_1,\gamma_2)\cup P_\psi^-(\gamma_1,\gamma_2)$ and $\gamma_2=P_\psi^+(\gamma_1,\gamma_2)\cap P_\psi^-(\gamma_1,\gamma_2)$.


 We define the map $\mathcal{R}:P_\psi^+(\gamma_1,\gamma_2)\to P_\psi(S_1,S_2)$ to be the restriction of the map $\bar{\mathcal{R}}$ to $P_\psi^+(\gamma_1,\gamma_2)$  (cf. Definition~\ref{barR} for $\bar{\mathcal{R}}$).} 
The image of $\mathcal{R}$ is precisely $P_\psi^-(\gamma_1,\gamma_2)$. By Lemma~\ref{lem:dehnrot},
\[
\mathcal{R}\big(\tau^n(\gamma_2)\cap P_\psi^+(\gamma_1,\gamma_2)\big)=\tau^n(\gamma_2)\cap P_\psi^-(\gamma_1,\gamma_2),
\]
and hence, $\tau^n(\gamma_2)$ is invariant under the map $\bar{\mathcal{R}}$. Note that $\gamma_1$ is also invariant under $\bar{\mathcal{R}}$.
So both $\gamma_1$ and $\tau^n(\gamma_2)$ belong to $\mathcal{C}_{ad}$, and all bigons of $\gamma_1$ and $\tau^n(\gamma_2)$ appear in pairs related by $\mathcal{R}$.
Let $q_1$ and $q_2$ be the antipodal points of the plumbing point $p$ in $S_1$ and $S_2$ respectively. Denote $\Sigma=\{p,q_1,q_2\}$. As a map on $P_\psi(\gamma_1,\gamma_2)$, each symplectic Dehn twist $\tau_i$ satisfies $\tau_i(\Sigma)=\Sigma$, hence the curve $\tau^n(\gamma_2)$ contains precisely two elements of $\Sigma$.
Then all bigons $\Omega$ of $\gamma_1$ and $\tau^n(\gamma_2)$ are divided into three types:
\begin{enumerate}
	\item[I.] $\Omega\cap \Sigma=\{p,q_1\}$;
	\item[II.]  $\Omega\cap \Sigma=\{p\}$ or $\{q_1\}$;
	\item[III.]  $\Omega\cap \Sigma=\emptyset$.
\end{enumerate}
 Let $D_+$ and $D_-=\mathcal{R}(D_+)$ be a pair of such bigons.
If $\gamma_1$ and $\tau^n(\gamma_2)$ has a bigon of type I, then $\tau^n(\gamma_2)$ contains a simple closed curve homotopic to $\gamma_1$ and thus $\tau^n(\gamma_2)$ itself is homotopic to $\gamma_1$. If $D_+$ and $D_-$ are of type II or type III, one can eliminate these two bigons by deforming $\gamma_1$ to an admissibly isotopic curve $\gamma_1'\in \mathcal{C}_{ad}$ with $p$ and $q_1$ being fixed as illustrated in the Figure \ref{deform}.
\end{proof}
\begin{figure}[ht]
	\begin{center}
		\includegraphics[width=0.9\textwidth]{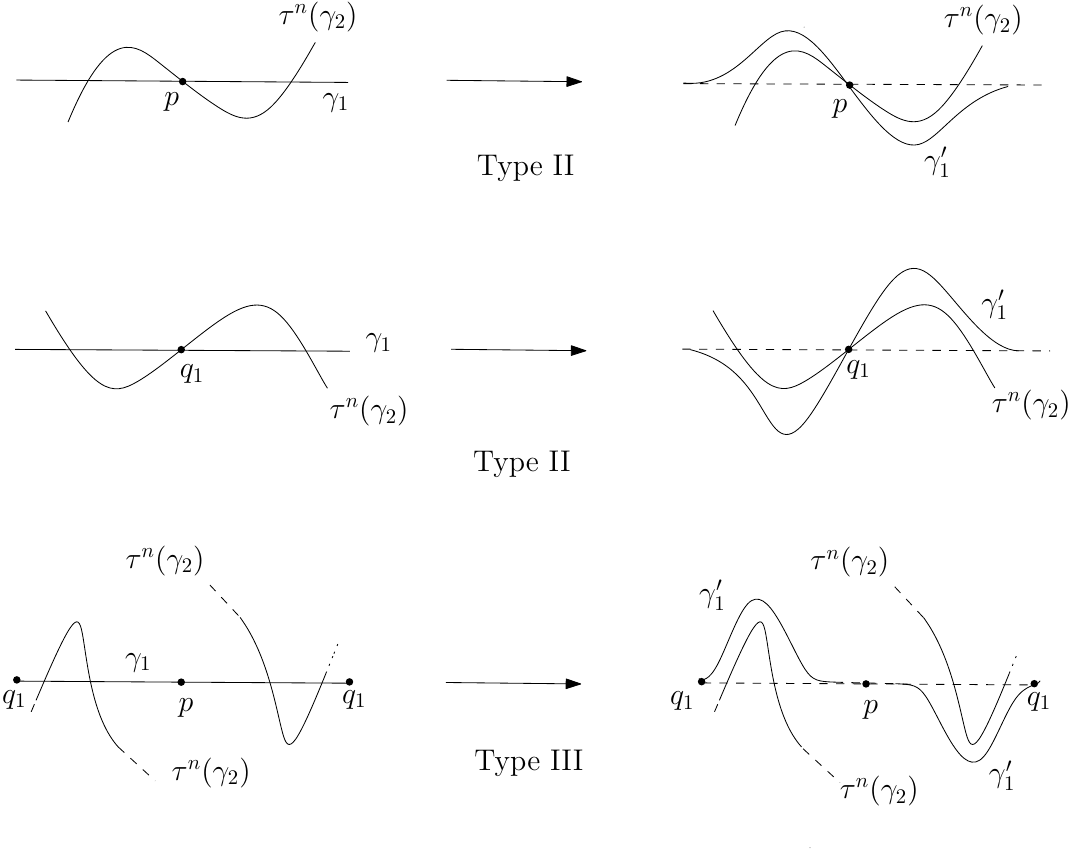}
	\end{center}
		\caption{Deforming $\gamma_1$ to eliminate the bigons}\label{deform}
\end{figure}

The last lemma enables us to complete the proof of  Lemma \ref{LmIntersection}.

\begin{proof}[Proof of Lemma \ref{LmIntersection}]
	
	For the two simple closed curves $\tau^n(\gamma_2),\gamma_1:\mathbb{R}/\mathbb{Z}\to P_\psi(\gamma_1,\gamma_2)$, whenever $\tau^n(\gamma_2)$ and $\gamma_1$ are not admissibly isotopic, by successively killing the unnecessary intersection points we can find an isotopy
	\[
	h:[0,1]\times \mathbb{R}/\mathbb{Z}\longrightarrow P_\psi(\gamma_1,\gamma_2),\quad h(0,t)=\gamma_1(t),\;h(1,t)=\gamma_1'(t)
	\]
	satisfying the following properties: during the isotopy each $h_s=h(s,\cdot)$ is a simple closed curve containing
	$\{p,q_1\}$, and $\gamma_1'$ and $\tau^n(\gamma_2)$ are transverse
	and in minimal  position in $P_\psi(\gamma_1,\gamma_2)$, and the image of each $h_s$ is invariant under  {$\bar{\mathcal{R}}$}.
	
	Then by Lemma~\ref{lem:isot} the Lagrangian spheres $S_1'=L_{\gamma_1'}$ and $S_1$ are Lagrangian isotopic via $L_{h_s}$, and hence exact Lagrangian isotopic because of $H^1(\mathbb S^n,\mathbb{R})=0$ with $n\geq 2$.
By Lemma~\ref{LmClean}, the condition~(C1) of Theorem~\ref{thm:KS} for $S_1'$ and $\tau^n(S_2)$ are still satisfied. Since $S_1'\cap M^\iota=\gamma_1'$ and $\tau^n(S_2)\cap M^\iota=\tau^n(\gamma_2)$, by Lemma~\ref{lem:bigon} the condition~(C2) of Theorem~\ref{thm:KS} for $S_1'$ and $\tau^n(S_2)$ holds. Then by Theorem~\ref{thm:KS} we have
	\begin{equation}\label{e:hf}
		\rk \HF\big(S_1,\tau^n(S_2)\big)=\rk \HF\big(S_1',\tau^n(S_2)\big)=\rk H^*(C,\mathbb{Z}/2),
	\end{equation}
	where $C=S_1'\cap\tau^n(S_2)$ and $C^\iota=\gamma_1'\cap\tau^n(\gamma_2)$. Replacing $\gamma_1$ by $\gamma_1'$ in (\ref{e:cap}), we find that if $z\in C^\iota\setminus\{p,q_1\}$ then the corresponding connected component of $C$ is a sphere and thus its contribution to $\rk H^*(C,\mathbb{Z}/2)$ is $2$, and if $z\in\{p,q_1\}$ then its contribution is $1$. Since $\gamma_1'$ and $\tau^n(\gamma_2)$ are invariant under $\bar {\mathcal{R}}$, the intersection points of $\gamma_1'$ and $\tau^n(\gamma_2)$ appear in pairs except for $\{p,q_1\}\subset C^\iota$. So we get
$
		\rk \;H^*(C,\mathbb{Z}/2)=I\big([\gamma_1'],[\tau^n(\gamma_2)]\big).
$
	This, together with (\ref{e:hf}), implies the statement.	
\end{proof}

\subsection{Proof of Lemma \ref{multiintersect}}\label{proofmultint}
 { The case of $m>2$ is  similar. 
\begin{proof}[Proof of Lemma \ref{multiintersect}] The idea of the proof is to deform the generating curves $\gamma_i$ and $\tau^n(\gamma_j)$ into admissible curves which are transverse and in minimal position, and then apply Theorem~\ref{thm:KS} to conclude the equality~(\ref{e:mhf=i}).  In our configuration of the plumbing domain $P_\psi(S_1,\ldots,S_m)$ (resp. $P_\psi(\gamma_1,\ldots,\gamma_m)$), the disk cotangent bundles $D^*S_i$ (resp. $D^*\gamma_i$) are plumbed along a straight line, we can perform the deformation as in Figure~\ref{deform} on each plumbing domain $P_\psi(\gamma_i,\gamma_{i+1})$ of two adjacent unit disk cotangent bundle of circles to eliminate all bigons between $\gamma_i$ and $\tau^n(\gamma_j)$. Note that this process does not bring new bigons outside of $P_\psi(\gamma_i,\gamma_{i+1})$ in $P_\psi(\gamma_1,\ldots,\gamma_m)$ since the plumbing points throughout the deformation are fixed.  Once this was done, replacing Lemma~\ref{LmClean} with Lemma~\ref{mLmClean} in the proof of Lemma \ref{LmIntersection}, we finish the proof. 
\end{proof}
}

\section{Speculations on Measure theoretic entropy}\label{me}
Since Dehn twists are symplectomorphisms on the manifold $M$, they also preserve the Lebesgue measure $\mu$ on $M$.
We conjecture that the measure theoretic entropy of $\tau=\tau_1^k\tau_2^\ell, kl<0$ on $P_\psi(S_1,S_2)$ with respect to $\mu$ is also positive. However, there are certain substantial obstacles, similar to the ones we see for the standard map.

Let us first point out the difficulty in the standard map for the purpose of comparison with the current setting. In \cite{BXY} the authors adopted a special form of a class of 2 dimensional maps including the standard map: $F(x,y)=(y+L\psi(x),-x)$ defined on the torus, where $L$ is a parameter and $\psi$ is a generic smooth function on $\mathbb T^1$. The derivative matrix has the form $DF(x,y)=\left[\begin{array}{cc}
L\psi'(x)&1\\
-1&0
\end{array}\right]$.  For large $L$ in most part of the domain ($\{|\psi'(x)|>L^{-1/2}\}$), the matrix has a large eigenvalue with almost horizontal eigenvector. However, when $L\psi'(x)$ is close to zero, the matrix $DF$ is almost a rotation by $\pi/2$, which can mix expanding and contracting directions.

We shall see below that similar problem appears in our composite Dehn twists in a disguised form.
\subsection{Periodic points}
We consider the action of $\tau=\tau_1\tau_2^{-1}$ on $P_\psi(S_1,S_2)$ for $dim(S_1)=dim(S_2)=2$ and for any monotone $r(t)$. Before trying to calculate Lyapunov exponents on the whole manifold, we  first consider the Lyapunov exponents at the periodic points, which are given by the eigenvalues of the differential of iterations.
\begin{prop}
	 There exist finitely many $1$-dimensional periodic circles on $P_\psi(S_1,S_2)$ on which the Lyapunov exponent of  $\tau_1\tau_2^{-1}$  is zero.

\end{prop}
\begin{proof}
We notice that $(q,v)\in T^*S_1$ is a periodic point of $\tau_1$ if and only if $r(|v|)\in \mathbb{Q}\pi$. Thus if a point $p_0=(q_2,s_2)=\psi(q_1,s_1)\in P_\psi(S_1,S_2), (q_1,s_1)\in T^*S_1, (q_2,s_2)\in T^*S_2$ satisfies $r(|s_i|)\in \mathbb{Q}\pi,r(|s_i|)>2\epsilon,i=1,2$, then $p_0$ is a periodic point under $\tau$. For such periodic points, we  explicitly calculate the differentials of the Dehn twist in local coordinates.

We consider the plumbing space $P_\psi(S_1,S_2)$ plumbed at the point $p=(0,0,1)\in S_i\subset \R^3$, given by local coordinates $\phi_i:S_i\to \R^2$ given by $\phi_i(x_1,x_2,x_3)=(x_1,x_2)$, and $T^*\phi_i(x_1,x_2,x_3,y_1,y_2,y_3)=(x_1,x_2,y_1-\frac{x_1y_3}{x_3},y_2-\frac{x_2y_3}{x_3})$. We also suppose that $r^\prime\le \frac{C}{\epsilon}$ for some constant $C$.

So if $\tau_1^k(p)\in U_1$ for $p\in U_1$, take $\phi_1(p)=(u,v)\in \R^4, u,v,\in\R^2$, then we have, in local coordinates
$$\tilde{\tau_1}^k(u,v):=\phi_1\tau_1^k\phi_1^{-1}(u,v)=\bigg(f_1(u,v) u+\frac{\sin(kr(\rho))}{\rho} v, \bigg(f_2(u,v)-\frac{f_2}{f_1}\bigg)u+\cos(kr(\rho))v\bigg),$$
where $\rho(u,v)=\sqrt{|v|^2-\frac{\langle u,v\rangle^2}{1-|u|^2}+\frac{\langle u,v\rangle^2|u|^2}{(1-|u|^2)^2}}$ and $$f_1(u,v)=\cos kr(\rho)-\frac{\sin kr(\rho)}{\rho}\frac{\langle u,v\rangle}{1-|u|^2}; f_2(u,v)=-\sin kr(\rho)\rho-\cos kr(\rho)\frac{\langle u,v\rangle}{1-|u|^2}.$$

Suppose $\tau_1^k(\phi_1(u,v))\in U_1$, then treating $u,v\in \R^2$ as column vectors, when $(u,v)=\phi_1^{-1}(p_0)$ with $p_0$ a periodic point of $\tau_1$ with period $k$, we have

\beq\begin{split}d\tilde{\tau_1}^k (u,v)=\left(\begin{pmatrix}I & 0\\0&I\end{pmatrix}-\begin{pmatrix} 0 & \frac{kr^\prime }{\rho}\\ -kr^\prime \rho & 0\end{pmatrix}\begin{pmatrix} \frac{\partial \rho}{\partial u}u^t&\frac{\partial \rho}{\partial v}u^t\\\frac{\partial \rho}{\partial u}v^t&\frac{\partial \rho}{\partial v}v^t\end{pmatrix}\right.
\\\left.+\frac{\langle u,v\rangle}{1-|u|^2}\begin{pmatrix}\frac{kr^\prime}{\rho} \frac{\partial \rho}{\partial u}u^t&\frac{kr^\prime}{\rho} \frac{\partial \rho}{\partial v}u^t\\-1+\frac{\langle u,v\rangle}{1-|u|^2}\frac{kr^\prime}{\rho} \frac{\partial \rho}{\partial u}u^t&\frac{\langle u,v\rangle}{1-|u|^2}\frac{kr^\prime}{\rho} \frac{\partial \rho}{\partial v}u^t\end{pmatrix}\right)\begin{pmatrix}du \\ dv\end{pmatrix}.\end{split}\eeq

Furthermore, if we suppose $\phi(p_0)=(u,v)$ with $\langle u,v\rangle=0$, then we have the exact equality $$d\tilde{\tau_1}^k (u,v)=\begin{pmatrix}I & kr^\prime(|v|)\hat{v}\hat{v}^t\\0&I\end{pmatrix}\begin{pmatrix}du \\ dv\end{pmatrix}$$ where we write $\hat{v}=\frac{v}{|v|}$.

Next we consider the action of the second Dehn twist $\tau_2$, suppose $p_0$ is a periodic point of $\tau_2$ of period $\ell$. Since the coordinates $(u,v)$ on $T^*S_1$ is identified to coordinates $(u^\prime, v^\prime)$ on $T^*S_2$ through $J: (u,v)\to (u^\prime,v^\prime)=(-v,u)$, we have
 $$d(J^{-1}\phi_2\tau_2\phi_2^{-1}J)^\ell(u^\prime,v^\prime)=\begin{pmatrix}0 & I\\-I&0\end{pmatrix}\begin{pmatrix}I & \ell r^\prime(|u|)\hat{u}\hat{u}^t\\0&I\end{pmatrix}\begin{pmatrix}0 &-I\\I&0\end{pmatrix}\begin{pmatrix}du \\ dv\end{pmatrix}$$$$=\begin{pmatrix}I & 0\\-\ell r^\prime(|u|)\hat{u}\hat{u}^t&I\end{pmatrix}\begin{pmatrix}du \\ dv\end{pmatrix}.$$

Now we  calculate the differential of an iteration of $\tau$ at $p_0\in T^*U_1\cap T^*U_2\subset P_\psi(S_1,S_2)$. Suppose $p_0=(q_2,s_2)=\psi(q_1,s_1)\in P_\psi(S_1,S_2), (q_1,s_1)\in T^*S_1, (q_2,s_2)\in T^*S_2$ with $r(|s_1|)=\frac{1}{k}\ge 2\epsilon ,r(|s_2|)=\frac{1}{l}\ge 2\epsilon$, then $p_0$ is also a periodic point of $\tau_1\tau_2^{-1}$, of period $k+|\ell|-1$. By the calculations above, we see that for $p_0=\phi^{-1}(u,v)$ with $\langle u,v\rangle = 0$, we have $$d(\tau_1\tau_2^{-1})^{k+\ell-1}(u,v)=d(\tau_1^k)d(\tau_2^{-l})(u,v)=\begin{pmatrix}I & kr^\prime(|v|)\hat{v}\hat{v}^t\\\ell r^\prime(|u|)\hat{u}\hat{u}^t&I\end{pmatrix}\begin{pmatrix}du \\ dv\end{pmatrix}.$$

Thus $d(\tau_1\tau_2^{-1})^{k+\ell-1}(u,v)$ has eigenvalues 1, so we have $\chi^+(p_0)=0$ for finitely many $p_0$.\end{proof}

\subsection{Hyperbolic cones near the elliptic points}
Now if we remove the condition $\langle u,v\rangle=0$ for $p_0$ and keep the conditions that $p_0$ is periodic under both $\tau_1$ and $\tau_2$ we would get a partially hyperbolic periodic point of $\tau$. However, we shall see that in a neighborhood of the elliptic periodic point, the eigenvectors of the hyperbolic periodic points ``flip", so the arguments we used in Section \ref{topent} to prove positive measure theoretic entropy on the subsystem cannot be applied to prove positive entropy on the whole system with respect to the Lebesgue measure.

Now for $p_0=(q_2,s_2)=\psi(q_1,s_1)\in P_\psi(S_1,S_2), (q_1,s_1)\in T^*S_1, (q_2,s_2)\in T^*S_2$ with $r(|s_i|\in \mathbb{Q},r(|s_i|)>2\epsilon,i=1,2$, we see from the second condition that $|u|,|v|<\epsilon$, which gives us
$$d\tilde{\tau_1}^k (u,v)=\left(\begin{pmatrix}I & kr^\prime(|v|)\hat{v}\hat{v}^t\\0&I\end{pmatrix}+O(\langle u,v\rangle)\right)\begin{pmatrix}du \\ dv\end{pmatrix} .$$

Applying the same argument for $\tau_2$, we have
$$d\tilde{\tau_2}^\ell (u,v)=\left(\begin{pmatrix}I & 0\\-\ell r^\prime(|u|)\hat{u}\hat{u}^t&I\end{pmatrix}+O(\langle u,v\rangle)\right)\begin{pmatrix}du \\ dv\end{pmatrix} .$$

Thus $p_0$ as also a periodic point of $\tau_1\tau_2^{-1}$ of period $k+|\ell|-1$ has differential $$d(\tau_1\tau_2^{-1})^{k+\ell-1}(u,v)=\left(\begin{pmatrix}\ell+k\ell r^\prime(|u|)r^\prime(|v|)\langle \hat{u},\hat{v}\rangle \hat{v}\hat{u}^t  & kr^\prime(|v|)\hat{v}\hat{v}^t\\\ell r^\prime(|u|)\hat{u}\hat{u}^t&I\end{pmatrix}+O(\langle u,v\rangle/\epsilon)\right)\begin{pmatrix}du \\ dv\end{pmatrix}.$$

Thus for $p_0=\phi_1(u,v)$ with $\langle u,v\rangle\neq 0$, we have $tr(d(\tau_1\tau_2^{-1})^{k+\ell-1})>4$ which shows that $p_0$ has positive Lyapunov exponent. The eigenvalues of $d(\tau_1\tau_2^{-1})^{k+\ell-1}$ are approximately $$\lambda_1=1+\ell r^\prime(|u|)\langle \hat{u},\hat{v}\rangle \frac{k\ell r^\prime(|u|)r^\prime(|v|)\langle \hat{u},\hat{v}\rangle+\sqrt{(k\ell r^\prime(|u|)r^\prime(|v|)\langle \hat{u},\hat{v}\rangle)^2+4(kr^\prime(|v|))^2}}{2},$$$$\lambda_2=\lambda_3=1,$$$$ \lambda_4=1+\ell r^\prime(|u|)\langle \hat{u},\hat{v}\rangle \frac{k\ell r^\prime(|u|)r^\prime(|v|)\langle \hat{u},\hat{v}\rangle-\sqrt{(k\ell r^\prime(|u|)r^\prime(|v|)\langle \hat{u},\hat{v}\rangle)^2+4(kr^\prime(|v|))^2}}{2}$$ with eigenvectors $$w_1=\begin{pmatrix}\frac{k\ell r^\prime(|u|)r^\prime(|v|)\langle \hat{u},\hat{v}\rangle+\sqrt{(k\ell r^\prime(|u|)r^\prime(|v|)\langle \hat{u},\hat{v}\rangle)^2+4(kr^\prime(|v|))^2}}{2} \hat{v}\\\hat{u}\end{pmatrix}$$ and $$w_4=\begin{pmatrix}\frac{k\ell r^\prime(|u|)r^\prime(|v|)\langle \hat{u},\hat{v}\rangle-\sqrt{(k\ell r^\prime(|u|)r^\prime(|v|)\langle \hat{u},\hat{v}\rangle)^2+4(kr^\prime(|v|))^2}}{2} \hat{v}\\ \hat{u}\end{pmatrix}.$$

So we see that if $\langle u,v\rangle >0$ we have $\lambda_1>1$ and $\lambda_4<1$, and when $\langle u,v\rangle <0$ we have $\lambda_1<1$ and $\lambda_4>1$. However the corresponding eigenvectors $w_1$ and $w_4$ is continuous with respect to $u,v$. Thus if a tangent vector returns to the neighborhood of an elliptic periodic point, it can start shrinking even if it was originally expanding.

\begin{thebibliography}{SK}
	
	\bibitem{Ab} M. Abouzaid, {\it  A topological model for the Fukaya categories of plumbings}.  J. Differential Geom.   {\bf 87} (2011), 1--80.

 \bibitem{AS} M. Abouzaid and I. Smith, {\it Exact Lagrangians in plumbings}. Geom. Funct. Anal.  {\bf 22} (2012), 785--831.
	
	\bibitem{Arnold} V. I. Arnold, {\it  Some remarks on symplectic monodromy of Milnor fibrations, } The Floer Memorial Volume (H. Hofer, C. Taubes, A. Weinstein, and E. Zehnder, eds.), Progress in Mathematics, vol. 133, Birkh\"auser, 1995, pp. 99-104.

	\bibitem{Ar2} V. I. Arnold, {\it  Dynamics of complexity of intersections}.  Bol. Soc. Brasil. Mat. (N.S.)  {\bf 21} (1990), 1--10.


	\bibitem{BK} F. Barbacovi and J. Kim, {\it Entropy of the composition of two spherical twists}, Osaka J. Math. {\bf 60} (2023), no.~3, 653--670; MR4612509
    
		\bibitem{bar} L. Barreira, Y. Pesin. {\it  Nonuniform Hyperbolicity: Dynamics of Systems with Nonzero Lyapunov Exponents}. Encyclopedia of Mathematics and Its Applications, 115 Cambridge University Press.


    \bibitem{brin}Brin M, Stuck G. {\it Introduction to Dynamical Systems}, Cambridge University Press; 2002.
	\bibitem{bur}R. Burton and R. W. Easton. {\it  Ergodicity of linked twist maps}. In
	Global theory of dynamical systems (Proc. Internat. Conf., Northwestern Univ., Evanston, Ill., 1979), volume 819 of Lecture Notes in Math., pages 35–49. Springer, Berlin, 1980.
	\bibitem{BXY} A. Blumenthal, J. Xue, L.-S. Young, {\it Lyapunov exponents for random perturbations of some area-preserving maps including the standard map,}  Ann. Math. 	185 (2017), 285-310

\bibitem{CGG} E. Cineli, V.L. Ginzburg, and B.Z. Gurel. {\it Topological entropy of hamiltonian diffeomorphisms: a persistence homology and floer theory perspective.} Math. Z. {\bf 308} (2024), no. 4, Paper No. 73, 38 pp.

	\bibitem{DHKK} G. Dimitrov, F. Haiden, L. Katzarkov, and M. Kontsevich, {\it  Dynamical systems and categories}, Contemp.
	Math. 621 (2014), 133-170.

 \bibitem{Eva} J.D. Evans.  {\it Symplectic mapping class groups of some Stein and rational surfaces.}  J. Symplectic Geom.,
{\bf9} (2011), 45--82.
	
	\bibitem{FFHKL} Y-W. Fan, S. Filip, F. Haiden, L. Katzarkov, and Y. Liu, {\it  On pseudo-Anosov autoequivalences}.
	Adv. Math., 384 (2021),
	
	\bibitem{FM} B. Farb and D. Margalit. {\it A primer on mapping class groups}, volume 49 of Princeton Mathematical Series. Princeton University Press, Princeton, NJ, 2012.
	
	\bibitem{fathi} A. Fathi. {\it  Dehn twists and pseudo-Anosov diffeomorphisms}. Invent. Math., {\bf 87} (1987), 129–151.
	
	\bibitem{FLP} A. Fathi, F. Laudenbach and V. Po\'{e}naru, editors. {\it Travaux de Thurston sur les surfaces}, volume 66 of Ast\'{e}risque. Soci\'{e}t\'{e} Math\'{e}matique de France, Paris, 1979.
	
	\bibitem{Fl} A. Floer, {\it  Morse theory for Lagrangian intersections,}  J. Differential Geom. {\bf 28} (1988),
	513--547.
	
	\bibitem{Fl2} A. Floer, {\it  Witten's complex and infinite dimensional Morse theory},  J. Differential Geom. {\bf 30} (1989), 207--221.
		
        \bibitem{FS}  U. Frauenfelder and F. Schlenk, {\it  Volume growth in the component of the Dehn-Seidel twist}, Geom. Funct. Anal. {\bf 15} (2005), 809--838.


	\bibitem{Ja} A. Jannaud,  {\it  Dehn-Seidel twist, $C^0$-symplectic topology and barcodes.} arXiv:2101.07878.
	
	\bibitem{Katok}A. Katok and B. Hasselblatt. {\it  Introduction to the modern theory of dynamical systems}, volume 54 of Encyclopedia of Mathematics and its Applications. Cambridge University Press, Cambridge, 1995. With a supplementary chapter by Katok and Leonardo Mendoza.
	
	\bibitem{Ke} A. M. Keating, {\it  Dehn twists and free subgroups of symplectic mapping class groups},   J. Topol. {\bf 7} (2014), 436--474.

 \bibitem{KR} A. Keating, Ailsa and O. Randal-Williams, {\it On the order of Dehn twists},  New York J. Math. {\bf 29} (2023), 203--212.
	
	
	\bibitem{KS} M. Khovanov and P. Seidel, {\it  Quivers, Floer cohomology, and braid group actions}.   J. Amer. Math. Soc. {\bf 15} (2002), 203--271.

 	\bibitem{KO} K. Kikuta and G. Ouchi.  {\it Hochschild entropy and categorical entropy}. Arnold Math. J. {\bf 9} (2023),  223–244.



	
	\bibitem{Oh} Y.-G. Oh,  {\it  Symplectic topology as the geometry of the action functional I,}  J. Differential Geom. {\bf 46} (1977) 1--55.
    
		\bibitem{Ose}V. I. Oseledets, {\it A multiplicative ergodic theorem. Characteristic Liapunov exponents of dynamical systems,} Tr. Mosk. Mat. Obs., 19, MSU, M., 1968, 179–210; Trans. Moscow Math. Soc., 19 (1968), 197–231
        
	\bibitem{Pe} R. C. Penner, {\it  A construction of pseudo-Anosov homeomorphisms}.  Trans. Amer. Math. Soc. {\bf  310} (1988) 179--197.
	
	\bibitem{Po} M. Po\'{z}niak, {\it  Floer homology, Novikov rings, and clean intersections.}  Northern California
		Symplectic Geometry Seminar, Amer. Math. Soc. Transl., Ser. 2, 196, AMS, Providence, 1999, 119-181.
	\bibitem{PS} L. Polterovich and E. Shelukhin, {\it  Autonomous Hamiltonian flows, Hofer’s geometry and persistence modules}, Selecta Mathematica volume 22, pages227-296 (2016)
	
	\bibitem{SZ} D. Salamon and E. Zehnder, {\it  Morse theory for periodic solutions of Hamiltonian systems and
	the Maslov index},  Comm. Pure Appl. Math. {\bf 45} (1992), 130--1360.
	
	
	\bibitem{Se} P. Seidel, {\it  Lagrangian two-spheres can be symplectically knotted,}  J. Differential Geom. {\bf 52}
	(1999), 147--173.
	
	\bibitem{Se1} P. Seidel, {\it  A long exact sequence for symplectic Floer cohomology}.  Topology {\bf 42} (2003), 1003--1063.
	
	\bibitem{Se2} P. Seidel, {\it  Lectures on Categorical Dynamics and Symplectic Topology}.

 \bibitem{Sm} I. Smith, {\it A symplectic prolegomenon}. Bull. Amer. Math. Soc. (N.S.) {\bf52} (2015), 415--464.
		
	\bibitem{Th} W. P. Thurston, {\it  On the geometry and dynamics of diffeomorphisms of surfaces}, Bull. Amer. Math. Soc.
	(N.S.) 19 (1988), no. 2, 417-431.

	\bibitem{wu} W. Wu. {\it  Exact Lagrangians in $A_n$-surface singularities.} Math. Ann., 359(1-2):153–168, 2014.
	
	\bibitem{Yomdin} Y. Yomdin, {\it Volume growth and entropy}, Israel J. Math. 57(3) (1987) 285–300.
    

\end{thebibliography}

\appendix

\section{Topological entropy and non-uniformly hyperbolic dynamics}\label{apd:etp}
Here we introduce the basic definitions of topological and measure theoretic entropy and its relation to hyperbolicity of the system.
The contents of this subsection are classical and can be found in e.g. Chapters 3, 4 and Supplement of \cite{Katok}.

\begin{defi}[Topological entropy] The topological entropy of a map $f:X\to X$ in a compact metric space $(X,d)$ is given by $$h_d(f)=\lim_{\delta\to 0} \limsup_{n\to \infty} \frac{1}{n}\log S_d(f,\delta,n)$$ where $S_d(f,\delta,n)$ is the minimal cardinality of sets of balls $B(x,\delta, d^f_n)$ that covers the space $X$.  Here, the balls $B(x,\delta, d^f_n)=\{y\in X: d^f_n(x,y)<\delta\}$ are taken with respect to the metric $d^f_n(x,y)=\max_{0\le i\le n-1} d(f^i(x),f^i(y))$.
\end{defi}

The topological entropy has the following properties.
\begin{prop} \label{proptop}
	\begin{enumerate}
		\item The topological entropy does not depend on the metric $d$. If $d^\prime$ is a metric on $X$ such that $(X,d)$ and $(X,d^\prime)$ has the same topology then $h_d(f)=h_{d^\prime}(f)$. Thus we  define $h_{top}(f)=h_d(f)$ for any metric $d$ compatible with the topology.
		\item The topological entropy is invariant under conjugacy, i.e. $h_{top}(f)=h_{top}(g^{-1}fg)$ for any homeomorphism $g:X\to Y$.
		\item If $\Lambda$ is a closed $f$-invariant subset of $X$, then $h_{top}(f\mid_\Lambda)\le h_{top}(f)$.
		\item $h_{top}(f^m)=|m|h_{top}(f)$.
	\end{enumerate}
\end{prop}

Furthermore, the topological entropy is bounded below by the action on the cohomology ring.

\begin{thm}[Yomdin's inequality \cite{Yomdin}]\label{yom}
Let $X$ be a compact smooth manifold and $f :
X \to X$ be a $C^\infty$ diffeomorphism. Then we have
$$h_{top}(f) \ge \log \rho(f^*:H^*(X)\to H^*(X))$$
where $\rho(\cdot)$ denotes the largest absolute value of eigenvalues of a linear map.
\end{thm}

Another description of entropy is the measure theoretic entropy, given with respect to a measure. This is a more global description of the complexity of a dynamical system.
\begin{defi}[Measure theoretic entropy]A \emph{measurable partition} of a probability space $(X,\mathcal{B},\mu)$ is a collection of measurable subsets $\xi=\{C_\alpha\in\mathcal{B}\ |\ \alpha\in I\}$ such that $\mu(\cup C_\alpha)=1$ and $\mu(C_\alpha\cap C_\beta)=0\text{ or }\mu(C_\alpha\Delta C_\beta)=0,\forall \alpha\neq \beta$.
	
	The entropy of a measurable partition $\xi=\{C_\alpha\in\mathcal{B}\ |\ \alpha\in I\}$ is given by $$H(\xi)=-\sum_{\alpha\in I}\mu(C_\alpha)\log(\mu(C_\alpha)).$$
	
	The entropy of a measure-preserving transformation $T$ on $(X,\mu)$ is defined as $$h_\mu(T)=\sup_{\xi}\lim_{n\to\infty} \frac{1}{n}H(\{\bigvee_{i=1}^n T^{-i+1}C_i \ |\ C_i\in\xi\})$$ where $\xi$ is taken over any measurable partition with $H(\xi)<\infty$.
	
\end{defi}

The measure theoretic entropy and the topological entropy of are related by the variational principle.
\begin{thm}[Variational Principle, see e.g. \cite{brin} Theorem 9.5.4]Let $X$ be a compact metric space and let $f:X\to X$ be a homeomorphism. Set $M(f)$ to be the set of Borel probability measures that are $f$-invariant. Then we have $$h_{top}(f)=\sup_{\mu\in M(f)} h_\mu(f).$$
\end{thm}

One way of calculating measure theoretic entropy is by examining the ``hyperbolicity" of the map $f$. As we shall see in Pesin entropy formula, a diffeomorphism $f$ has positive entropy if it possesses some hyperbolicity. A diffeomorphism $f:X\to X$ on a Riemannian manifold is said to be {\it hyperbolic} if a every point $p$, $Df$ expands a subspace $E^+(p)\subset T_pX$ at rate $\lambda>1$ and contracts a subspace $E^-(p)\subset T_pX$ at rate $1/\lambda<1$ such that $T_pX=E^+(p)\oplus E^-(p)$ and $E^+, E^-$ are invariant under $f$.

A more precise description of hyperbolicity at a generic point is given by the Lyapunov exponent which describes the expansion rate of $f$ on a tangent vector $(p,u)$.
\begin{defi}\label{DefLyap}Let $X$ be a Riemannian manifold with $f:X\to X$ a diffeomorphism of $X$. The \emph{Lyapunov exponent} of $f$ at a point $(p,u)\in TX$ is defined as $$\chi^+(p,u)=\lim_{m\to \infty}\frac{1}{m} \log\|df^{m}_p(u)\|$$ if the limit exists.
\end{defi}

Although {\it a priori} the Lyapunov exponent of a tangent vector may not exist, Oseledets \cite{Ose} proved that for almost every point on the manifold $X$ the Lyapunov exponent at any vector is well defined, and we can split the tangent space into subspaces on which $f$ has uniform hyperbolicity.

\begin{thm}[Oseledets Multiplicative Ergodic Theorem, \cite{Ose}]Let $X$ be a Riemannian manifold with $f:X\to X$ a diffeomorphism, let $\mu$ be a $f$-invariant Borel measure on $X$ such that $\log^+\|Df(\cdot)\|\in L^1(X,\mu)$, then there exists a set $Y\subset X$ with $\mu(X\backslash Y)=0$ such that for each $p\in Y$, there exists a decomposition $T_pX=\oplus_{i=1}^{k(p)}H_i(p)$ that is invariant under $Df$ and a set of $f$-invariant functions $\chi_1(p)<\chi_2(p)<...<\chi_{k(p)}(p)$ $($called the \emph{Lyapunov exponents} of $p)$ such that the Lyapunov exponent exists for any $u\in T_pX$ and $\chi^+(p,u)=\pm\chi_i(p),\ \forall\  u\in H_i(p).$ The functions $\chi_i(\cdot),k_i(\cdot),H_i(\cdot)$ depend measurably on the base points.
\end{thm}

Furthermore, we can use the splitting of the tangent spaces to generate stable and unstable local submanifolds of a non-uniformly hyperbolic map.  {The theorem holds for manifolds of arbitrary dimension under some conditions of the Lyapunov exponents. Here}, to avoid technical details, we only state the result for Riemannian surfaces which we shall use in Section~\ref{mtwist}. More details on this subject can be found in  \cite[Chap. 7 and 8]{bar}.

\begin{thm}[{Pesin's stable manifold theorem, see  \cite[Theorem~7.1]{bar}}] \label{thm:Pesin} Let $X$ be a Riemannian surface and $f:X\to X$ a 
$C^{1+\alpha}$ diffeomorphism preserving an invariant measure $\mu$ such that the Lyapunov exponent $\chi_1(p)<0$ for $\mu$-almost every $p\in X$, then for $\epsilon$ sufficiently small, there exist  measurable functions $C_\epsilon: X\to \R$ and $\lambda:X\to (0,1)$, such that for almost every $p\in X$, there exists a stable manifold~$($curve$)$ $W^s_f(p)$ such that $x\in W^s_f(p), T_pW^s(p)=H_1(p)$, and for any $q\in W^s_f(p)$ we have $$\mathrm{dist}(f^n(p),f^n(q))\le C_\epsilon(x)(\lambda+\epsilon)^n(x)\mathrm{ dist}(x,y)$$
where $C_\epsilon(f^m(x))\leq C_\epsilon(x) e^{2\epsilon|m|}, \ m\in \Z. $
Furthermore, the stable manifolds satisfy $W^s(q)=W^s(p)$ for any $q\in W^s(p)$ and $W^s(q)\cap W^s(p)=\varnothing$ if $q\notin W^s(p)$.
\end{thm}

Similarly, if we consider local stable manifolds $W^s_{f^{-1}}(p)$ of $f^{-1}$, then the action of $f$ expands the distance of points on the manifold exponentially. We call them the local unstable manifolds of $f$ and denote by $W^u_f(p)=W^s_{f^{-1}}(p)$.

The Lyapunov exponents  {are} connected with the entropy of a map by the Pesin entropy formula.
\begin{thm}[{Pesin entropy formula, see \cite[Theorem 9.25]{bar}}]\label{pef}Let $X$ be a compact Riemannian manifold with $f:X\to X$ a $C^{1+\alpha}$ diffeomorphism, suppose that $f$ preserves a smooth Borel measure $\mu$ on $X$ that is equivalent to the volume, then we have $$h_\mu(f)=\int \chi^+  d\mu$$ where $\chi^+(p)=\sum_{\chi_i(p)>0}\chi_i(p)$ with $\chi_i$ the Lyapunov exponents of $p$ as given in the Oseledets Multiplicative Ergodic Theorem.
\end{thm}

\section{Floer cohomology}\label{apd:fl}

In the section we give a brief review of Lagrangian Floer cohomology. The definition of Floer cohomology here is essentially Floer's original one~\cite{Fl}. For our purpose in the present paper, we work with ungraded groups.

Let $(\overline{M},d\bar{\lambda})$ be a $2n$-dimensional exact symplectic manifold with contact-type boundary, and let $(\overline{M},d\bar{\lambda})\cup (\partial \overline{M}\times [0,\infty),d\lambda=d(e^r\bar{\lambda}))$ be
its symplectization. In the following we denote $$M_R:=\overline{M}\cup \partial \overline{M}\times [0,R].$$ Let $(S_1,S_2)$ be a pair of connected compact exact Lagrangian submanifolds of $\overline{M}$ with $\lambda |_{S_i}=df_i, i=1,2$, where $f_i\in C^\infty(S_i,\mathbb{R})$. We also ask that $S_i\cap\partial \overline{M}=\emptyset$, $i=1,2$. 
The action functional on the path space
$$\mathcal{P}(S_1,S_2)=\big\{\gamma:[0,1]\to M|\gamma(0)\in S_1,\gamma(1)\in S_2\big\}$$
is defined as
$$\mathcal{A}_{S_1,S_2}(\gamma)=\int\gamma^*\lambda+f_1(\gamma(0))-f_2(\gamma(1)).$$

Clearly, the critical points of $\mathcal{A}_{S_1,S_2}$ are constant paths $\gamma_x$ at the intersection points $x$ of $S_1$ and $S_2$.
If $S_1$ and $S_2$ intersect transversely, the Floer cohomology $\HF(S_1,S_2)$ is well defined by the standard argument of transversality and gluing, see~\cite{Oh}. 
In general, to define Lagrangian Floer cohomology of $S_1$ and $S_2$ one needs to consider Hamiltonian perturbations to achieve transversality. Let us give a quick review of this construction. Let $H\in \mathcal{H}=C_c^\infty([0,1]\times M,\mathbb{R})$. For every $\gamma\in \mathcal{P}(S_1,S_2)$, the Hamiltonian action of $\gamma$ is
$$\mathcal{A}_{S_1,S_2}^H(\gamma)=\int\gamma^*\lambda-\int^1_0H(t,\gamma(t))dt+f_1(\gamma(0))-f_2(\gamma(1)).$$
Denote by $\mathcal{P}_H(S_1,S_2)\subset \mathcal{P}(S_1,S_2)$ the critical points of this functional, which are the paths of the Hamiltonian flow $\phi_H^t$ with endpoints in $S_i$, i.e.,  $\gamma:[0,1]\to M$ such that $\dot{\gamma}(t)=X_H(t,\gamma(t))$, $\gamma(0)\in S_1$ and $\gamma(1)\in S_2$. So there is a one-to-one correspondence between $\mathcal{P}_H(S_1,S_2)$ and $\phi^1_H(S_1)\cap S_2$. For generic $H\in\mathcal{H}^{reg}\subseteq\mathcal{H}$, $\phi^1_H(S_1)$ and $S_2$ intersect transversely. The Floer cochain $CF(S_1,S_2;H)$ is the vector space over $\mathbb{Z}/2$ with a base given by these intersection points. Denote by $\mathcal{J}$ the set of one-parameter families $(J_t)_{t\in\R}$ of $d\lambda$-compatible almost complex structures on $M$ which are of contact type on $\partial \overline{M}\times (-\epsilon,\infty)$ for some $\epsilon>0$, i.e., $dr\circ J_t=-\lambda$ with coordinate $r\in (-\epsilon,\infty)$. For $J\in \mathcal{J}$, consider the maps $u:\mathbb{R}\times [0,1]\to M$ which solves the Floer equation
\begin{equation}\label{e:CR}
	\partial_su+J(t,u)\big(\partial_t u-X_H(t,u)\big)=0
\end{equation}
subject to the boundary conditions
\[u(s,0)\in S_1\;\hbox{and}\;u(s,1)\in S_2,\quad s\in\mathbb{R},
\]
\[
\lim\limits_{s\to+\infty}u(s,t)=x,\qquad \lim\limits_{s\to+\infty}u(s,t)=y,
\]
and the finite energy condition
\[
E(u)=\frac{1}{2}\int^1_0\int^1_0\big|\partial_su\big|^2+\big|\partial_tu-X_H(u)\big|^2dsdt<\infty.
\]

Denote by $\widehat{\mathcal{M}}(x,y;H,J)$ the space of the above maps $u$. There is a natural $\mathbb{R}$-translation on $\widehat{\mathcal M}(x,y;H,J)$ in $s$-direction, and its quotient space is denoted by $\mathcal{M}(x,y;$ $H,J)$. Solutions of (\ref{e:CR}) can be thought as negative gradient flow lines for $\mathcal{A}_{S_1,S_2}^H$ in an $L^2$-metric on $\mathcal{P}(S_1,S_2)$. For each $u\in \widehat{\mathcal{M}}(x,y;H,J)$ we linearize (\ref{e:CR}) and obtain a Fredholm operator $D_{H,J,u}$ in suitable Sobolev spaces.
Under the assumptions that $\langle [\omega],\pi_2(M,S_i)\rangle=0,i=1,2$, there is a dense subspace $\mathcal{J}^{reg}\subset \mathcal{J}$ of almost complex structures such that $D_{H,J,u}$ are onto for all $u\in \widehat{\mathcal{M}}(x,y;H,J)$, see~\cite{Fl}. Hence the spaces
$\widehat{\mathcal{M}}(x,y;H,J)$, as well as $\mathcal{M}(x,y;H,J)$, are smooth manifolds.

In the setting of our present paper, these topological assumptions are met. Moreover, we have the following 
 \begin{lem}[{\cite[Lem. 5.5]{KS}}]\label{lem:compact}
Let $\Sigma$ be an open subset of $\R\times [0,1]$, and $(J_{s,t})$ a smooth family of $d\lambda$-compatible almost
complex structures parametrized by $(s,t)\in\Sigma$ which is of contact type on $\partial \overline{M}\times (-\epsilon,\infty)$ for some $\epsilon>0$.  Let $u=(u_1,u_2):\Sigma\to\partial \overline{M} \times (-\epsilon,\infty) $ be a smooth map which satisfies
\[
\partial u/\partial s+J_{s,t}(u)\partial u/\partial t=0.
\]
Then $u_1$ has no local maxima. 
\end{lem}
Suppose now that $\cup_t{\rm Supp}(H_t)\subset M_R$. 
From Lemma~\ref{lem:compact}, we see that each solution $u$ to (\ref{e:CR}) has image in a compact set of $M_R\setminus\partial M_R$. Then for generic $J\in\mathcal{J}^{reg}$, we can define the Floer differential $d_{H,J}:\ \CF(S_1,S_2;H)\to \CF(S_1,S_2;H)$ by counting isolated points in $\mathcal{M}(x,y;H,J)$ mod $2$, i.e.,
$$d_{H,J}(y)=\sum_{x}\sharp_2\mathcal{M}(x,y;H,J)\cdot x.$$
This map has square zero, i.e., $d_{H,J}^2=0$, and hence $\CF(S_1,S_2;H)$ is a complex over  $\mathbb{Z}/2$. $\HF(S_1,S_2;H,J)$ is defined to be its cohomology $\mathrm{Ker}(d_{H,J})/\mathrm{Im}(d_{H,J})$.  For two generic pairs $(H^0,J^0)$ and $(H^1,J^1)$ with $\cup_{t\in[0,1]}({\rm supp}(H^0_t)\cup {\rm supp}(H^1_t))\subset M_R$,  we choose a family of Hamiltonians $(H_{s,t})_{(s,t)\in\R\times[0,1]}$ connecting $H^0$ and $H^1$ with $\cup_{s,t}{\rm supp}(H_{s,t})\subset M_R$ and a family of $d\lambda$-compatible almost complex structures $(J_{s,t})_{(s,t)\in\R\times[0,1]}$ connecting $J^0$ and $J^1$ of contact type on $\partial\overline{M}\times (0,\infty)$ with the property that there is a large number $A>0$ such that $(H_{s,t},J_{s,t})=(H^0,J^0)$ for $s<-A$ and  $(H_{s,t},J_{s,t})=(H^1,J^1)$ for $s>A$. Consider the $s$-dependent Floer equation
\begin{equation}\label{e:sCR}
	\partial_su+J(s,t,u)\big(\partial_t u-X_H(s,t,u)\big)=0
\end{equation}
with boundary conditions $u(\R,0)\subset S_1,\;u(\R,1)\subset S_2$ and finite energy. Then by Lemma~\ref{lem:compact} again, each solution $u$ to $(\ref{e:sCR})$ has image in a compact set of  $ M_R\setminus \partial M_R$. Consequently, using the usual compactness and gluing arguments about the moduli spaces used to define chain maps between two Floer chains, one can show that Floer cohomology $\HF(S_1,S_2;H^0,J^0)$ is isomorphic to $\HF(S_1,S_2;H^1,J^1)$, see~\cite{SZ,Se1}. So we  define Lagrangian Floer cohomology of $S_1$ and $S_2$ as 
\[
\HF^M(S_1,S_2)=\HF(S_1,S_2;H,J)
\]
for any choice of $H\in C_c^\infty([0,1]\times M,\mathbb{R})$ such that $\phi^1_H(S_1)$ intersects $S_2$ transversely and any $J\in \mathcal{J}^{reg}$.

\end{document}